\documentclass[11pt]{amsart}
\title[Constructions of variations of mixed Hodge structures]
{Techniques of constructions of variations of mixed Hodge structures}

\author{Hisashi Kasuya}

\usepackage{amssymb}
\usepackage{amsmath}
\usepackage{amscd}
\usepackage{amstext}
\usepackage{amsfonts}
\usepackage[all]{xy}

\theoremstyle{plain}

\theoremstyle{plain}

\theoremstyle{plain}

\theoremstyle{plain}
\newtheorem{theorem}{Theorem}[subsection] 
\theoremstyle{remark}
\newtheorem{remark}[theorem]{Remark}
\theoremstyle{remark}
\newtheorem{Important note}[theorem]{Important note}
\theoremstyle{Main result}
\newtheorem{main result}{Main result}
\theoremstyle{lemma}
\newtheorem{lemma}[theorem]{Lemma}
\theoremstyle{definition}
\newtheorem{definition}[theorem]{Definition}
\theoremstyle{proposition}
\newtheorem{proposition}[theorem]{Proposition}
\theoremstyle{corollary}
\newtheorem{corollary}[theorem]{Corollary}
\theoremstyle{remark}
\newtheorem{example}[theorem]{Example}
\theoremstyle{plain}

\theoremstyle{problem}

\newtheorem*{prothm}{Theorem(Prototype)}
\newtheorem*{cond(V)}{Condition (V)}
\newtheorem*{condV2}{Condition (V')}

\address[Hisashi Kasuya]{\begin{itemize}
\item
Department of Mathematics, Graduate School of Science, Osaka University, Osaka,  Japan.
\item
Institut de Matht\'ematiques de Jussieu - Paris Rive Gauche, Paris, France
\end{itemize}}
\email{kasuya@math.sci.osaka-u.ac.jp}


\newcommand{\C}{\mathbb{C}}
\newcommand{\R}{\mathbb{R}}

\newcommand{\Z}{\mathbb{Z}}

\newcommand{\V}{\mathcal V}
\newcommand{\W}{\mathcal W}
\newcommand{\X}{\mathcal X}

\begin{document} 

\maketitle
\begin{abstract}
We give a way of constructing real variations of mixed Hodge structures over compact K\"ahler manifolds by using  mixed Hodge structures on Sullivan's $1$-minimal models  of  certain  differential graded algebras  associated with real variations of Hodge structures.

\end{abstract}

\section{Introduction}
A {\em real variation of mixed Hodge structure} ($\R$-VMHS)  over a complex manifold $M$ is  $({\bf E}, {\bf W}_{\ast}, {\bf F}^{\ast})$ so that:
\begin{enumerate}
\item ${\bf E}$ is a local system of  finite-dimensional $\R$-vector spaces.
\item ${\bf W}_{\ast}$ is an increasing filtration of the local system ${\bf E}$.
\item  ${\bf F}^{\ast}$ is a decreasing filtration of the holomorphic vector bundle ${\bf E}\otimes_{\R}{\mathcal O}_{M}$.
\item The Griffiths transversality $D{\bf F}^{r}\subset A^{1}(M, {\bf F}^{r-1})$ holds where $D$ is the flat connection associated with the local system ${\bf E}_{ \C}$.
\item For any $k\in \Z$, the local system $Gr^{\bf W}_{k}({\bf E})$ with the filtration induced by $ {\bf F}^{\ast}$ is a real variation of Hodge structure of weight $k$.
\end{enumerate}
The purpose of this paper is to give a  way of constructing $\R$-VMHSs over  compact K\"ahler manifolds starting from  real variations of Hodge structures.

\smallskip
\subsection{Prototype}
We first introduce our main results for the simplest case.
We briefly review Sullivan's $1$-minimal model (see \cite{DGMS}, \cite{GMo},  \cite{Sul} for more details).
We fix a ground field $\mathbb{K}$ of characteristic zero. 
A differential graded algebra (DGA) ${\mathcal M}^{\ast}$ is {\em$1$-minimal} if ${\mathcal M}^{\ast}$  is the increasing union of sub-DGAs
\[\mathbb{K}={\mathcal M}^{\ast}(0)\subset {\mathcal M}^{\ast}(1)\subset \dots 
\]
such that ${\mathcal M}^{\ast}(k)$ is the exterior algebra $\bigwedge (\V_{1}\oplus\dots \oplus \V_{k})$ of the direct sum $\V_{1}\oplus\dots \oplus\V_{k}$ of vector spaces $\V_{1},\dots, \V_{k}$  with   $d\V_{k}\subset {\mathcal M}^{2}(k-1)=\bigwedge^{2} (\V_{1}\oplus\dots \oplus\V_{k-1})$ for each $k\ge1$ where the degree of any element in  each $V_{i}$ is of degree $1$.
On the dual space of $\mathcal M^{1}$, the dual map of the differential $d:\mathcal M^{1}\to \mathcal M^{1}\wedge \mathcal M^{1}$ is a Lie bracket.
Such Lie algebra is called the {\em dual Lie algebra} of $\mathcal M^{\ast}$.
A {\em $1$-minimal model} of a DGA $A^{\ast}$ is a  $1$-minimal DGA $\mathcal M^{\ast}$ with a morphism $\phi:{\mathcal M}^{\ast}\to A^{\ast}$ which induces isomorphisms on $0$th and first cohomologies and an injection on second cohomology.
For any cohomologically connected DGA $A^{\ast}$, a $1$-minimal model $\mathcal M^{\ast}$ of $A^{\ast}$ exists and unique up to isomorphisms of DGAs
where cohomologically connected  means  $H^{0}(A^{\ast})\cong \mathbb{K}$.

Let $M$ be a compact complex manifold with a K\"ahler metric $g$.
Then, by Morgan's work (\cite{Mor}), there exists an $\R$-mixed Hodge structure $(W_{\ast},F^{\ast})$ on the $1$-minimal model ${\mathcal M}^{\ast}$ of the de Rham complex $A^{\ast}(M)$ of $M$ (Morgan's $\R$-mixed Hodge structure).
We notice that there are various choices of Morgan's $\R$-mixed Hodge structures.
On the dual Lie algebra $\frak u$ of ${\mathcal M}^{\ast}$ with the dual  $\R$-mixed Hodge structure induced by one Morgan's $\R$-mixed Hodge structure,
we define:
\begin{definition}\label{mixuupr}
A {\em mixed Hodge $\frak u$-representation} is  $(V, W_{\ast}, F^{\ast}, \Omega)$  so that:
\begin{itemize}
\item $V$ is a finite-dimensional $\R$-vector space with an $\R$-mixed Hodge structure $(W_{\ast}, F^{\ast})$

\item 
$\Omega:{\frak u}\to {\rm End}(V)$ is a representation and a morphism of $\R$-mixed Hodge structures.

\end{itemize}
\end{definition}
We notice that $\Omega$ is regarded as $\Omega\in {\mathcal M}^{1}\otimes {\rm End}(V)$ satisfying the Maurer-Cartan equation $d\Omega+\frac{1}{2}[\Omega,\Omega]=0$.
Thus, $\Omega_{\phi}=\phi(\Omega)\in A^{\ast}(M)\otimes {\rm End}(V)$ gives a flat connection on the trivial  ${\mathcal C}^{\infty}$-vector bundle $M\times V$.

We state a prototype of the main result.
\begin{prothm}\label{proun}
We can take   canonical maps $\phi:{\mathcal M}^{\ast}\to A^{\ast}(M)$ and $\varphi^{\prime} : {\mathcal M}^{\ast}_{\C}\to A^{\ast}(M)\otimes \C$  which induce isomorphisms on $0$th and first cohomologies and  injections on second cohomology 
 and define a special   $\R$-mixed Hodge structure $(W_{\ast},F^{\ast})$ on the $1$-minimal model ${\mathcal M}^{\ast}$ of $A^{\ast}(M)$ so that  for any mixed Hodge 
$\frak u$-representation ${\frak V}=(V, W_{\ast}, F^{\ast}, \Omega)$, we can construct an $\R$-VMHS  $({\bf E}_{\frak V}, {\bf W}_{\frak V \ast},{\bf F}^{\ast}_{\frak V})$  satisfying the following conditions:
\begin{enumerate}
\item ${\bf E}_{\frak V}$ is the trivial ${\mathcal C}^{\infty}$-vector bundle $M\times V$ with the flat connection $d+ \Omega_{\phi}$ where $\Omega_{\phi}=\phi(\Omega)\in A^{\ast}(M)\otimes {\rm End}(V)$.
\item $ {\bf W}_{\frak V \ast}$ is the filtration of  the vector bundle $M\times V$ induced by the weight filtration $W_{\ast}$ on $V$.
\item For some gauge transformation $a$ of the vector bundle $M\times V_{\C}$, we have ${\bf F}^{\ast}_{\frak V}=a {\bf F}^{\ast}$ such that ${\bf F}^{\ast}$ is the filtration of the vector bundle $M\times V_{\C}$ induced by the Hodge filtration $F^{\ast}$ on $V_{\C}$.
\item  $\varphi^{\prime}(\Omega)=a^{-1}da+a^{-1}\Omega_{\phi} a$.
\end{enumerate}
\end{prothm}

\begin{remark}\label{MIXMO4444}
For constructing Morgan's $\R$-mixed Hodge structure, in \cite[Section 6 and 7]{Mor}, Morgan makes:
\begin{enumerate}
\renewcommand{\labelenumi}{(\roman{enumi})}
   \item A bigraded complex $1$-minimal model ${\mathcal N}^{\ast}=\bigoplus {\mathcal N}^{\ast}(p,q)$.
\item A filtered real $1$-minimal model $( {\mathcal M}^{\ast},W_{\ast})$.

\item A filtration preserving isomorphism ${\mathcal I}: ({\mathcal M}^{\ast}\otimes \C,W_{\ast})\to  ({\mathcal N}^{\ast},W_{\ast})$ where 
\[W_{r}({\mathcal N}^{\ast})=\bigoplus_{p+q\le r} {\mathcal N}^{\ast}(p,q).\] 
\end{enumerate} 
For the filtration ${\mathcal F}^{\ast}$ on ${\mathcal M}^{\ast}$  so that 
\[F^{r} ({\mathcal M}^{\ast})={\mathcal I}^{-1}\left(\bigoplus_{p\ge r}{\mathcal M}^{\ast}(p,q)\right),
\]
$(W_{\ast},F^{\ast})$ is 
an $\R$-mixed Hodge structure on ${\mathcal M}^{\ast}$ (\cite[Section 8]{Mor}).
In this construction an isomorphism ${\mathcal I}$ is not unique and Morgan does not give explicit  one.

Our  construction  of "a special   $\R$-mixed Hodge structure" in this theorem is to give an explicit  ${\mathcal I}$ canonically determined by a K\"ahler metric $g$ without using a base point.

\end{remark}

\begin{remark}
 $\R$-VMHSs $({\bf E}_{\frak V}, {\bf W}_{\frak V \ast},{\bf F}^{\ast}_{\frak V})$ in this theorem are unipotent in the sense of Hain-Zucker \cite{HZ}.
 In \cite{HZ}, by using the mixed Hodge structure on the fundamental group derived from iterated integrals, Hain and Zucker constructed unipotent  $\R$-VMHSs associated with unipotent mixed Hodge representations of the fundamental group.
 Our construction is very similar to Hain-Zucker's construction since Morgan's mixed Hodge structure can be regarded as a mixed Hodge structure on the "tensor product" of  the fundamental group and  the field  $\R$ by Sullivan's de Rham homotopy theory (see \cite{Mor}). 
  But they are  different since Morgan's mixed Hodge structure on the $1$-minimal model is different from the mixed Hodge structure on the fundamental group derived from iterated integrals (see \cite[Section 8 and  9]{PS}).
 In fact, for our construction, we do not use "base point"
 unlike Hain-Zucker's construction.
Our construction  depends on the choice of a K\"ahler metric.
The maps $\phi:{\mathcal M}^{\ast}\to A^{\ast}(M)$ and $\varphi^{\prime} : {\mathcal M}^{\ast}_{\C}\to A^{\ast}(M)\otimes \C$ canonically  determined by a K\"ahler structure.
 An advantage of our construction is that we obtain an explicit globally defined connection forms $\Omega_{\phi}\in A^{\ast}(M)\otimes {\rm End}(V)$ and $\varphi^{\prime}(\Omega)\in A^{\ast}(M)\otimes {\rm End}(V\otimes \C)$.
 
 A more precise comparison between Hain-Zucker's construction and our construction is given in Section \ref{HZZZ}.
We can show that any unipotent $\R$-VMHS is isomorphic to  one which is constructed in Theorem (Prototype).
\end{remark}
\smallskip

\subsection{Main construction}
In  this paper we  give an extended version of  Theorem (Prototype) for obtaining non-unipotent $\R$-VMHSs.
Let $M$ be a compact K\"ahler manifold and $\rho:\pi_{1}(M,x)\to GL(V_{0})$ be a real valued representation.
Consider the real local system ${\bf E}_{0}=(\tilde{M} \times V_{0})/\pi_{1}(M,x)$ where $\tilde{M}$ is the universal covering of $M$.
We assume that $\bf E_{0}$ admits an $\R$-VHS $(\bf E_{0}, \bf F^{\ast})$ with a polarization $\bf S$.
Consider the bilinear form ${\bf S}_{x}:V_{0}\times V_{0}\to \R$.
Then we have $\rho(\pi_{1}(M,x))\subset T={\rm Aut}(V_{0},{\bf S}_{x})$.
We assume that $\rho(\pi_{1}(M,x))$ is Zariski-dense in $T$.

In this assumption, we will set up the main construction by the following way.
\begin{itemize}
\item Corresponding to $\rho:\pi_{1}(M,x)\to T$,   we consider the DGA $A^{\ast}(M,{\mathcal O}_{\rho})$ of differential forms on $M$ with values in a certain local system  equipped with  the $T$-action  defined by Deligne and Hain  \cite{Hai}.
\item We construct the canonical $T$-equivariant real $1$-minimal model $\phi:{\mathcal M}^{\ast}\to A^{\ast}(M,{\mathcal O}_{\rho})$ and the canonical  $T$-equivariant complex $1$-minimal model $\varphi:{\mathcal N}^{\ast}\to A^{\ast}(M,{\mathcal O}_{\rho}\otimes\C)$.
\item 
We construct an 
 $\R$-mixed Hodge structure on ${\mathcal M}^{\ast}$ by using  $\varphi:{\mathcal N}^{\ast}\to A^{\ast}(M,{\mathcal O}_{\rho}\otimes\C)$ which relates to the $T$-action.
\item For the $\R$-mixed Hodge structure on the dual Lie algebra $\frak u$ of ${\mathcal M}^{\ast}$, we define a {\em mixed Hodge $(T,\frak u)$-representation} $(V, W_{\ast}, F^{\ast}, \Omega)$ as a mixed Hodge $\frak u$-representation as in Definition \ref{mixuupr} with some conditions on $T$-actions.
\item  For any mixed Hodge 
$(T,\frak u)$-representation  ${\frak V}=(V, W_{\ast}, F^{\ast}, \Omega)$, we  construct an $\R$-VMHS  $({\bf E}_{\frak V}, {\bf W}_{\frak V \ast},{\bf F}^{\ast}_{\frak V})$ satisfying similar conditions as in Theorem (Prototype).
\end{itemize}
Obtained  $\R$-VMHSs $({\bf E}_{\frak V}, {\bf W}_{\frak V \ast},{\bf F}^{\ast}_{\frak V})$ can  be non-unipotent.
For each irreducible  representation $V_{\alpha}$ of $T$, we have a  $\R$-VHS on  ${\bf E}_{\alpha}=(\tilde{M} \times V_{\alpha})/\pi_{1}(M,x)$.
Such variations appear in $Gr^{\bf W}_{k}({\bf E}_{\frak V})$.
We notice that our construction is closely related to Eyssidieux-Simpson's construction in \cite{ES}.
By our construction, we can construct $\R$-VMHSs which are very similar to Eyssidieux-Simpson's VMHSs.
\smallskip
\subsection{Arrangement of the paper}

\smallskip
In Section \ref{premho}--\ref{HOK}, we will give basics of the main objects of this paper.
The main part of this paper is Section \ref{1-mKal}--\ref{HZZZ}.
In Section \ref{1-mKal}, we will give details of constructions of canonical $1$-minimal models and mixed Hodge structures.
In Section \ref{MAIVM}, we will give the definition of mixed Hodge $(T,\frak u)$-representations ${\frak V}=(V, W_{\ast}, F^{\ast}, \Omega)$ and details of constructions of $\R$-VMHSs $({\bf E}_{\frak V}, {\bf W}_{\frak V \ast},{\bf F}^{\ast}_{\frak V})$.
In Section \ref{SUBSMMM} and Section \ref{DGLMH}, we will give  techniques of producing mixed Hodge $(T,\frak u)$-modules.
In section \ref{DGLMH}, inspired by Eyssidieux-Simpson's work in \cite{ES},
we will give $\R$-VMHSs starting from any $T$-module, by using the deformation theory of differential graded Lie algebras.
In section \ref{HZZZ}, we  show that any unipotent $\R$-VMHS is isomorphic to  one which is constructed in Theorem (Prototype).
\smallskip
 {\bf  Acknowledgements.} 

The author  would  like to thank  B. Klingler  for  helpful comments and discussions on an earlier version of this paper.
He also would  like to thank P. Eyssidieux, R. Hain, M. Nori and  C. Simpson for useful conversations at Luminy. 
He also would  like to thank J. Prihdham for explaining his work.

\section{Variations of mixed Hodge structures}\label{premho}
\subsection{Mixed Hodge structures}
\begin{definition}
An {\em$\R$-Hodge structure} of weight $n$ on a $\R$-vector space $V$ is a  bigrading 
\[V_{\C}=\bigoplus_{p+q=n}V^{p,q}
\]
on the complexification $V_{\C}=V\otimes \C$
such that
\[\overline{V^{p,q}}=V^{q,p},
\]
or equivalently, a finite decreasing filtration $F^{\ast}$ on $V_{\C}$ 
such that 
\[F^{p}(V_{\C})\oplus \overline {F^{n+1-p}(V_{\C})}=V_{\C}
\]
for each $p$. 

\end{definition}
An $\R$-Hodge structure of weight $n$ corresponds to a rational  representation $h:U(1)\to GL(V)$.
This correspondence is given by 
\[V^{p,q}=\{v\in V_{\C}\vert h(t)v=t^{n-2q}v\}
\]
for $p+q=n$.

\begin{definition}
A polarization of an $\R$-Hodge structure of weight $n$ is a $(-1)^{n}$-symmetric
bilinear form $S:V\times V\to \R$ so that:
\begin{enumerate}
\item The decomposition $V_{\C}=\bigoplus_{p+q=n}V^{p,q}$ is orthogonal for the sesquilinear form $S:V_{\C}\times \overline{V_{\C}}\to \C$ .
\item $h:V_{\C}\times V_{\C}\to \C$ defined as $h(u,v)=S(Cu,\bar v)$ is a positive-definite hermitian form where $C$ is the Weil operator $C_{\vert V^{p,q}}=(\sqrt{-1})^{p-q}$.

\end{enumerate} 

\end{definition}

Suppose $V$ is finite-dimensional.
Consider the homomorphism $h:U(1)\to GL(V)$ associated with  an $\R$-Hodge structure of weight $n$.
Then for a polarization $S$, we have $h(U(1))\subset {\rm Aut}(V,S)$.

\begin{definition}
An {\em$\R$-mixed  Hodge structure} on an $\R$-vector space $V$ is a pair $(W_{\ast},F^{\ast})$
such that:
\begin{enumerate}
\item $W_{\ast}$ is an increasing filtration on $V$.
\item $F^{\ast}$ is a decreasing filtration  on $V_{\C}$  such that
the filtration on $Gr_{n}^{W} V_{\C}$ induced by $F^{\ast}$ is an $\R$-Hodge structure of weight $n$.
\end{enumerate}
We call $W_{\ast}$ the weight filtration and $F^{\ast}$ the Hodge filtration.
\end{definition}
\begin{proposition}[\cite{Del},\cite{CKS},\cite{Mor}]\label{BIGG}
Let $(W_{\ast},F^{\ast})$ be an $\R$-mixed  Hodge structure on an $\R$-vector space $V$.
Define $V^{p,q}=R^{p,q}\cap L^{p,q}$ where
\[R^{p,q}=W_{p+q}(V_{\C})\cap F^{p}(V_{\C})\] and
\[L^{p,q}=W^{p+q}(V_{\C})\cap \overline{F^{q}(V_{\C})}+\sum_{i\ge 2} W_{p+q-i}(V_{\C})\cap \overline{F^{q-i+1}(V_{\C})}.\]
Then we have 
  the   bigrading $V_{\C}=\bigoplus V^{p,q}$ such that
  \[ \overline{V^{p,q}}=V^{q,p} \quad {\rm mod} \bigoplus_{r+s<p+q} V^{r,s}, \]
\[W_{i}(V_{\C})=\bigoplus_{p+q\le i}V^{p,q} \qquad {\rm  and}\qquad
F^{i}(V_{\C})=\bigoplus_{p\ge i} V^{p,q}.\]

\end{proposition}
The  bigrading $V_{\C}=\bigoplus V^{p,q}$ is called the  bigrading of an $\R$-mixed  Hodge structure $(W_{\ast},F^{\ast})$.
If $\overline{V^{p,q}}=V^{q,p} $, then we say that  $(W_{\ast},F^{\ast})$ is  $\R$-split.
We have the converse statement.

\begin{proposition}[\cite{CKS},\cite{Mor}]\label{spmix}
Let $V$ be an $\R$-vector space.
We suppose that we have a bigrading $V_{\C}=\bigoplus V^{p,q}$ 
such that $\bigoplus_{p+q\ge n} V^{p,q}$ is an $\R$-subspace and 
  \[\overline{V^{p,q}}=V^{q,p} \qquad {\rm mod} \bigoplus_{r+s<p+q} V^{r,s}. \]
Then the filtrations $W$ and $F$ such that
$ W_{i}(V_{\C})=\bigoplus_{p+q\le i}V^{p,q}$ and 
$F^{i}(V_{\C})=\bigoplus_{p\ge i} V^{p,q}$
give an $\R$-mixed Hodge structure on $V$.
 
\end{proposition}

Indeed, these two propositions give an equivalence of categories between $\R$-mixed Hodge structures and bigradings as in Proposition \ref{spmix} see \cite{CKS}.

For a vector space $V$ with a filtration $W_{\ast}$, we denote by
${\rm Aut}(V,W_{\ast})$ the group of automorphisms preserving the filtration $W_{\ast}$
and by ${\rm Aut}_{1}(V,W_{\ast})$ the subgroup of ${\rm Aut}(V,W_{\ast})$  consisting of $b\in {\rm Aut}(V,W_{\ast})$ which induces the identity map on $Gr_{W}^{n}V$ for each $n$.
For  an $\R$-mixed Hodge structure $(W_{\ast},F^{\ast})$ on an $\R$-vector space $V$,
for any $b\in  {\rm Aut}_{1}(V_{\C},W_{\ast})$ the pair $(W_{\ast}, bF^{\ast})$ is also an $\R$-mixed Hodge structure.
\begin{proposition}[\cite{CKS}]\label{spitto}
Let $(W_{\ast},F^{\ast})$ be an $\R$-mixed  Hodge structure on an $\R$-vector space $V$. 
Then there exists  $b\in  {\rm Aut}_{1}(V_{\C},W_{\ast})$ so that $(W_{\ast},bF^{\ast})$ is  $\R$-split.

\end{proposition}
In fact, in this paper our important $\R$-mixed  Hodge structures appear as $(W_{\ast},b^{-1}F^{\ast}_{sp})$ for  $\R$-split
$\R$-mixed Hodge structures $(W_{\ast},F^{\ast}_{sp})$.
\subsection{Variations of Hodge structures}\label{VHSDEC}
Let $M$ be a complex manifold.
\begin{definition}
A {\em real variation of Hodge structure} ($\R$-VHS) of weight $n\in \Z$ over $M$ is a pair $({\bf E}, {\bf F}^{\ast})$ so that:
\begin{enumerate}
\item ${\bf E}$ is a local system of  finite-dimensional $\R$-vector spaces.
\item  ${\bf F}^{\ast}$ is a decreasing filtration of the holomorphic vector bundle ${\bf E}\otimes_{\R}{\mathcal O}_{M}$.
\item The Griffiths transversality $D{\bf F}^{r}\subset A^{1}(M, {\bf F}^{r-1})$ holds where $D$ is the flat connection associated with the local system ${\bf E}_{ \C}$.
\item For any $x\in M$, $(E_{x}, F^{\ast}_{x})$ is a $\R$-Hodge structure of weight $n$.
\end{enumerate}
 
\end{definition}
For an $\R$-VHS of weight $n$, we have the decomposition ${\bf E}_{\C}=\oplus_{p+q=n} {\bf E}^{p,q}$ of $\mathcal C^{\infty}$ vector bundles  so that ${\bf F}^{r}=\oplus_{p\ge r} {\bf E}^{p,q}$.
By the Griffiths transversality, the differential $D$ on $A^{\ast}(M,{\bf E}_{\C})$  decomposes 
$D=\partial+\theta+ \bar\partial +\bar\theta$ so that:
\[\partial: A^{a,b}({\bf E}^{c,d})\to  A^{a+1,b}({\bf E}^{c,d}) ,
\]
\[\bar\partial: A^{a,b}({\bf E}^{c,d})\to  A^{a,b+1}({\bf E}^{c,d}) ,
\]
\[\theta: A^{a,b}({\bf E}^{c,d})\to  A^{a+1,b}({\bf E}^{c-1,d+1}) 
\]
and 
\[\bar\theta: A^{a,b}({\bf E}^{c,d})\to  A^{a,b+1}({\bf E}^{c+1,d-1}).
\]
We define 
\[A^{\ast}(M,{\bf E}_{\C})^{P,Q}=\bigoplus_{a+c=P,b+d=Q}A^{a,b}({\bf E}^{c,d}),
\]
$D^{\prime}=\partial+ \bar\theta$ and $D^{\prime\prime}=\bar\partial+\theta$.
Then we have the double complex 
\[(A^{\ast}(M,{\bf E}_{\C})^{P,Q}, D^{\prime},D^{\prime\prime})\]
 as  the usual  Dolbeault complex.

\begin{definition}
A polarization of an  $\R$-VHS is a non-degenerate pairing $\bf S: {\bf E}\times {\bf E}\to \R$ so that for any $x\in M$ ${\bf S}_{x}$ is a polarization of the $\R$-hodge structure $(E_{x}, F^{\ast}_{x})$. 

\end{definition}

\subsection{Variations of mixed Hodge structures}
\begin{definition}
A {\em real variation of mixed Hodge structure} ($\R$-VMHS)  over $M$ is  $({\bf E}, {\bf F}^{\ast},{\bf W}_{\ast})$ so that:
\begin{enumerate}
\item ${\bf E}$ is a local system of  finite-dimensional $\R$-vector spaces.
\item ${\bf W}_{\ast}$ is an increasing filtration of the local system ${\bf E}$.
\item  ${\bf F}^{\ast}$ is a decreasing filtration of the holomorphic vector bundle ${\bf E}\otimes_{\R}{\mathcal O}_{M}$.

\item The Griffiths transversality $D{\bf F}^{r}\subset A^{1}(M, {\bf F}^{r-1})$ holds where $D$ is the flat connection associated with the local system ${\bf E}_{ \C}$.
\item For any $k\in \Z$, the local system $Gr^{\bf W}_{k}({\bf E})$ with the filtration induced by $ {\bf F}^{\ast}$ is an $\R$-VHS of weight $k$.

\end{enumerate}

\end{definition}

\begin{example}\label{triVMHSS}
We introduce essentially trivial cases:
\begin{itemize}
\item Let $V$ be a finite-dimensional $\R$-vector space with an $\R$-mixed Hodge structure.
Regarding $V$ as the trivial vector bundle over $M$,
$V$  is an $\R$-VMHS.
\item Let ${\bf E}_{1},\dots , {\bf E}_{k}$ be $\R$-VHSs.
Then the direct sum ${\bf E}_{1}\oplus \dots \oplus {\bf E}_{k}$ is an $\R$-VMHS.
\item Let $V_{1},\dots, V_{k}$ be finite-dimensional $\R$-vector spaces with  $\R$-mixed Hodge structures and ${\bf E}_{1},\dots , {\bf E}_{k}$ be $\R$-VHSs.
Then $\bigoplus V_{i}\otimes  {\bf E}_{i}$ is an $\R$-VMHS.

\end{itemize}

\end{example}

\section{Representations of reductive algebraic groups} 
\subsection{Coordinate rings}
Let $ T$ be a reductive algebraic group over $\R$ and $\R[T]$ the  coordinate ring  of $T$.
Let $\{V_{\alpha}\}$ be the set of isomorphism classes of irreducible representations  of $T$.
Then as a $T\times T$-module, we have an isomorphism  
$\Theta:\bigoplus_{\alpha}V_{\alpha}^{\ast}\otimes V_{\alpha}\cong \R[T]$  sending  
$f\otimes v\in V^{\ast}_{\alpha}\otimes V_{\alpha}$ to the function $T\ni t\to f(tv)\in\R$
(see \cite[Section 3]{Hai} or \cite[Theorem 27.3.9]{TY}).
Hence we have the algebra structure on $\bigoplus_{\alpha}V_{\alpha}^{\ast}\otimes V_{\alpha}$.

The multiplication on $\bigoplus_{\alpha}V_{\alpha}^{\ast}\otimes V_{\alpha}$ is given by decompositions of tensor products as the following way.
For irreducible representations $V_{\alpha}$, $V_{\beta}$, take an irreducible decomposition $V_{\alpha}\otimes V_{\beta}=\bigoplus_{i}V_{\gamma_{i}}$.
We should remark that representations $\{\gamma_{i}\}$ have multiplicities.
 For $f\otimes v\in V_{\alpha}^{\ast}\otimes V_{\alpha}$ and  $g\otimes w\in V_{\beta}^{\ast}\otimes V_{\beta}$, for this decomposition,  we take $f\otimes g=\sum h_{i}$ and $v\otimes w=\sum u_{i}$ so that $h_{i}\in V^{\ast}_{\gamma_{i}}$ and $u_{i}\in V_{\gamma_{i}}$.
Then, for $t\in T$, we have
\begin{multline*}
\left(\Theta(f\otimes v)\Theta(g\otimes w)\right)(t)=f(tv)g(tw)
=(f\otimes g)(t(v\otimes w))\\
=\left(\sum_{i} h_{i}\right)\left(\sum_{j} tu_{j}\right)=\sum_{i} h_{i}(tu_{i}).
\end{multline*}
Take the element
\[\sum_{\alpha}\left(\sum_{\gamma_{i}=\alpha} h_{i}\otimes u_{i}\right)\in \bigoplus_{\alpha}V_{\alpha}^{\ast}\otimes V_{\alpha}.
\]
Then, we have
\[\Theta\left(\sum_{\alpha}\left(\sum_{\gamma_{i}=\alpha} h_{i}\otimes u_{i}\right)\right) =\sum_{i} h_{i}(tu_{i})
\]
Thus, the multiplication induced by the isomorphism $\Theta:\bigoplus_{\alpha}V_{\alpha}^{\ast}\otimes V_{\alpha}\cong \R[T]$ is given by 
\[(f\otimes v)\cdot (g\otimes w)=\sum_{\alpha}\left(\sum_{\gamma_{i}=\alpha} h_{i}\otimes u_{i}\right)\in \bigoplus_{\alpha}V_{\alpha}^{\ast}\otimes V_{\alpha}.
\]
\begin{example}\label{SL2}
Let $T=SL_{2}(\R)$ and $V$ be the standard representation on $\R^{2}$.
We have 
\[\bigoplus_{\alpha}V_{\alpha}^{\ast}\otimes V_{\alpha}=\bigoplus_{k} S^{k}V^{\ast}\otimes S^{k}V
\]
Then $V\otimes V=\wedge^{2}V\oplus S^{2}V$.
For $f\otimes v\in V^{\ast}\otimes V$ and  $g\otimes w\in V^{\ast}\otimes V$
\begin{multline*}
(f\otimes v)\cdot (g\otimes w)\\
=( f\wedge g)\otimes (v\wedge w)+(f\times g)\otimes  (g\times w)\in \wedge^{2}V^{\ast}\otimes \wedge^{2}V\oplus S^{2}V^{\ast}\otimes S^{2}V
\end{multline*}
where $\wedge^{2} V=\R$.
\end{example}

\begin{example}
Let $T=SL_{n}(\R)$ with $n\ge 3$ and $V$ be the standard representation on $\R^{n}$. 
It is known that  any irreducible representation  $V_{\alpha}$ is given by
\[{\mathbb S}_{\lambda} V
\]
so that  ${\mathbb S}_{\lambda}$ is the Schur functor associated with a partition $\lambda$ of $d$ see \cite{FH}.
For certain set of partitions of numbers $\Lambda$, we have a bijection
\[\Lambda\ni \lambda\mapsto {\mathbb S}_{\lambda} V\in \{V_{\alpha}\}.
\]
Precisely, $\Lambda$ is the set of partitions $\lambda=(\lambda_{1},\lambda_{2},\dots, \lambda_{k})$ 
with $k<n$(see \cite[Section 15]{FH}).
Thus, each  irreducible representation  $V_{\alpha}$ can be subscriptable as  $V_{(\lambda_{1},\lambda_{2},\dots, \lambda_{k})}$.
If $n\ge 5$,  by the Littlewood–Richardson rule, we have
\[V_{(2,1)}\otimes V_{(2,1)}=V_{(4,2)}\oplus V_{(4,1,1)}\oplus V_{(3,3)}\oplus V_{(3,2,1)}^{2\oplus}\oplus V_{(3,1,1,1)}\oplus V_{(2,2,2)}\oplus V_{(2,2,1,1)}
\]
(see \cite[Section 15, Appendix A]{FH}).
 For $f\otimes v\in V_{(2,1)}^{\ast}\otimes V_{(2,1)}$ and  $g\otimes w\in V_{(2,1)}^{\ast}\otimes V_{(2,1)}$,
 we take 
 \[f\otimes g=h_{(4,2)}+h_{(4,1,1)}+ h_{(3,3)}+ h_{(3,2,1)}^{1}+ h_{(3,2,1)}^{2} +h_{(3,1,1,1)}+ h_{(2,2,2)}+h_{(2,2,1,1)}
 \]
 and 
 \[v\otimes w=u_{(4,2)}+u_{(4,1,1)}+ u_{(3,3)}+ u_{(3,2,1)}^{1}+ u_{(3,2,1)}^{2} +u_{(3,1,1,1)}+ u_{(2,2,2)}+u_{(2,2,1,1)}
 \]
 corresponding to this decomposition.
Then the  multiplication induced by the isomorphism $\Theta:\bigoplus_{\alpha}V_{\alpha}^{\ast}\otimes V_{\alpha}\cong \R[T]$ is given by 
\begin{multline*}
(f\otimes v)\cdot (g\otimes w)\\
=h_{(4,2)}\otimes u_{(4,2)}+h_{(4,1,1)}\otimes u_{(4,1,1)}+ h_{(3,3)}\otimes u_{(3,3)} \\
+h_{(3,2,1)}^{1}\otimes u_{(3,2,1)}^{1}+ h_{(3,2,1)}^{2}\otimes u_{(3,2,1)}^{2}\\
 +h_{(3,1,1,1)}\otimes u_{(3,1,1,1)}+ h_{(2,2,2)}\otimes u_{(2,2,2)}+h_{(2,2,1,1)}\otimes u_{(2,2,1,1)} \in \bigoplus_{\alpha}V_{\alpha}^{\ast}\otimes V_{\alpha}.
\end{multline*}
Since there exists a multiplicity on the irreducible representation  $V_{(3,2,1)}$, we should remark 
\[h_{(3,2,1)}^{1}\otimes u_{(3,2,1)}^{1}+ h_{(3,2,1)}^{2}\otimes u_{(3,2,1)}^{2}\in V_{(3,2,1)}^{\ast}\otimes V_{(3,2,1)}.
\]

\end{example}

\subsection{Automorphism groups of polarizations}\label{AUre}

Let $V$ be a real vector space and $F^{\ast}$ be an $\R$-Hodge structure of weight $n$ with a polarization $S$.
Take the automorphism group $T={\rm Aut}(V, S)$.
Consider $T$ as an algebraic group over $\R$.
We have $T(\C)=Sp_{2m}(\C)$ when the weight $n$ is odd or  
 $T(\C)=O(m, \C)$ when $n$ is even.
 Precisely, for the decomposition $V\otimes \C=\oplus V^{p,q}$ with
$k=\sum_{p\,\, odd}\dim V^{p,q}$ and $  l=\sum_{p\,\, even }\dim V^{p,q}$,
we have $T=Sp_{2k}(\R)$ when $n$ is odd or  
 $T=O(k,l)$ when $n$ is even.
It is known that  any irreducible representation  $V_{\alpha}$ is given by
\[{\mathbb S}_{\lambda} V\cap V^{[d]}
\]
so that  ${\mathbb S}_{\lambda}$ is the Schur functor associated with a partition $\lambda$ of $d$ and 
$V^{[d]}$ is the intersection of the kernels of all contractions $V^{\otimes d}\to V^{\otimes (d-2)}$ see \cite{FH}.
Moreover, for certain set of partitions of numbers $\Lambda$, we have a bijection
\[\Lambda\ni \lambda\mapsto {\mathbb S}_{\lambda} V\cap V^{[d]}\in \{V_{\alpha}\}.
\]
Precisely, $\Lambda$ is the set of partitions $\lambda$ so that  the Young diagrams of $\lambda$ have at most $m$ rows in case $T(\C)=Sp_{2m}(\C)$ (see \cite[Section 17.3]{FH}) or  so that the sum of the lengths of the first two columns of the Young diagrams of $\lambda$ is at most $m$ in case $T(\C)=O(m,\C)$ (see \cite[Section 19.5]{FH}).
By this, $V_{\alpha}$ admits an $\R$-Hodge structure which is induced by the $\R$-Hodge structure on $V$.
These Hodge structures are described by the following way.
Consider the homomorphism $h:U(1)\to  GL(V)$ associated with the $\R$-Hodge structure on $V$.
Then we have $h(U(1))\subset T$.
Thus,  the $\R$-Hodge structure on $V_{\alpha}$ is determined by the homomorphism $\alpha\circ h: U(1)\to GL(V_{\alpha})$.

For an irreducible representation  $V_{\alpha}$ of $T$, 
we consider the $T\times T$-module $V_{\alpha}^{\ast}\otimes V_{\alpha}$.
Take a presentation $V_{\alpha}={\mathbb S}_{\lambda} V\cap V^{[d]}$.
Then  $V_{\alpha}^{\ast}$ admits an $\R$-Hodge structure of weight $-nd$ induced by $F^{\ast}$ and $V_{\alpha}$ admits an  $\R$-Hodge structure of weight $nd$ induced by $F^{\ast}$ and so $V_{\alpha}^{\ast}\otimes V_{\alpha}$ admits an $\R$-Hodge structure of weight $0$.
By the isomorphism  $\Theta:\bigoplus_{\alpha}V_{\alpha}^{\ast}\otimes V_{\alpha}\cong \R[T]$, 
we obtain the $\R$-Hodge structure of weight $0$ on $ \R[T]$.
Consider the homomorphisms $h:U(1)\to   T$ associated with the $\R$-Hodge structure $F^{\ast}$ on $V$.
Then, the $\R$-Hodge structure of weight $0$ on $ \R[T]$ is determined by the $U(1)$-action induced by the homomorphism $h\times h: U(1)\to  T\times T$ and the $T\times T$-module structure on $ \R[T]$.
Since the multiplication  $ \R[T]\otimes  \R[T]\to   \R[T]$ is $T\times T$-equivariant,
we can say that  the multiplication $ \R[T]\otimes  \R[T]\to   \R[T]$ is a morphism of $\R$-Hodge structures.
Moreover,  for each  irreducible representation  $V_{\alpha}$ with the $\R$-Hodge structure induced by $F^{\ast}$, 
the corresponding co-module structure $V_{\alpha}\to  V_{\alpha}\otimes \R[T]$ is a morphism of  $\R$-Hodge structures.
 
 Let $t\in T$.
We consider the  varied  $\R$-Hodge structure $tF^{\ast}$ on $V$.
This  $\R$-Hodge structure corresponds to the homomorphism $tht^{-1}: U(1)\to   T$.
On $\bigoplus_{\alpha}V_{\alpha}^{\ast}\otimes V_{\alpha}\cong \R[T]$, we change the $\R$-Hodge structure on each left $V_{\alpha}^{\ast}$ (resp. right $V_{\alpha}$) to the one  associated with the varied $\R$-Hodge structure $tF^{\ast}$.
Then, by the $T\times T$-equivariance of the multiplication $ \R[T]\otimes  \R[T]\to   \R[T]$, we can also say that the multiplication $ \R[T]\otimes  \R[T]\to   \R[T]$ is a morphism of $\R$-Hodge structures for such alternative $\R$-Hodge structure on $\bigoplus_{\alpha}V_{\alpha}^{\ast}\otimes V_{\alpha}\cong \R[T]$.
This is important for considering $\R$-VHSs.
We notice that $T$ acts on transitively on the classifying space of polarized  $\R$-Hodge structures on $V$ with the fixed Hodge numbers (see \cite{Car} for instance).


\section{DGA $A^{\ast}(M,{\mathcal O}_{\rho})$}\label{Hadg}
Let $M$ be a ${\mathcal C}^{\infty}$-manifold, $T$  a reductive algebraic group over $\R$ and  $\rho:\pi_{1}(M,x)\to T$ be a real valued representation.
Assume that $\rho(\pi_{1}(M,x))$ is Zariski-dense in $T$.
Let $\{V_{\alpha}\} $ be the set of isomorphism classes of irreducible representations of $T$.
Consider the local systems ${\bf E}_{\alpha}=(\tilde{M} \times V_{\alpha})/\pi_{1}(M,x)$.
Denote by $A^{\ast}(M,{\bf E}_{\alpha})$  the space of ${\bf E}_{\alpha}$-valued $C^{\infty}$-differential forms.
Consider the cochain complex 
\[A^{\ast}(M,{\mathcal O}_{\rho})=\bigoplus_{\alpha} A^{\ast}(M, {\bf E}^{\ast}_{\alpha})\otimes V_{\alpha}\]
with the differential  $D=\bigoplus_{\alpha} D_{\alpha}$.
Then by the wedge product and the multiplication on $\bigoplus_{\alpha} V_{\alpha}^{\ast}\otimes V_{\alpha}\cong \R[T]$, $\left(A(M, {\mathcal O}_{\rho}), D\right)$ is a cohomologically connected  DGA with the $T$-action given by each  irreducible representation $V_{\alpha}$.

Let $V$ a finite-dimensional $T$-module.
We consider the algebra $A^{\ast}(M,{\mathcal O}_{\rho})\otimes {\rm End}(V)$ with the $T$-action and the $T$-module $V\otimes \R[T]$.
Then, by the isomorphism  $\left(V\otimes \R[T]\right)^{T}\cong V$, 
 the space $(A^{\ast}(M,{\mathcal O}_{\rho})\otimes {\rm End}(V))^{T}$  is identified with $A^{\ast}(M, {\rm End} ({\bf E}))$
where ${\bf E}$ is  the local system   associated with $\rho:\pi_{1}(M,x)\to T$ and the  $T$-module structure on $V$.
Hence, an element $\Omega\in (A^{1}(M,{\mathcal O}_{\rho})\otimes {\rm End}(V))^{T}$ satisfying the Maurer-Cartan equation $D\Omega+\frac{1}{2}[\Omega,\Omega]=0$ gives the deformed flat connection $D+\Omega$ on the ${\mathcal C}^{\infty}$ vector bundle $\bf E$.

\section{Hodge theory on compact K\"ahler manifolds}\label{HOK}
Let $M$ be a compact complex manifold.
We assume that $M$ admits a K\"ahler metric $g$. 
Let $({\bf E}, {\bf F}^{\ast})$  be an $\R$-VHS of weight $n$ over $M$ with a polarization $\bf S$.
We consider the double complex $(A^{\ast}(M,{\bf E}_{\C})^{P,Q}, D^{\prime},D^{\prime\prime})$ as in Subsection \ref{VHSDEC}.
We define the differential operator $D^{c}=\sqrt{-1}(D^{\prime\prime}-D^{\prime})$ on the real valued differential forms $A^{\ast}(M,{\bf E})$.
By the Hermitian  metric  on ${\bf E}_{\C}$ associated with the polarization $\bf S$ and the K\"ahler metric $g$, we define the adjoints
$D^{\ast}$, $(D^{\prime})^{\ast}$, $(D^{\prime\prime})^{\ast}$
and $(D^{c})^{\ast}$ of  differential operators.
For the K\"ahler form $\omega$ associated with $g$, we  consider the adjoint operator $\Lambda$ of the Lefschetz operator $A^{\ast}(M,{\bf E})\ni \alpha\mapsto \omega\wedge \alpha \in A^{\ast+2}(M, {\bf E})$.
In the same way as   the usual K\"ahler identity, we have 
\[[\Lambda, D]=-(D^{c})^{\ast}
\]
and  this equation gives
\[\Delta_{D}=2\Delta_{D^{\prime}}=2\Delta_{D^{\prime\prime}}
\]
where $\Delta_{D}$, $\Delta_{D^{\prime}}$ and $\Delta_{D^{\prime\prime}}$  are the Laplacian operators (see \cite{Zu}).
Write
\[{\mathcal H}^{r}(M,{\bf E})={\rm ker}(\Delta_{D})_{\vert  A^{r}(M,{\bf E}) } \qquad {\rm and}\ \qquad {\mathcal H}^{P,Q}(M,{\bf E}_{\C})={\rm ker}(\Delta_{D^{\prime\prime}})_{\vert ( A^{\ast}(M,{\bf E}_{\C}))^{P,Q} }.\]
Then we have the Hodge  decomposition 
\[{\mathcal H}^{r}(M,{\bf E}_{\C})=\bigoplus_{P+Q=n+r}{\mathcal H}^{P,Q}(M,{\bf E}_{\C}).
\]
Since $\Lambda $ is a map of degree $-2$, by the K\"ahler identity,  we have the following useful equations
\[{\mathcal H}^{1}(M,{\bf E})={\rm ker} D_{\vert  A^{1}(M,{\bf E}) }\cap {\rm ker} D^{c} _{\vert  A^{1}(M,{\bf E}) }
\]
and 
\[{\mathcal H}^{P,Q}(M,{\bf E}_{\C})={\rm ker}D^{\prime}_{\vert  A^{\ast}(M,{\bf E}_{\C})^{P,Q}} \cap {\rm ker}D^{\prime\prime}_{\vert  A^{\ast}(M,{\bf E}_{\C})^{P,Q} }
\]
for $P+Q=1+n$.
By these equations, we can say that the $\R$-Hodge structure on ${\mathcal H}^{1}(M,{\bf E}) $ is independent of the K\"ahler metric $g$ and the polarization on $({\bf E}, {\bf F}^{\ast})$.
As the usual way,
we have the polarization 
\[{\mathcal H}^{1}(M,{\bf E})\times {\mathcal H}^{1}(M,{\bf E})\ni (\alpha,\beta)\mapsto \int \alpha\wedge \beta \wedge\omega^{\dim_{\C}M-1}\in \R
\]
which depends on  the K\"ahler metric $g$ and the polarization  $\bf S$  on $({\bf E}, {\bf F}^{\ast})$.

By the same argument as \cite[Section 5]{DGMS}, we have the following $DD^{c}$-Lemma and $D^{\prime}D^{\prime\prime}$-Lemma.
\begin{theorem}\label{DDc}
\begin{description}
\item[($DD^{c}$-Lemma)]
On $A^{\ast}(M,{\bf E})$,
\[{\rm im}D\cap {\rm ker}D^{c}={\rm ker}D\cap {\rm im}D^{c}={\rm im}DD^{c}.
\]
Moreover, there exists a linear map $F_{g}: {\rm im}D\cap {\rm ker}D^{c}\to A^{\ast-2}(M,{\bf E})$  so that $\alpha=DD^{c}F_{g}\alpha$ for $\alpha\in {\rm im}D\cap {\rm ker}D^{c}$.
\item[($D^{\prime}D^{\prime\prime}$-Lemma)]
On $A^{\ast}(M,{\bf E_{\C}})^{P,Q}$,
\[{\rm im}D^{\prime}\cap {\rm ker}D^{\prime\prime}={\rm ker}D^{\prime}\cap {\rm im}D^{\prime\prime}={\rm im}D^{\prime}D^{\prime\prime}.
\]
Moreover, there exists a linear map $F^{\prime}_{g}: {\rm im}D^{\prime}\cap {\rm ker}D^{\prime\prime}\to A^{\ast}(M,{\bf E}_{\C})^{P-1,Q-1}$  so that $\alpha=D^{\prime}D^{\prime\prime}F^{\prime}_{g}\alpha$ for $\alpha\in {\rm im}D^{\prime}\cap {\rm ker}D^{\prime\prime}$.

\end{description}
\end{theorem}
In fact, we can write $F_{g}= D^{\ast}G_{D}(D^{c})^{\ast}G_{D^{c}}$ and $F^{\prime}_{g}=(D^{\prime})^{\ast}G_{D^{\prime}}(D^{\prime\prime})^{\ast}G_{D^{\prime\prime}}$ where $G_{D}$, $G_{D^{c}}$, $G_{D^{\prime}}$ and $G_{D^{\prime\prime}}$ are the Green operators (see \cite[Proof of (5.11)]{DGMS}).

We consider the sub-complexes
 \[{\rm ker}D^{c} \subset A^{\ast}(M,{\bf E}) \qquad {\rm and }\qquad {\rm ker}D^{\prime}\subset  A^{\ast}(M,{\bf E_{\C}}).\]
The $DD^{c}$-Lemma and $D^{\prime}D^{\prime\prime}$-Lemma 
imply the following  "formality" results (see \cite[Section 6]{DGMS}, also \cite[Section 7]{Gold}). 
\begin{corollary}\label{Forma}
\begin{itemize}
\item The inclusions \[{\rm ker}D^{c} \subset A^{\ast}(M,{\bf E})  \qquad {\rm and }\qquad  {\rm ker}D^{\prime}\subset  A^{\ast}(M,{\bf E_{\C}})\]
induce cohomology isomorphisms.
\item The quotients 
\[{\rm ker}D^{c} \to H^{\ast}( A^{\ast}(M,{\bf E}), D^{c})  \qquad {\rm and }\qquad  {\rm ker}D^{\prime}\to  H^{\ast}(A^{\ast}(M,{\bf E_{\C}}), D^{\prime})\]
 induce cohomology isomorphisms.
\item We have isomorphisms 
\[H^{\ast}( A^{\ast}(M,{\bf E}), D)\cong  H^{\ast}( A^{\ast}(M,{\bf E}), D^{c}))\]
 and 
\[ H^{\ast}(A^{\ast}(M,{\bf E_{\C}}), D) \cong H^{\ast}(A^{\ast}(M,{\bf E_{\C}}), D^{\prime}).\]
\end{itemize}

\end{corollary}

\begin{remark}\label{ddcinvvv}
By $DD^{c}$-Lemma and $D^{\prime}D^{\prime\prime}$-Lemma, on $0$-forms $A^{0}(M,{\bf E}_{\C})$, 
we have
\[{\rm ker}D={\rm ker}D^{\prime}={\rm ker}D^{\prime\prime}={\rm ker}D^{c}={\rm ker}D^{\prime}D^{\prime\prime}.
\]
By this, if $\alpha \in A^{0}(M,{\bf E}_{\C})$ satisfies $D^{\prime}D^{\prime\prime}\alpha=0$ (resp. $DD^{c}\alpha=0$), then we have $D^{\prime}\alpha=0$ (resp. $D^{c}\alpha=0$)
and so  for $\beta\in {\rm im}D\cap {\rm ker}D^{c}\cap A^{2}(M,{\bf E})$ (resp. ${\rm im}D^{\prime}\cap {\rm ker}D^{\prime\prime}\cap A^{2}(M,{\bf E}_{\C})^{P,Q}$), a $D^{c}$-exact (resp. $D^{\prime}$-exact) $1$-form $D^{c}\alpha$ (resp. $D^{\prime}\alpha$) satisfying $\beta=DD^{c}\alpha$ (resp. $\beta=D^{\prime}D^{\prime\prime}\alpha$) is unique.
Hence the maps  \[D^{c}F_{g}: {\rm im}D\cap {\rm ker}D^{c}\cap A^{2}(M,{\bf E}) \to A^{1}(M,{\bf E})\] and \[D^{\prime}F^{\prime}_{g}: {\rm im}D^{\prime}\cap {\rm ker}D^{\prime\prime}\cap A^{2}(M,{\bf E}_{\C})^{P,Q}\to A^{1}(M,{\bf E}_{\C})^{P,Q-1}\]
  do not depend on the choice of a K\"ahler metric $g$.

Let $f:M^{\prime}\to M$ be a holomorphic map from a compact K\"ahler manifold $M^{\prime}$.
Then we have the pull-back $\R$-VHS $(f^{\ast}{\bf E}, f^{\ast}{\bf F}^{\ast})$ over $M^{\prime}$.
Thus, we also define the maps $F_{g}$ and $F^{\prime}_{g}$ for the differential  forms on $M^{\prime}$ with values in $f^{\ast}{\bf E}$.
Consider the pull-back map $f^{\ast}:A^{\ast}(M,{\bf E})\to A^{\ast}(M^{\prime}, f^{\ast}{\bf E})$.
Then, by the above argument, we can say 
\[f^{\ast}\circ  D^{c}F_{g}=D^{c}F_{g}\circ f^{\ast}  \qquad  {\rm on} \qquad  {\rm im}D\cap {\rm ker}D^{c}\cap A^{2}(M,{\bf E})\] and 
 \[f^{\ast}\circ  D^{\prime}F^{\prime}_{g}=D^{\prime}F^{\prime}_{g}\circ f^{\ast} \qquad   {\rm  on} \qquad  {\rm im}D^{\prime}\cap {\rm ker}D^{\prime\prime}\cap A^{2}(M,{\bf E}_{\C})^{P,Q}.\]
\end{remark}

\section{$1$-minimal models  on compact K\"ahler manifolds}\label{1-mKal}
\subsection{Assumptions}
Let $M$ be a compact K\"ahler manifold and $\rho:\pi_{1}(M,x)\to GL(V_{0})$ be a real valued representation.
Consider the real local system ${\bf E}_{0}=(\tilde{M} \times V_{0})/\pi_{1}(M,x)$ where $\tilde{M}$ is the universal covering of $M$.
We assume that $\bf E_{0}$ admits an $\R$-VHS $(\bf E_{0}, \bf F^{\ast})$ of weight $N_{0}$ with a polarization $\bf S$.
Consider the bilinear form ${\bf S}_{x}:V_{0}\times V_{0}\to \R$.
Then we have $\rho(\pi_{1}(M,x))\subset T={\rm Aut}(V_{0},{\bf S}_{x})$.
We assume that  $\rho(\pi_{1}(M,x))$ is Zariski-dense in $T$.

\subsection{Summary  of this section}
In this section, we will give the canonical mixed Hodge structure on the $1$-minimal model of the DGA  $A^{\ast}(M,{\mathcal O}_{\rho})$ by the following way:
\begin{enumerate}
\item We consider the DGA  $A^{\ast}(M,{\mathcal O}_{\rho})$ associated with the representation  $\rho:\pi_{1}(M,x)\to T$.
We take the structure of bidifferential bigraded algebra  on $A^{\ast}(M,{\mathcal O}_{\rho}\otimes \C)$ by using the variation on ${\bf E}_{0}$. (Subsection \ref{DGKAA})
\item We construct a $T$-equivariant "real" $1$-minimal model $\phi:{\mathcal M}^{\ast}\to A^{\ast}(M,{\mathcal O}_{\rho})$ with a grading ${\mathcal M}^{\ast}=\bigoplus {\mathcal M}^{\ast}_{k}$. (Subsection \ref{dedec})
\item We construct a $T$-equivariant "complex" $1$-minimal model $\varphi:{\mathcal N}^{\ast}\to A^{\ast}(M,{\mathcal O}_{\rho}\otimes \C)$ with a bigrading ${\mathcal N}^{\ast}=\bigoplus ({\mathcal N}^{\ast})^{P,Q}$. (Subsection \ref{dbd})

\item We take a $T$-equivariant isomorphism ${\mathcal I}:{\mathcal M}^{\ast}_{\C}\to {\mathcal N}$ which is compatible with filtrations $W_{k}({\mathcal M}^{\ast})=\bigoplus_{i\le k} {\mathcal M}^{\ast}_{i}$ and $W_{k}({\mathcal N}^{\ast})=\bigoplus_{P+Q\le k} ({\mathcal N}^{\ast})^{P,Q}$
and a $T$-equivariant homotopy $H:{\mathcal M}^{\ast}_{\C}\to A^{\ast}(M,{\mathcal O}_{\rho}\otimes \C)\otimes [t,dt]$ from  $\varphi \circ {\mathcal I}$ to $\phi$.
(Subsection \ref{isho})
\item By $F^{r}({\mathcal M}^{\ast}_{\C})={\mathcal I}^{-1}(\bigoplus_{P\ge r} ({\mathcal N}^{\ast})^{P,Q})$, we define an $\R$-mixed Hodge structure on ${\mathcal M}^{\ast}$.
(Subsection \ref{MIXMIX})
\end{enumerate}

Avoiding the arguments on $T$, we can find these constructions in \cite{DGMS} and \cite{Mor} for the usual de Rham complex $A^{\ast}(M)$.

\subsection{DGAs $A^{\ast}(M,{\mathcal O}_{\rho})$ on compact K\"ahler manifolds}\label{DGKAA}

Let $\{V_{\alpha}\} $ be the set of isomorphism classes of irreducible representations of $T$.
Consider the local systems ${\bf E}_{\alpha}=(\tilde{M} \times V_{\alpha})/\pi_{1}(M,x)$.
Then, for certain set $\Lambda$ of partitions of numbers, we have a bijection
\[\Lambda\ni \lambda\mapsto {\mathbb S}_{\lambda}{\bf E}_{0}\cap {\bf E}^{[d]}_{0}\in \{{\bf E}_{\alpha}\} \,\,\, ({\rm resp.}\,\, {\mathbb S}_{\lambda}V_{0}\cap V^{[d]}_{0}\in\{V_{\alpha}\}).
\]
Precisely, $\Lambda$ is the set of partitions $\lambda$ so that  the Young diagrams of $\lambda$ have at most $m$ rows in case $T(\C)=Sp_{2m}(\C)$ (see \cite[Section 17.3]{FH}) or  so that the sum of the lengths of the first two columns of the Young diagrams of $\lambda$ is at most $m$ in case $T(\C)=O(m,\C)$ (see \cite[Section 19.5]{FH}).
Hence, any ${\bf E}_{\alpha}$ (resp. $V_{\alpha}$) admits a polarized $\R$-VHS (resp.  $\R$-Hodge structure) induced by $({\bf E}_{0}, {\bf F}^{\ast})$ (resp. $(V_{0}, {\bf F}^{\ast}_{x})$).
We put the $\R$-Hodge structure on $\R[T]$ of weight $0$ induced by $(V_{0}, {\bf F}^{\ast}_{x})$ as in Section \ref{AUre}.
Then,
for each  irreducible representation  $V_{\alpha}$, 
the corresponding co-module structure $V_{\alpha}\to  V_{\alpha}\otimes \R[T]$ is a morphism of  $\R$-Hodge structures.
By the $\R$-Hodge structure on $\R[T]$,
we have the bigrading $\C[T]=\bigoplus_{p} \C[T]^{p,-p}$.

We consider the DGA $A^{\ast}(M,{\mathcal O}_{\rho})$ as in Section \ref{Hadg}.
We define the bigrading on $A^{\ast}(M,{\mathcal O}_{\rho}\otimes \C)$ as the following way.
\begin{definition}
\[A^{\ast}(M,{\mathcal O}_{\rho}\otimes \C)^{P,Q}=\bigoplus_{\alpha}\bigoplus_{S+U=P,T+V=Q} A^{\ast}(M,{\bf E}^{\ast}_{\alpha}\otimes \C)^{S,T}\otimes (V_{\alpha}\otimes \C)^{U,V}.
\]
\end{definition}
By the argument in Subsection \ref{AUre}, on each case, we can say that the product on  $A^{\ast}(M,{\mathcal O}_{\rho}\otimes \C)$ is compatible with the bigrading.
Hence,  we have the bidifferential  bigraded algebra structure
\[(A^{\ast}(M,{\mathcal O}_{\rho}\otimes \C)^{P,Q}, D^{\prime},D^{\prime\prime}).
\]
We have
\[A^{R}(M,{\mathcal O}_{\rho}\otimes \C)=\bigoplus_{R=P+Q}A^{\ast}(M,{\mathcal O}_{\rho}\otimes \C)^{P,Q}.
\]

Since the co-module structure $V_{\alpha}\to  V_{\alpha}\otimes \R[T]$ is a morphism of  $\R$-Hodge structures,
the co-module structure 
\[A^{\ast}(M,{\mathcal O}_{\rho}\otimes \C)\to A^{\ast}(M,{\mathcal O}_{\rho}\otimes \C)\otimes \C[T]
\]
corresponding to the $T$-action preserves the bigradings on $\C[T]$ and $A^{\ast}(M,{\mathcal O}_{\rho}\otimes \C)$.

\begin{example}\label{reimsu}
Let $M$ be a compact Riemann surface of genus $g\ge 2$.
Then $M$ is a compact quotient ${\mathbb H}/\Gamma$ of the upper-half plane ${\mathbb H}$ by a discrete subgroup $\Gamma$ in $PSL_{2}(\R)$.
Take a lifting $\rho:\Gamma\to SL_{2}(\R)$ of the embedding of $\Gamma$ into $PSL_{2}(\R)$.
By the Borel density, $\rho(\Gamma)$ is Zariski-dense in $SL_{2}(\R)$.
Consider the local system ${\bf E}_{0}=({\mathbb H}\times \R^{2})/\Gamma$.
We regard ${\mathbb H}$ as the classifying space of polarized $\R$-Hodge structures of weight $1$ on $\R^{2}$. 
Then, considering  the identity map on ${\mathbb H}$  as a period map, the local system ${\bf E}_{0}$ admits a polarized   $\R$-VHS by taking certain decomposition  ${\bf E}_{0}\otimes \C={\bf E}_{0}^{1,0}\oplus{\bf E}_{0}^{0,1}$.
It is known that we can take ${\bf E}_{0}^{1,0}=K^{\frac{1}{2}}$ and ${\bf E}_{0}^{0,1}=K^{-\frac{1}{2}}$ where $K^{\frac{1}{2}}$ is a square-root of the canonical bundle $K$ on $M$ and $D^{\prime\prime}=\bar\partial+\theta$ such that  $\theta$ is $1 \in K\otimes {\rm Hom}(K^{\frac{1}{2}},K^{-\frac{1}{2}})\cong \C$ (see \cite{Hit}, \cite{Sim1} and \cite{Goll}).

We consider the DGA $A^{\ast}(M,{\mathcal O}_{\rho})$.
By Example \ref{SL2}, we have
\[A^{\ast}(M,{\mathcal O}_{\rho})=\bigoplus_{k=0}^{\infty}A^{\ast}(M, S^{k}{\bf E}_{0}^{\ast})\otimes S^{k}V_{0}
\]
where $V_{0}=({\bf E}_{0})_{x}$.
By the multiplication on the coordinate ring $\R[SL_{2}(\R)]$ (see Example \ref{SL2}), we have
\[\left(A^{\ast}(M, {\bf E}_{0}^{\ast})\otimes  V_{0}\right)\wedge \left(A^{\ast}(M, {\bf E}_{0}^{\ast})\otimes  V_{0}\right)\subset \left(A^{\ast}(M, \R)\otimes  \R \right)\oplus \left( A^{\ast}(M, S^{2}{\bf E}_{0}^{\ast})\otimes  S^{2}V_{0}\right).
\]

By the decomposition ${\bf E}_{0}\otimes \C={\bf E}_{0}^{1,0}\oplus{\bf E}_{0}^{0,1}$,
we have the decompositions $S^{k}{\bf E}_{0}^{\ast}={\bf E}_{0}^{k,0}\oplus {\bf E}_{0}^{k-1,1}\dots \oplus {\bf E}_{0}^{0,k}$ and $S^{k}V_{0}^{\ast}=V_{0}^{k,0}\oplus V_{0}^{k-1,1}\dots \oplus V_{0}^{0,k}$ where ${\bf E}_{0}^{p,q}=({\bf E}_{0}^{1,0})^{p}\times ({\bf E}_{0}^{0,1})^{q}$ and $V_{0}^{p,q}=(V_{0}^{1,0})^{p}\times (V_{0}^{0,1})^{q}$ such that the multiplications are symmetric products.
We have
\[A^{\ast}(M,{\mathcal O}_{\rho}\otimes \C)=\bigoplus_{k=0}^{\infty}\bigoplus_{0\le i,j\le k}A^{\ast,\ast}(M, ({\bf E}_{0}^{i,k-i})^{\ast})\otimes V_{0}^{j,k-j}.
\]
For the bigrading $A^{\ast}(M,{\mathcal O}_{\rho}\otimes \C)^{P,Q}$ as the above argument, 
an element in  \[A^{a,b}(M, ({\bf E}_{0}^{i,k-1})^{\ast})\otimes V_{0}^{j,k-j}\] is of type $(a-i+j, b+i-j)$.
By the multiplication on $\R[SL_{2}(\R)]$, we have
\begin{multline*}
\left(A^{a,b}(M, ({\bf E}_{0}^{1,0})^{\ast})\otimes V_{0}^{0,1}\right)\wedge \left(A^{c,d}(M, ({\bf E}_{0}^{0,1})^{\ast})\otimes V_{0}^{1,0}\right)\\
\subset  \left(A^{a+c,b+d}(M, \C)\otimes \C\right)\oplus  \left( A^{a+c,b+d}(M, ({\bf E}_{0}^{1,1})^{\ast})\otimes V_{0}^{1,1}\right).
\end{multline*}

\end{example}

\subsection{The $1$-minimal model associated with $DD^{c}$-Lemma}\label{dedec}
On the DGA $A^{\ast}(M,{\mathcal O}_{\rho})$ with the  $T$-action,
we also consider another differential $D^{c}=\sqrt{-1}(D^{\prime\prime}-D^{\prime})$.
Then, by Theorem \ref{DDc} and the last subsection, 
we can say that 
on $A^{\ast}(M,{\mathcal O}_{\rho})$, we have the equality
\[{\rm im}D\cap {\rm ker}D^{c}={\rm ker}D\cap {\rm im}D^{c}={\rm im}DD^{c}
\]
and  there exist a $T$-equivariant linear map $F_{g}: {\rm im}D\cap {\rm ker}D^{c}\to A^{\ast-2}(M,{\mathcal O}_{\rho})$  so that $\alpha=DD^{c}F_{g}\alpha$ for $\alpha\in {\rm im}D\cap {\rm ker}D^{c}$.
By using this, we construct the DGAs ${\mathcal M}^{\ast}(n)=\bigwedge ( \V_{1}\oplus\dots \oplus \V_{n})$ generated by  elements of degree $1$ and the homomorphisms $\phi_{n}:{\mathcal M}^{\ast}(n)\to {\rm ker}D^{c}\subset A^{\ast}(M,{\mathcal O}_{\rho})$ by the  inductive way.
More precisely, we will construct the components $\V_{n}$ and the maps $d: \V_{n}\to \sum_{i+j=n}\V_{i}\wedge \V_{j}$ by induction in the following.
\begin{itemize}
\item $\V_{1}={\rm ker}D\cap {\rm ker}D^{c}\cap A^{1}(M,{\mathcal O}_{\rho})$, the homomorphism
$\phi_{1}:\bigwedge \V_{1}\to {\rm ker}D^{c}$
so that on $\V_{1}$, $\phi_{1} $   is the natural inclusion $\V_{1}\hookrightarrow {\rm ker}D^{c}$.
\item For the  quotient map $q:{\rm ker}D\cap {\rm ker}D^{c}\to H^{\ast}({\rm ker}D^{c})$, 
\[\V_{2}={\rm ker} \left( q\circ \phi_{1}: \bigwedge^{2} \V_{1}\to H^{2}({\rm ker}D^{c})\right)
\]
Define the DGA ${\mathcal M}^{\ast}(2)=\bigwedge (\V_{1}\oplus \V_{2})$ with the differential $d$ so that $d$ is $0$ on $\V_{1}$ and $d$ on $\V_{2}$ is the natural inclusion $\V_{2}\hookrightarrow  \bigwedge^{2} \V_{1}$.
Define the homomorphism $\phi_{2}:{\mathcal M}^{\ast}(2)\to {\rm ker}D^{c}$ which is an extension of $\phi_{1}$ so that $\phi_{2}(v)=D^{c}F_{g}(\phi_{1}(dv))$.
\item For $n\ge 2$, consider the DGA ${\mathcal M}^{\ast}(n)=\bigwedge(\V_{1}\oplus \V_{2}\oplus\dots\oplus \V_{n})$ with the homomorphism $\phi_{n}:{\mathcal M}^{\ast}(n)\to {\rm ker}D^{c}$ we have constructed.
We can say that $\phi_{n}(v)\in {\rm ker}D\cap {\rm ker}D^{c}$ for $v\in \V_{1}$ and as an inductive hypothesis we assume $\phi_{n}(v)\in {\rm im} D^{c}$ for $v\in \V_{2}\oplus\dots\oplus \V_{n}$.

Let
\[\V_{n+1}={\rm ker} d_{ \vert \sum_{i+j=n+1} \V_{i}\wedge \V_{j}}.\]
 Define the extended DGA ${\mathcal M}^{\ast}(n+1)={\mathcal M}^{\ast}(n)\otimes \bigwedge \V_{n+1}$
so that the differential $d$ is defined on $\V_{n+1}$ as the natural inclusion $\V_{n+1}\hookrightarrow \sum_{i+j=n+1} \V_{i}\wedge \V_{j}$.
The homomorphism $\phi_{n+1}:{\mathcal M}^{\ast}(n+1)\to {\rm ker}D^{c}$ is defined by 
$\phi_{n+1}(v)=D^{c}F_{g}(\phi_{n}(dv))$ for $v\in \V_{n+1}$.

\end{itemize}

Let $\varinjlim {\mathcal M}^{\ast}(n)={\mathcal M}^{\ast}$ and $\phi=\varinjlim \phi_{n}:{\mathcal M}^{\ast}\to {\rm ker}D^{c}\subset A^{\ast}(M,{\mathcal O}_{\rho})$.

By the construction, 
each $\V_{i}$ is a $T$-module so that the map $\phi: {\mathcal M}^{\ast}\to  A^{\ast}(M,{\mathcal O}_{\rho})$ is $T$-equivariant.
Our construction is in fact the construction of a $1$-minimal model of ${\rm ker}D^{c}$ (see \cite[Theorem 13.1]{GMo}).
Since the inclusion $ {\rm ker}D^{c}\subset A^{\ast}(M,{\mathcal O}_{\rho})$ induces a cohomology isomorphism,
we can say that 
the map $\phi: {\mathcal M}^{\ast}\to  A^{\ast}(M,{\mathcal O}_{\rho})$ induces isomorphisms on $0$-th and first cohomology and an injection on second cohomology.
Define the multiplicative grading  ${\mathcal M}^{\ast}_{k}$ so that elements in $\V_{k}$ are of type   $k$.
Then, each ${\mathcal M}^{\ast}_{k}$ is a $T$-module and  we have $d{\mathcal M}^{\ast}_{k}\subset {\mathcal M}^{\ast}_{k}$.
Define the increasing filtration $W_{\ast}$ so that $W_{n}({\mathcal M}^{\ast})=\bigoplus_{k\le n}{\mathcal M}^{\ast}_{k}$.

\begin{remark}\label{sp11}
By the isomorphism $H^{\ast}(M,{\mathcal O}_{\rho})\cong H^{\ast}({\rm ker}D^{c}) $ as in  Theorem \ref{Forma}, the map 
$q\circ \phi_{1}: \bigwedge^{2} \V_{1}\to H^{2}({\rm ker}D^{c})$ is identified with the cup product 
\[\bigwedge^{2} H^{1}(M,{\mathcal O}_{\rho})\to H^{2}(M,{\mathcal O}_{\rho}).\]
For $n\ge 2$, forgetting the construction of the homomorphism $\phi_{n+1}$, the DGA 
${\mathcal M}^{\ast}(n+1)={\mathcal M}^{\ast}(n)\otimes \bigwedge \V_{n+1}$ is only  determined by the DGA ${\mathcal M}^{\ast}(n)$.
Thus, the DGA ${\mathcal M}^{\ast}$ is only determined by 
$H^{1}(M,{\mathcal O}_{\rho})$ with the cup product.
This property is called $1$-formality.

Consider the $\R$-Hodge structures on $H^{1}(M,{\mathcal O}_{\rho})$ and $H^{2}(M,{\mathcal O}_{\rho})$.
Then the cup product on $H^{1}(M,{\mathcal O}_{\rho})$  is a morphism of $\R$-Hodge structures.
Thus, $\V_{2}$ admits a $\R$-Hodge structure of weight $2$ and $d:\V_{2}\to \V_{1}\wedge \V_{1}$ is a morphism of Hodge structure.
Inductively, we can easily say that each $\V_{n}$ admits an $\R$-Hodge structure of weight $n$ and the restriction  $d:\V_{n}\to \sum_{i+j=n} \V_{i}\wedge \V_{j}$ is 
a morphism of  $\R$-Hodge structures.

\end{remark}


\subsection{The $1$-minimal model associated with $D^{\prime}D^{\prime\prime}$-Lemma}\label{dbd}
Consider the bidifferential bigraded algebra 
\[\left(A^{\ast}(M,{\mathcal O}_{\rho}\otimes \C)^{P,Q},D^{\prime}, D^{\prime\prime}\right).
\]
Then, by Theorem \ref{DDc}, 
we can say that 
on $A^{\ast}(M,{\mathcal O}_{\rho}\otimes \C)^{P,Q}$,
\[{\rm im}D^{\prime}\cap {\rm ker}D^{\prime\prime}={\rm ker}D^{\prime}\cap {\rm im}D^{\prime\prime}={\rm im}D^{\prime}D^{\prime\prime}.
\]
and  there exist a  $T$-equivariant linear map $F^{\prime}_{g}: {\rm im}D^{\prime}\cap {\rm ker}D^{\prime\prime}\to A^{\ast}(M,{\mathcal O}_{\rho}\otimes \C)^{P-1,Q-1}$  so that $\alpha=D^{\prime}D^{\prime\prime}F^{\prime}_{g}\alpha$ for $\alpha\in {\rm im}D^{\prime}\cap {\rm ker}D^{\prime\prime}$.
By using this, we construct the DGAs ${\mathcal N}^{\ast}(n)=\bigwedge (\bigoplus_{1\le P+Q\le n} \V^{P,Q})$ generated by  elements of degree $1$ and the homomorphisms $\varphi_{n}:{\mathcal N}^{\ast}(n)\to {\rm ker}D^{\prime}\subset A^{\ast}(M,{\mathcal O}_{\rho}\otimes \C)$ by the  inductive way.
More precisely, we will construct the components $\V^{P,Q}$ and the maps $d: \V^{P,Q}\to  \bigoplus_{S+U=P,\\ T+V=Q} \V^{S,T}\wedge \V^{U,V}$ by induction on $P+Q$ in the following.

\begin{itemize}
\item
For $P,Q\in \Z$ with $P+Q=1$, let
$\V^{P,Q}={\rm ker}D^{\prime}\cap {\rm ker}D^{\prime\prime}\cap A^{\ast}(M,{\mathcal O}_{\rho}\otimes \C)^{P,Q}$.
Define the homomorphism
$\varphi_{1}:\bigwedge (\bigoplus_{P+Q=1} \V^{P,Q})\to A^{\ast}(M,{\mathcal O}_{\rho}\otimes \C)$
so that on $\V^{P,Q}$, $\varphi_{1} $   is the natural inclusion $\V^{P,Q}\hookrightarrow A^{1}(M,{\mathcal O}_{\rho}\otimes\C)$.
\item For the quotient map $q:{\rm ker}D^{\prime}\cap {\rm ker}D^{\prime\prime}\to H^{\ast}({\rm ker}D^{\prime})$,  for $P,Q\in \Z$ with $P+Q=2$, define
\[\V^{P,Q}={\rm ker} \left(q \circ \varphi_{1} : \bigoplus_{S+U=P,\\ T+V=Q} \V^{S,T}\wedge \V^{U,V}\to H^{2}({\rm ker}D^{\prime})\right).
\]
Define the DGA ${\mathcal N}^{\ast}(2)=\bigwedge (\bigoplus_{1\le P+Q\le 2} \V^{P,Q})$ with the differential $d$ so that $d$ is $0$ on $\V^{P,Q}$ for $P+Q=1$ and $d$ on $\V^{P,Q}$ is the natural inclusion 
$\V^{P,Q}\hookrightarrow   \bigoplus_{\begin{array}{cc}S+U=P,\,\,T+V=Q\end{array}} \V^{S,T}\wedge \V^{U,V}$ for $P+Q=2$.
Define the homomorphism $\varphi_{2}:{\mathcal M}^{\ast}(2)\to {\rm ker}D^{\prime}$ which is an extension of $\varphi_{1}$ so that $\varphi_{2}(v)=D^{\prime}F^{\prime}_{g}(\varphi_{1}(dv))$ for $v\in \V^{P,Q}$ with $P+Q=2$.

\item For $n\ge 2$, consider the DGA ${\mathcal N}^{\ast}(n)=\bigwedge (\bigoplus_{1\le P+Q\le n} \V^{P,Q})$ with the homomorphism $\varphi_{n}:{\mathcal N}^{\ast}(n)\to A^{\ast}(M,{\mathcal O}_{\rho})$ we have constructed.
We can say that $\varphi_{n}(v)\in {\rm ker}D\cap {\rm ker}D^{c}$ for $v\in \V^{P,Q}$ with $P+Q=1$ and as an inductive hypothesis we assume $\varphi_{n}(v)\in {\rm im} D^{\prime}$ for $v\in \V^{P,Q}$ with $P+Q\ge 2$.

For $P+Q=n+1$, let
\[\V^{P,Q}={\rm ker} d_{ \vert \sum_{S+U=P,T+V=Q} \V^{S,T}\wedge \V^{U,V}}.\]
 Define the extended DGA ${\mathcal N}^{\ast}(n+1)={\mathcal N}^{\ast}(n)\otimes \bigwedge(\bigoplus_{P+Q=n+1} \V^{P,Q})$
so that the differential $d$ is defined on $\V^{P,Q}$ with $P+Q=n+1$ as the natural inclusion $\V^{P,Q}\hookrightarrow \sum_{S+U=P,T+V=Q} \V^{S,T}\wedge \V^{U,V}$.
The homomorphism $\varphi_{n+1}:{\mathcal N}^{\ast}(n+1)\to A^{\ast}(M,{\mathcal O}_{\rho})$ is defined by 
$\varphi_{n+1}(v)=D^{\prime}F^{\prime}_{g}(\varphi_{n}(dv))$ for $v\in \V^{P,Q}$ with $P+Q=n+1$.

\end{itemize}

Let $\varinjlim {\mathcal N}^{\ast}(n)={\mathcal N}^{\ast}$ and $\varphi=\varinjlim \varphi_{n}:{\mathcal N}^{\ast}\to {\rm ker}D^{\prime\prime}\subset A^{\ast}(M,{\mathcal O}_{\rho}\otimes \C)$.

By the construction, 
each $\bigoplus_{P+Q=k} \V^{P,Q}$ is a $T$-module so that the map $\varphi: {\mathcal N}^{\ast}\to  A^{\ast}(M,{\mathcal O}_{\rho}\otimes \C)$ is $T$-equivariant.
Our construction is in fact the construction of a $1$-minimal model of ${\rm ker}D^{\prime\prime}$.
Since the inclusion $ {\rm ker}D^{\prime\prime}\subset A^{\ast}(M,{\mathcal O}_{\rho}\otimes \C)$ induces a cohomology isomorphism,
we can say that 
the map $\varphi: {\mathcal N}^{\ast}\to  A^{\ast}(M,{\mathcal O}_{\rho}\otimes \C)$ induces isomorphisms on $0$-th and first cohomology and an injection on second cohomology.

Define the multiplicative bigrading  $({\mathcal N}^{\ast})^{P,Q}$ so that elements in $\V^{P,Q}$ are of type   $(P,Q)$.
By the construction,   we have $d({\mathcal N}^{\ast})^{P,Q}\subset ({\mathcal N}^{\ast})^{P,Q}$.
 Define the increasing filtration $W_{\ast}$ so that $W_{n}({\mathcal N}^{\ast})=\bigoplus_{P+Q\le n}({\mathcal N}^{\ast})^{P,Q}$.

\begin{remark}\label{sp22}
Let $\W_{n}=\bigoplus_{P+Q=n}\V^{P,Q}$ for each $n$.
Then we have $\V_{1}\otimes \C=\W_{1}$
By the isomorphism $H^{\ast}(M,{\mathcal O}_{\rho})\cong H^{\ast}({\rm ker}D^{\prime}) $ as in  Theorem \ref{Forma}, the map 
$q\circ \varphi_{1} : \bigwedge^{2} \W_{1}\to H^{2}({\rm ker}D^{c})$ is identified with the cup product 
\[\bigwedge^{2} H^{1}(M,{\mathcal O}_{\rho})\to H^{2}(M,{\mathcal O}_{\rho}).\]
Thus, forgetting the homomorphism $\varphi_{n}$, each $\W_{n}$ is constructed  in the same manner as $\V_{n}$.
By this we obtain the isomorphism ${\mathcal I}_{sp}:{\mathcal M}\otimes \C\to   {\mathcal N}^{\ast}$ so that  ${\mathcal I}_{sp}$ is the identity map on $\V_{1}\otimes \C=\W_{1}$ and 
 ${\mathcal I}_{sp}(\V_{n}\otimes \C)=\W_{n}$ for each $n$.
But, this isomorphism ${\mathcal I}_{sp}$ is not involved  with the homomorphisms $\phi$ and $\varphi$ but $T$-equivariant.

Via $\V_{1}\otimes \C=\W_{1}$, $\W_{1}=\bigoplus_{P+Q=1}\V^{P,Q}$ is in fact the bigrading of the $\R$-Hodge structure on  $\V_{1}\cong H^{1}(M,{\mathcal O}_{\rho})$.
We can easily check inductively that via  ${\mathcal I}_{sp}(\V_{n}\otimes \C)=\W_{n}$, $\W_{n}=\bigoplus_{P+Q=n}\V^{P,Q}$ is considered as the bigrading of the $\R$-Hodge structure on $\V_{n}$ as in Remark \ref{sp11}.

By the construction, we have $\V_{n}\subset \sum_{i+j=n}\V_{i}\wedge \V_{j}$.
By this, we can inductively say that $\V_{n}$ is a $T$-submodule of a direct sum of copies of the $n$-th tensor power $\V_{1}^{\otimes n}$ of $\V_{1}$ and  the $\R$-Hodge structure on $\V_{n}$ is an $\R$-Hodge substructure of a direct sum of $\V_{1}^{\otimes n}$.
As we saw in Section \ref{HOK}, the  $\R$-Hodge structure on  $\V_{1}=\bigoplus {\mathcal H}^{1}(M,E_{\alpha}^{\ast})\otimes V_{\alpha}$ is polarizable.
Thus the $\R$-Hodge structure on each  $\V_{n}$ is polarizable (see \cite[Corollary 2.12]{PS}). 
Since the co-module structure $\V_{1}\to \V_{1}\otimes \R[T]$ corresponding to the $T$-module structure is a morphism of $\R$-Hodge structures, 
inductively we can say that each co-module structure $\V_{n}\to \V_{n}\otimes \R[T] $ is also a morphism of $\R$-Hodge structures.
In particular,  the co-module structure ${\mathcal N}^{\ast}\to {\mathcal N}^{\ast}\otimes \C[T]$ preserves the bigradings on $\C[T]$ and ${\mathcal N}^{\ast}$.

\end{remark}

\subsection{Morgan's mixed Hodge structures and the split mixed Hodge structure}\label{Mormix}

By the Hodge decomposition as in Section \ref{HOK},
the $r$-th cohomology  $H^{r}(M,{\mathcal O}_{\rho})$ admits a Hodge structure of weight $r$ which induced by the Hodge filtration $F^{\ast}$ on $A^{\ast}(M,{\mathcal O}_{\rho}\otimes \C)$ so that 
\[F^{p}(A^{\ast}(M,{\mathcal O}_{\rho}\otimes \C))=\bigoplus_{P\ge p} A^{\ast}(M,{\mathcal O}_{\rho}\otimes \C)^{P,Q}.\]
Define  the increasing filtration $W_{\ast}$  on $A^{\ast}(M,{\mathcal O}_{\rho})$ so that $W_{-1}(A^{\ast}(M,{\mathcal O}_{\rho}))=0$ and $W_{0}(A^{\ast}(M,{\mathcal O}_{\rho}))=A^{\ast}(M,{\mathcal O}_{\rho})$.
Then we can say that  
${\rm id}: (A^{\ast}(M,{\mathcal O}_{\rho}), W)_{\C}\to   (A^{\ast}(M,{\mathcal O}_{\rho}\otimes \C), W_{\ast},  F^{\ast})$ is a  mixed Hodge diagram over $\R$ in the sense of Morgan (\cite{Mor}).
In \cite{Mor}, Morgan proves that there exists an $\R$-mixed Hodge structure (Morgan's mixed Hodge structure)  on the $1$-minimal model of a mixed Hodge diagram over $\R$ by the way in Remark \ref{MIXMO4444}.
Hence the $1$-minimal model ${\mathcal M}^{\ast}$ admits an  $\R$-mixed Hodge structure.
We should remark that  Morgan's mixed Hodge structure is not unique in general.
We  give the simplest  $\R$-mixed Hodge structure on ${\mathcal M}^{\ast}$.

\begin{example}
We put the $\R$-Hodge structure of weight $n$ on $\V_{n}$ for each $n$ as in Remark \ref{sp11} (see also Remark \ref{sp22}).
Then, we obtain a split $\R$-mixed Hodge structure $(W_{\ast},F^{\ast}_{sp})$ on ${\mathcal M}^{\ast}$ where $W_{\ast}$ is the increasing filtration defined in Section \ref{dedec}.

By Remark \ref{sp22}, we have $F^{r}_{sp}({\mathcal M}^{\ast}_{\C})={\mathcal I}^{-1}_{sp}(\bigoplus_{P\ge r} ({\mathcal N}^{\ast})^{P,Q})$.
\end{example}
But this $\R$-mixed Hodge structure is not interesting for studying the geometry of $M$ since  this structure is completely determined by the cup product on $H^{1}(M,{\mathcal O}_{\rho})$.

\subsection{The canonical isomorphism ${\mathcal I}:{\mathcal M}^{\ast}_{\C}\to {\mathcal N}^{\ast}$  with the homotopy $H:{\mathcal M}^{\ast}_{\C}\to A^{\ast}(M,{\mathcal O}_{\rho}\otimes \C)\otimes [t,dt]$}\label{isho}


The purpose of this subsection is to construct 
a $T$-equivariant isomorphism 
${\mathcal I}:{\mathcal M}^{\ast}_{\C}\to {\mathcal N}^{\ast}$ which is compatible with the increasing filtrations 
\[W_{k}({\mathcal M}^{\ast})=\bigoplus_{i\le k} {\mathcal M}^{\ast}_{i}\,\,\,\,\,\,\, {\rm and}\,\,\,\,\,\,\,  W_{k}({\mathcal N}^{\ast})=\bigoplus_{P+Q\le k} ({\mathcal N}^{\ast})^{P,Q}\]
  and 
 $T$-equivariant homotopy $H:{\mathcal M}^{\ast}_{\C}\to A^{\ast}(M,{\mathcal O}_{\rho}\otimes \C)\otimes [t,dt]$ from  $\varphi \circ {\mathcal I}$ to $\phi$  
 which are canonically determined by a K\"ahler metric $g$
 where $[t,dt]$ means the tensor product of polynomials on $t$ with the exterior algebra of $dt$ (see e.g.  \cite[Chapter 11]{GMo}).

\smallskip

For the construction, we will  trace the arguments in   \cite[(5.5)$\sim$(5.9), (7.3)$\sim$(7.5)]{Mor}.
We show by induction.
We will use the homotopy theory of DGAs as in \cite[Section 5]{Mor} and \cite[Chapter 11]{GMo}.
We assume that we have a $T$-equivariant isomorphism   ${\mathcal I}:{\mathcal M}^{\ast}_{\C}(n)\to  {\mathcal N}^{\ast}(n)$ so that there exists a homotopy $H :{\mathcal M}^{\ast}_{\C}(n)\to  A^{\ast}(M,{\mathcal O}_{\rho}\otimes \C)\otimes[t,dt]$ from  $\varphi_{n} \circ {\mathcal I}$ to $\phi_{n}$.
For $v\in \V_{n+1}$,  we have
\[\phi_{n}(dv)-\varphi_{n}({\mathcal I}(dv))=D\int^{1}_{0}H(dv)
\]
by applying  \cite[Proposition 5.5]{Mor} to $dv$.
By the construction of $\phi_{n+1}$,  we have $\phi_{n}(dv)=d\phi_{n+1}(v)$ and so
 $[\varphi_{n}({\mathcal I}(dv))]=0$ in $H^{2}(A^{\ast}(M,{\mathcal O}_{\rho}\otimes \C))$.
By the construction in Section \ref{dbd}, we have ${\mathcal I}(dv)\in d\left(\bigoplus_{P+Q\le n+1}{\mathcal V}^{P,Q}\right)$.
By ${\rm ker}\, d_{\vert \bigoplus_{P+Q\le n+1} {\mathcal V}^{P,Q}}=\bigoplus_{P+Q=1}{\mathcal V}^{P,Q}$,  we can take a  unique $a(v)\in  \bigoplus_{2\le P+Q\le n+1}{\mathcal V}^{P,Q}$ such that $da(v)={\mathcal I}(dv)$.
Consider 
\[\phi_{n+1}(v)-\varphi_{n+1}(a(v))-\int^{1}_{0}H(dv).
\]
Then it is closed  and so we have a unique $a^{\prime}(v)\in \bigoplus_{P+Q=1} \V^{P,Q}$ such that 
\[[\varphi_{n+1}(a^{\prime}(v))]=\left[\phi_{n+1}(v)-\varphi_{n+1}(a(v))-\int^{1}_{0}H(dv)\right] 
\]
in $H^{1}(A^{\ast}(M,{\mathcal O}_{\rho}\otimes \C))$.
We define the  extended DGA homomorphism  ${\mathcal I}:{\mathcal M}^{\ast}_{\C}(n+1)={\mathcal M}^{\ast}_{\C}(n)\otimes \bigwedge \V_{n+1} \to  {\mathcal N}^{\ast}(n+1)$ so that 
${\mathcal I}(v)=a(v)+a^{\prime}(v)$ for $v\in \V_{n+1}$.
We take 
\[b(v)=D^{\ast}G_{D}\left(\phi_{n+1}(v)-\varphi_{n+1}\left({\mathcal I}(v)\right)-\int^{1}_{0}H(dv)\right).
\]
where $D^{\ast}$ is the adjoint of the differential operator $D$ and $G_{D}$ is the Green operator.
Then we have $Db(v)=\phi_{n+1}(v)-\varphi_{n+1}\left({\mathcal I}(v)\right)-\int^{1}_{0}H(dv)$.
We define the  extended DGA homomorphism $H:{\mathcal M}^{\ast}_{\C}(n+1)={\mathcal M}^{\ast}_{\C}(n)\otimes \bigwedge \V_{n+1} \to A^{\ast}(M,{\mathcal O}_{\rho}\otimes \C)\otimes [t,dt]$
so that for $v\in \V_{n+1}$
\[H(v)=\varphi_{n+1}(\mathcal I(v))+\int^{t}_{0}H(dv)+D(b(v)t).
\]
Thus, we obtain   ${\mathcal I}:{\mathcal M}^{\ast}_{\C}(n+1)\to  {\mathcal N}^{\ast}(n+1)$ so that there exists a homotopy $H :{\mathcal M}^{\ast}_{\C}(n+1)\to  A^{\ast}(M,{\mathcal O}_{\rho}\otimes \C)\otimes[t,dt]$ from  $\varphi_{n+1} \circ {\mathcal I}$ to $\phi_{n+1}$.
We can easily  check that  extensions  $\mathcal I$ and $H$ are $T$-equivariant inductively.
In fact, since the involved maps $d$, $\phi$, $\varphi$, $D^{\ast}G_{D}$, etc. are $T$-equivariant, 
by the induction hypothesis, the above $a(v)$, $a^{\prime}(v)$ and $b(v)$ commute with the $T$-action and hence we san say the $T$-equivariance of the extensions  $\mathcal I$ and $H$.

For each step, $b(v)$  depends on the choice of a K\"ahler metric.
Thus, our construction of $I$ and $H$ depends on the choice of a K\"ahler metric.
\begin{remark}\label{volvol}
Let $d\mu_{g}$ be the volume form associated with the K\"ahler metric $g$.
Define 
\[C_{g}=\left\{\Psi\in A^{0}(M)\otimes \C:  \int \Psi d\mu_{g}=0  \right\}. 
\]
Then $D^{\ast}G_{D}: D( A^{0}(M,{\mathcal O}_{\rho}\otimes \C))\to A^{0}(M,{\mathcal O}_{\rho}\otimes \C)$ is in fact the inverse map of
\[D: C_{g}\oplus \bigoplus_{\alpha\not=1}A^{0}(M,{\bf E}^{\ast}_{\alpha}\otimes \C)\otimes V_{\alpha}\otimes \C\to D( A^{0}(M,{\mathcal O}_{\rho}\otimes \C)).
\]
We notice that $C_{g}$ is not closed under the multiplication.
Thus our construction may be  different from the construction by the reduced bar construction of an augmented multiplicative mixed Hodge complex as in \cite{HainI}.


\end{remark}

\begin{remark}\label{nonsp}
It is not obvious that we can take ${\mathcal I}(\V_{n}\otimes \C)\subset \bigoplus_{P+Q=n} \V^{P,Q}$ for $k\ge 3$.
For $n=1$, by the constructions, we can take ${\mathcal I}:\V_{1}\otimes \C\cong \bigoplus_{P+Q=1} \V^{P,Q}$  so that  $\phi_{1}=\varphi_{1} \circ {\mathcal I}$.
For $n=2$, we can take ${\mathcal I}:\V_{2}\otimes \C\cong \bigoplus_{P+Q=2} \V^{P,Q}$ so that 
the homotopy $H :{\mathcal M}^{\ast}_{\C}(n)\to  A^{\ast}(M,{\mathcal O}_{\rho}\otimes \C)\otimes [t,dt]$ is given by
\[H(v)=-2\sqrt{-1}D^{\prime}\alpha+\sqrt{-1}D\alpha t+\sqrt{-1}\alpha dt
\]
 for $v\in \V_{2}$ satisfying $\phi(v)=D^{c}\alpha$.
\end{remark}

\subsection{The main mixed Hodge structure}\label{MIXMIX}
We obtain the canonical $\R$-mixed Hodge structure on ${\mathcal M}^{\ast}$.
\begin{theorem}\label{Morg}
For the  isomorphism ${\mathcal I}:{\mathcal M}^{\ast}_{\C}\to {\mathcal N}^{\ast}$ as in the last subsection, taking the filtration $F^{r}({\mathcal M}^{\ast}_{\C})={\mathcal I}^{-1}(\bigoplus_{P\ge r} ({\mathcal N}^{\ast})^{P,Q})$, $({\mathcal M}^{\ast}, W_{\ast},F^{\ast})$ is an $\R$-mixed Hodge structure which is compatible with the differential and the multiplication.
The co-module structure ${\mathcal M}^{\ast}\to \R[T]\otimes {\mathcal M}^{\ast}$ corresponding to the $T$-module structure on ${\mathcal M}^{\ast}$ is a morphism of $\R$-mixed Hodge structures.

\end{theorem}
\begin{proof}
Since each $W_{n}({\mathcal M}^{\ast}) $ is a $T$-submodule, the $T$ action preserves the filtration $W_{\ast}$.
Since the co-module structure ${\mathcal N}^{\ast}\to {\mathcal N}^{\ast}\otimes \C[T]$ preserves the bigradings on $\C[T]$ and ${\mathcal N}^{\ast}$, by the $T$-equivariance of ${\mathcal I}:{\mathcal M}^{\ast}_{\C}\to {\mathcal N}^{\ast}$, the co-module structure ${\mathcal M}^{\ast}_{\C}\to {\mathcal M}^{\ast}_{\C}\otimes  \C[T]$ preserves the filtrations $F^{\ast}$ on 
${\mathcal M}^{\ast}_{\C}$  and  $\C[T]$.
Thus the second assertion follows when $({\mathcal M}^{\ast}, W_{\ast},F^{\ast})$ is an $\R$-mixed Hodge structure.

It is sufficient to show that for any $n$, $(\V_{1}\oplus \dots \oplus \V_{n}, W_{\ast},F^{\ast})$  is an $\R$-mixed Hodge structure.

Consider the functor $Gr^{W}$ from the category of filtered DGAs  to the category of DGAs.
Then by the definitions
 \[W_{k}({\mathcal M}^{\ast})=\bigoplus_{i\le k} {\mathcal M}^{\ast}_{i}\,\,\,\,\,\,\, {\rm and}\,\,\,\,\,\,\,  W_{k}({\mathcal N}^{\ast})=\bigoplus_{P+Q\le k} ({\mathcal N}^{\ast})^{P,Q},\]
we have $Gr^{W}({\mathcal M}^{\ast})={\mathcal M}^{\ast}$ and $Gr^{W}({\mathcal N}^{\ast})={\mathcal N}^{\ast}$.
We consider the map ${\mathcal I}^{\prime}:{\mathcal M}^{\ast}_{\C}\to {\mathcal N}^{\ast}$ which  corresponds to $Gr^{W}({\mathcal I}):Gr^{W}({\mathcal M}^{\ast}_{ \C}) \to  Gr^{W}({\mathcal N}^{\ast})$.
We will show ${\mathcal I}^{\prime}={\mathcal I}_{sp}$.

We prove this equality inductively on each ${\mathcal I}^{\prime}: {\mathcal M}_{\C}^{\ast}(n)\to {\mathcal N}^{\ast}(n)$.
On $n=1$, it is obvious.
We suppose this on $n$.
Then, by the construction of ${\mathcal I}$, for $v\in \V_{n+1}$,
we have $d{\mathcal I}^{\prime}(v)={\mathcal I}_{sp}(dv)$.
For $n\ge 1$, by the construction of ${\mathcal N}^{\ast}$,
on $\W_{n}=\bigoplus_{P+Q=n}\V^{P,Q}$ the differential operator $d$ is injective.
Thus, we have ${\mathcal I}^{\prime}(v)={\mathcal I}_{sp}(v)$.

This implies that for each $k$, on  $Gr_{k}^{W}(\V_{1}\oplus\dots \oplus \V_{n})=\V_{k}$ the filtration $F^{\ast}$ induces the Hodge structure of weight $k$ as in Remark \ref{sp11}. 
Hence the theorem follows.
\end{proof}

We say that  an $\R$-mixed Hodge structure is {\em graded-polarizable} if the $\R$-Hodge structure on each $k$-th graded quotient $Gr^{W}_{k}$ is polarizable.
We can easily show the following statement.
\begin{proposition}
The $\R$-mixed Hodge structure on ${\mathcal M}^{\ast}$ as in Theorem \ref{Morg} is  graded-polarizable.

\end{proposition}
\begin{proof}
It is sufficient to show that for any $n$, $(\V_{1}\oplus \dots \oplus \V_{n}, W_{\ast},F^{\ast})$  is graded-polarizable.
By the above proof, for each $k$, on  $Gr_{k}^{W}(\V_{1}\oplus\dots \oplus \V_{n})=\V_{k}$ the filtration $F^{\ast}$ induces the $\R$-Hodge structure of weight $k$ as in Remark \ref{sp11}. 
Such $\R$-Hodge structure is polarizable. 
Hence the proposition follows.
\end{proof}

\begin{remark}\label{remnoss}
By the proof of Theorem \ref{MIXMIX},  for each $n$, we can say that 
 the   linear map ${\mathcal I}-{\mathcal I}_{sp}: W_{n}({\mathcal M}^{\ast}\otimes \C) \to W_{n-1}({\mathcal N}^{\ast}\otimes \C) $ is an obstruction to splitting of the $\R$-mixed Hodge structure $(W_{\ast},F^{\ast})$ on ${\mathcal M}^{\ast}$.

It is not obvious that the direct sum $\V_{1}\oplus \dots \oplus \V_{n}$ is a splitting of $\R$-mixed Hodge structures for $n\ge 3$ since the map $\mathcal I$ does not satisfy  ${\mathcal I}(\V_{n})\subset \bigoplus_{P+Q=n}\V_{P,Q}$.

\end{remark}


\subsection{Pro-nilpotent Lie algebra ${\frak u}$}\label{proni}

We consider the  pro-nilpotent Lie algebra ${\frak u}$ which is dual to the $1$-minimal model ${\mathcal M}^{\ast}$.
Denote by $\V^{\ast}_{i}$ the dual space of $\V_{i}$.
Then, we have ${\frak u}=\bigoplus \V_{i}^{\ast}$ and  the Lie bracket is  dual to the differential $d: \V_{k}\to \sum_{i+j=k}\V_{i}\wedge \V_{j}$.
Hence, ${\frak u}=\bigoplus \V^{\ast}_{i}$ is a graded pro-nilpotent Lie algebra so that ${\frak u}_{k}=\bigoplus_{i\ge k} \V_{i}^{\ast}$ is the $k$-th term of the lower central series of ${\frak u}$.
By ${\mathcal M}^{1}=\bigoplus \V_{i}$ and the definition of the differential $d$ on ${\mathcal M}^{1}$ as in Section \ref{dedec}, the Lie bracket on $  {\frak u}$ is the dual to the differential $d: {\mathcal M}^{1}\to {\mathcal M}^{2}={\mathcal M}^{1}\wedge {\mathcal M}^{1}$.
By the $\R$-mixed Hodge structure on ${\mathcal M}^{1}$ as in Theorem \ref{Morg},
we obtain the $\R$-mixed Hodge structure $(W_{\ast}, F^{\ast})$ on ${\frak u}$. 
 Since the differential $d: {\mathcal M}^{1}\to {\mathcal M}^{1}\wedge {\mathcal M}^{1}$ is a morphism of $\R$-mixed Hodge structures by Theorem \ref{Morg}, the $\R$-mixed Hodge structure  on ${\frak u}$ is compatible with its Lie bracket.
We have $W_{-k}({\frak u})=\bigoplus_{i\ge k} \V_{i}^{\ast}={\frak u}_{k}$ and so this filtration is the natural filtration which is given by the lower central series of ${\frak u}$.
In particular, we have ${\frak u}/[\frak u,\frak u]\cong \V^{\ast}_{1}$.
We recall that $\V_{1}$ is identified with the space of harmonic forms in $A^{\ast}(M, {\mathcal O}_{\rho})$ and hence it is isomorphic to  $H^{1}(M,  {\mathcal O}_{\rho})$.
As in Remark \ref{remnoss},  the decomposition 
${\frak u}=\bigoplus \V^{\ast}_{i}$ does not give a splitting of the $\R$-mixed Hodge structure  and  the $\R$-mixed Hodge structure on ${\frak u}$ is not $\R$-split.


\subsection{Functoriality}\label{Funcmx}

Let $(M_{1},g_{1})$ and $(M_{2},g_{2})$ be compact K\"ahler manifolds.
Let $f:M_{2}\to M_{1}$ be a holomorphic map.
For an $\R$-VHS $({\bf E}, {\bf F}^{\ast})$  over $M_{1}$ with a polarization $\bf S$, we have the pull-back $\R$-VHS $(f^{\ast}{\bf E}, f^{\ast}{\bf F}^{\ast})$ over $M_{2}$ with the polarization $f^{\ast}\bf S$.
We have the homomorphism $f^{\ast}: A^{\ast}(M_{1},{\bf E})\to A^{\ast}(M_{2}, f^{\ast}{\bf E})$ which is compatible with bigradings and  commutes 
with all differentials  as in  Section \ref{HOK}.
Moreover, by Remark \ref{ddcinvvv}, we have
\[f^{\ast}\circ  D^{c}F_{g_{1}}=D^{c}F_{g_{2}}\circ f^{\ast}  \qquad  {\rm on} \qquad  {\rm im}D\cap {\rm ker}D^{c}\cap A^{2}(M_{1},{\bf E})\] and 
 \[f^{\ast}\circ  D^{\prime}F^{\prime}_{g_{1}}=D^{\prime}F^{\prime}_{g_{2}}\circ f^{\ast} \qquad   {\rm  on} \qquad  {\rm im}D^{\prime}\cap {\rm ker}D^{\prime\prime}\cap A^{2}(M_{1},{\bf E}_{\C})^{P,Q}.\]

Let $\rho:\pi_{1}(M_{1},f(x))\to GL(V_{0})$ be a real valued representation for $x\in M_{2}$.
Consider the real local system ${\bf E}_{0}=(\tilde{M}_{1} \times V_{0})/\pi_{1}(M_{1},f(x))$ where $\tilde{M}_{1}$ is the universal covering of $M_{1}$.
We assume that $\bf E_{0}$ admits an $\R$-VHS $(\bf E_{0}, \bf F^{\ast})$ over $M_{1}$ of weight $N_{0}$ with a polarization $\bf S$.
Consider the bilinear form ${\bf S}_{f(x)}:V_{0}\times V_{0}\to \R$.
Then we have $\rho(\pi_{1}(M,x))\subset T={\rm Aut}(V_{0},{\bf S}_{f(x)})$.
For the induced map $f_{\ast}:\pi_{1}(M_{2},x)\to \pi_{1}(M_{1},f(x))$ we assume that  $\rho(f_{\ast}(\pi_{1}(M_{2},x)))$ is Zariski-dense in $T$.
Let $\rho^{\prime}=\rho\circ f_{\ast}$
In this assumption, we can apply the constructions in this section to $ A^{\ast}(M_{1},{\mathcal O}_{\rho})$ and $ A^{\ast}(M_{2},{\mathcal O}_{\rho^{\prime}})$.
Moreover, we have the $T$-equivariant homomorphism $f^{\ast}: A^{\ast}(M_{1},{\mathcal O}_{\rho})\to A^{\ast}(M_{2},{\mathcal O}_{\rho^{\prime}})$
which is compatible with bigradings and  commutes 
with all differentials.

Take the canonical $1$-minimal models $\,_{1}\phi:\,_{1}{\mathcal M}^{\ast}\to A^{\ast}(M_{1}, {\mathcal O}_{\rho})$ and $\,_{2}\phi:\,_{2}{\mathcal M}^{\ast}\to A^{\ast}(M_{2},{\mathcal O}_{\rho^{\prime}})$ as in Subsection \ref{dedec}
and $\,_{1}\varphi:\,_{1}{\mathcal N}^{\ast}\to A^{\ast}(M_{1},{\mathcal O}_{\rho}\otimes \C)$ and $\,_{2}\varphi:\,_{2}{\mathcal N}^{\ast}\to A^{\ast}(M_{2},{\mathcal O}_{\rho^{\prime}}\otimes \C)$ as in Subsection \ref{dbd} (We use the notation $\,_{1}{\mathcal M}^{\ast}=\bigwedge (\,_{1}\V_{1}\oplus\,_{1}\V_{2}\dots)$ etc.).
By the property of $f^{\ast}$, the map $f^{\ast}$ can be restricted as $f^{\ast}: {\rm ker} D\cap {\rm ker} D^{c}\cap A^{\ast}(M_{1}, {\mathcal O}_{\rho})\to {\rm ker} D\cap {\rm ker} D^{c}\cap A^{\ast}(M_{2}, {\mathcal O}_{\rho^{\prime}})$.
Thus we get the map $\,_{1}\V_{1}\to \,_{2}\V_{1}$.
Since the map $d:\,_{l}\V_{n}\to \sum_{i+j=n} \,_{l}\V_{i}\wedge \,_{l}\V_{j}$ is injective for $l=1, 2$, we can extend this map to  the $T$-equivariant DGA map 
$f_{\mathcal M}:\,_{1}{\mathcal M}^{\ast}(n)\to \,_{2}{\mathcal M}^{\ast}(n)$ for each $n$.
By the commutativity between $f^{\ast}$ and the map $D^{c}F_{g}$ as above, we can easily check that we have $f^{\ast}\circ \,_{1}\phi=\,_{2}\phi \circ f_{\mathcal M}$.
By the similar way, we have the canonical  $T$-equivariant DGA map 
$f_{\mathcal N}:\,_{1}{\mathcal N}^{\ast}(n)\to \,_{2}{\mathcal N}^{\ast}(n)$ for each $n$  satisfying $f^{\ast}\circ \,_{1}\varphi=\,_{2}\varphi \circ f_{\mathcal N}$.
Moreover, since $f^{\ast}$   is compatible with the bigradings, the map $f_{\mathcal N}$ is compatible with the bigradings $(\,_{1}{\mathcal N})^{P,Q}$ and $(\,_{2}{\mathcal N})^{P,Q}$.

Take the isomorphisms $\,_{1}{\mathcal I}:\,_{1}{\mathcal M}^{\ast}_{\C}\to  \,_{1}{\mathcal N}^{\ast}$ and $\,_{2}{\mathcal I}:\,_{2}{\mathcal M}^{\ast}_{\C}\to  \,_{2}{\mathcal N}^{\ast}$   and 
the homotopies  $\,_{1}H:\,_{1}{\mathcal M}^{\ast}_{\C}\to  A^{\ast}(M_{1}, {\mathcal O}_{\rho}\otimes \C)\otimes [t,dt]$ and $\,_{2}H:\,_{2}{\mathcal M}^{\ast}_{\C}\to  A^{\ast}(M_{2}, {\mathcal O}_{\rho^{\prime}}\otimes \C)\otimes [t,dt]$  as in Subsection \ref{isho}.

Let $d\mu _{1}$ and  $d\mu _{2}$ be the volume forms associated with the K\"ahler metrics $g_{1}$ and $g_{2}$ respectively.
We consider the following condition on $f:M_{2}\to M_{1}$.
\begin{cond(V)}
 For any $\Psi\in A^{0}(M_{1})\otimes \C$ satisfying $\int_{M_{1}} \Psi d\mu_{1}=0$, we have
\[\int_{M_{2}} (f^{\ast}\Psi )d\mu_{2}=0.
\] 
\end{cond(V)}
Let 
\[C_{i}=\left\{\Psi\in A^{0}(M)\otimes \C:  \int_{M_{i}} \Psi d\mu_{i}=0  \right\}
\]
for $i=1, 2$.
Then the Condition (V) holds if and only if $f^{\ast}(C_{1})\subset C_{2}$.
By this, the Condition (V) implies the commutativity
\[f^{\ast}\circ D^{\ast}G_{D}=D^{\ast}G_{D}\circ f^{\ast}
\]
on $D(A^{0}(M_{1}, {\mathcal O}_{\rho}\otimes \C))$.
\begin{proposition}\label{Hcommu}
If $f$ satisfies the Condition (V), then
\[\,_{2}{\mathcal I}\circ f_{\mathcal M}=f_{\mathcal N}\circ \,_{1}{\mathcal I}\]
 and \[\,_{2}H\circ f_{\mathcal M}=(f^{\ast}\otimes {\rm id}_{[t,dt]})\circ \,_{1}H.\]
\end{proposition}
\begin{proof}
For each step $\,_{1}\phi_{n}:\,_{1}{\mathcal M}^{\ast}_{n}\to  A^{\ast}(M_{1}, {\mathcal O}_{\rho}\otimes \C)$, $\,_{1}\varphi_{n}:\,_{1}{\mathcal N}^{\ast}_{n}\to A^{\ast}(M_{1}, {\mathcal O}_{\rho^{\prime}}\otimes \C)$, $\,_{2}\phi_{n}:\,_{2}{\mathcal M}^{\ast}_{n}\to  A^{\ast}(M_{2}, {\mathcal O}_{\rho}\otimes \C)$, $\,_{2}\varphi_{n}:\,_{2}{\mathcal N}^{\ast}_{n}\to A^{\ast}(M_{2}, {\mathcal O}_{\rho^{\prime}}\otimes \C)$ (see Subsection  \ref{dedec}, \ref{dbd}),
we will prove the statement inductively.
For $n=1$, the claim is obvious.
Assuming  the claim for $n$, we prove for $n+1$.
It is sufficient to show that $f_{\mathcal N}(\,_{1}{\mathcal I}(v))=\,_{2}{\mathcal I}(v)( f_{\mathcal M}(v))$ and $\,_{2}H(f_{\mathcal M}( v))=(f^{\ast}\otimes {\rm id}_{[t,dt]})(\,_{1}H(v))$ for any $v\in \,_{1}\V_{n+1}$.
For $v\in \,_{1}\V_{n+1}$, the form
\[\,_{1}\phi_{n+1}(v)-\,_{1}\varphi_{n+1}(\,_{1}{\mathcal I}(v)))-\int^{1}_{0}\,_{1}H(dv)\]
is an exact form.
By the inductive assumption, 
we can say that $\,_{2}{\mathcal I}(df_{\mathcal M}( v))=d f_{\mathcal N} (\,_{1}{\mathcal I}(v))$ and 
\[\,_{2}\phi_{n+1}(f_{\mathcal M}(v))-\,_{2}\varphi_{n+1}(f_{\mathcal N}(\,_{1}{\mathcal I}(v)))-\int^{1}_{0}\,_{2}H(df_{\mathcal M}(v))\]
is an exact form.
This implies $f_{\mathcal N}(\,_{1}{\mathcal I}(v))=\,_{2}{\mathcal I}( f_{\mathcal M}(v))$ (see the construction of  Subsection \ref{isho}).
For
 \[b(f_{\mathcal M}( v))=D^{\ast}G_{D}\left(\,_{2}\phi_{n+1}(f_{\mathcal M}(v))-\,_{2}\varphi_{n+1}(f_{\mathcal N}(\,_{1}{\mathcal I}(v)))-\int^{1}_{0}\,_{2}H(df_{\mathcal M}(v))\right),
\]
 we have 
\[\,_{2}H(f_{\mathcal M}( v))=\,_{2}\varphi_{n+1}(\,_{2}{\mathcal I}(f_{\mathcal M}( v)))+\int^{t}_{0}\,_{2}H(df_{\mathcal M}(v))+D(b(f_{\mathcal M}( v))t).
\]
By the commutativity between $f^{\ast}$ and the map $D^{\ast}G_{D}$ as above,   the inductive assumption and  ${\mathcal I}_{2}(f_{\mathcal M}( v))=f_{\mathcal N}({\mathcal I}_{1}(v))$,  we can say that $b(f_{\mathcal M}( v))=f^{\ast}(b(v))$ and 
\begin{multline*}
\,_{2}H(f_{\mathcal M}( v))\\
=f^{\ast}(\,_{1}\varphi_{n+1}({\mathcal I}_{1}( v))+(f^{\ast}\otimes {\rm id}_{[t,dt]})\int^{t}_{0}\,_{1}H(dv)+(f^{\ast}\otimes {\rm id}_{[t,dt]})D(b( v)t)\\
=(f^{\ast}\otimes {\rm id}_{[t,dt]})(\,_{1}H(v)).
\end{multline*}
Hence the proposition follows.
\end{proof}
Since the map $f_{\mathcal N}$ is compatible with the bigradings, we have the following result.
\begin{corollary}\label{mixcom}
If $f$ satisfies the condition (V), then
$f_{\mathcal M}:\,_{1}{\mathcal M}^{\ast}\to \,_{2}{\mathcal M}^{\ast}$  is a morphism of $\R$-mixed Hodge structures.
\end{corollary}

Since $C_{i}=\Delta_{D} (A^{0}(M_{i})\otimes \C)$,
if $f^{\ast}$ commutes with the laplacian $\Delta_{D}$ on the $0$-forms $A^{0}(M_{1})\otimes \C$, then the condition (V) holds.
It is known that a differential map $f$ between compact Riemannian manifolds satisfies the commutativity between the pull-back $f^{\ast}$ and the laplacian  on the $0$-forms if and only if $f$ is a harmonic Riemannian submersion (see \cite{Wat}).
Since a holomorphic map between K\"ahler manifolds is harmonic, a K\"ahler submersion (i.e. holomorphic Riemannian submersion between K\"ahler manifolds) $f:M_{2}\to M_{1}$ satisfies the condition (V).




\begin{remark}
In \cite{Mor}, Morgan did not obtain the functoriality on his mixed Hodge structures. (see \cite{Morco}).
\end{remark}

\subsection{Constructions for alternate volume forms}\label{altV}
Let  $d\mu$ be a volume form which is different from the volume form $d\mu_{g}$ on $M$ associated with a K\"ahler metric $g$.
Let 
\[C_{\mu}=\left\{\Psi\in A^{0}(M)\otimes \C:  \int_{M} \Psi d\mu=0  \right\}. 
\]
Define $\delta_{\mu}$ as the inverse map of 
\[D: C_{\mu}\oplus \bigoplus_{\alpha\not=1}A^{0}(M,{\bf E}^{\ast}_{\alpha}\otimes \C)\otimes V_{\alpha}\otimes \C\to D( A^{0}(M,{\mathcal O}_{\rho}\otimes \C)).
\]
Then we rewrite the constructions of $\mathcal I$ and $H$ in Subsection \ref{isho}.
In each inductive step, we only replace 
\[b(v)=D^{\ast}G_{D}\left(\phi_{n+1}(v)-\varphi_{n+1}\left({\mathcal I}(v)\right)-\int^{1}_{0}H(dv)\right)
\]
with
\[b(v)=\delta_{\mu}\left(\phi_{n+1}(v)-\varphi_{n+1}\left({\mathcal I}(v)\right)-\int^{1}_{0}H(dv)\right).
\]
As the result of this replacement, we obtain another isomorphism  ${\mathcal I}_{\mu}$ and another homotopy $H_{\mu}$.
Then, by the same proof of Theorem \ref{Morg}, we also have:
\begin{theorem}\label{NWWFUN}
For the  isomorphism ${\mathcal I}_{\mu}:{\mathcal M}^{\ast}_{\C}\to {\mathcal N}$, taking the filtration $F^{r}_{\mu}({\mathcal M}^{\ast}_{\C})={\mathcal I}^{-1}_{\mu}(\bigoplus_{P\ge r} ({\mathcal N}^{\ast})^{P,Q})$, $({\mathcal M}^{\ast}, W_{\ast},F^{\ast}_{\mu})$ is an $\R$-mixed Hodge structure which is compatible with the differential and the multiplication.
The co-module structure ${\mathcal M}^{\ast}\to \R[T]\otimes {\mathcal M}^{\ast}$ corresponding to the $T$-module structure on ${\mathcal M}^{\ast}$ is a morphism of $\R$-mixed Hodge structures.

\end{theorem}
We notice that $C_{\mu}$ is not closed under the multiplication.
Thus our construction may be  different from the construction by the reduced bar construction of an augmented multiplicative mixed Hodge complex as in \cite{HainI}.

This alternate construction is interesting for the functoriality.
Let $(M_{1},g_{1})$ and $(M_{2},g_{2})$ be compact K\"ahler manifolds
and  $f:M_{2}\to M_{1}$ a holomorphic map.
Excepting the Condition (V),  we consider the same situation as in Subsection \ref{Funcmx}.
We have the
$T$-equivariant DGA maps $f_{\mathcal M}:\,_{1}{\mathcal M}^{\ast}\to \,_{2}{\mathcal M}^{\ast}$ and 
$f_{\mathcal N}:\,_{1}{\mathcal N}^{\ast}\to \,_{2}{\mathcal N}^{\ast}$.
Let $d\mu^{\prime}_{1}$ and $d\mu^{\prime}_{2}$ be volume forms on  $(M_{1},g_{1})$ and $(M_{2},g_{2})$ respectively 
which are different from the volume forms associated with the K\"ahler metrics.
We consider the following condition on $f:M_{2}\to M_{1}$.
\begin{condV2}
For any $\Psi\in A^{0}(M_{1})\otimes \C$ satisfying $\int_{M_{1}} \Psi d\mu^{\prime}_{1}=0$, we have
\[\int_{M_{2}} (f^{\ast}\Psi )d\mu^{\prime}_{2}=0.\] 
\end{condV2}
Let 
\[C_{\mu_{i}^{\prime}}=\left\{\Psi\in A^{0}(M)\otimes \C:  \int_{M_{i}} \Psi d\mu^{\prime}_{i}=0  \right\}
\]
for $i=1, 2$.
We define the maps $\delta_{\mu^{\prime}_{i}}$ as above.
The Condition (V')
 holds if and only if $f^{\ast}(C_{\mu_{1}^{\prime}})\subset C_{\mu_{2}^{\prime}}$.
By this, the Condition (V') implies the commutativity
\[f^{\ast}\circ \delta_{\mu^{\prime}_{1}}=\delta_{\mu^{\prime}_{2}}\circ f^{\ast}
\]
on $D(A^{0}(M_{1}, {\mathcal O}_{\rho}\otimes \C))$.
For the isomorphisms ${\mathcal I}_{\mu^{\prime}_{i}}$ and  homotopies $H_{\mu^{\prime}_{i}}$ associated with the volume forms $\mu^{\prime}_{i}$ as above, by the same proof of Proposition \ref{Hcommu},
we have:
\begin{proposition}
If $f$ satisfies the Condition (V'), then
\[{\mathcal I}_{\mu^{\prime}_{2}}\circ f_{\mathcal M}=f_{\mathcal N}\circ {\mathcal I}_{\mu^{\prime}_{1}}\]
 and \[H_{\mu^{\prime}_{2}}\circ f_{\mathcal M}=(f^{\ast}\otimes {\rm id}_{[t,dt]})\circ  H_{\mu^{\prime}_{1}}.\]
\end{proposition}

Thus we have the functoriality as Corollary \ref{mixcom} for the $\R$-mixed Hodge structures as in Theorem \ref{NWWFUN}.

Suppose  $f:M_{2}\to M_{1}$ is a holomorphic submersion.
Then we can define the push-forward $f_{\ast}: A^{\ast}(M_{2})\to A^{\ast-r}(M_{1})$ for $r=\dim M_{2}- \dim M_{1}$.
We have
\[\int_{M_{2}}(f^{\ast}\Psi) d\mu^{\prime}_{2} =\int_{M_{1} }\Psi f_{\ast}d\mu^{\prime}_{2}
\]
for any $\Psi\in A^{0}(M_{1})\otimes \C$.
Thus,  if $f_{\ast}d\mu^{\prime}_{2}=d\mu^{\prime}_{1}$, then the Condition (V') holds.

\section{Constructing VMHSs}\label{MAIVM}
In this section, we assume the same settings and use same notations as in the previous section.
We notice that the arguments in this section are valid for the constructions associated with arbitrary volume forms as in  Subsection \ref{altV}.
\subsection{Mixed Hodge $(T,\frak u)$-representations}

\begin{definition}\label{Tmh}
Let $V$ be a finite-dimensional $\R$-vector space with an $\R$-mixed Hodge structure $(W_{\ast},F^{\ast})$.
Then  a $T$-module structure on $V$ is called {\em mixed Hodge} if the corresponding $\R[T]$-co-module structure $V\to V\otimes \R[T]$ is a morphism of $\R$-mixed Hodge structures.
\end{definition}

\begin{lemma}\label{irdho}
 Let $V$ be a finite-dimensional $\R$-vector space
with an $\R$-mixed Hodge structure $(W_{\ast},F^{\ast})$.
We suppose that $V$ admits a mixed Hodge $T$-module structure.
Then,  for any irreducible representation $V_{\alpha}$ of $T$,  $(V_{\alpha}^{\ast}\otimes V)^{T}$ is an $\R$-mixed Hodge substructure  of $V_{\alpha}^{\ast}\otimes V$.

\end{lemma}

\begin{proof}
For the $T$-module $V_{\alpha}^{\ast}\otimes V$, we consider the corresponding $\R[T]$-co-module structure $V_{\alpha}^{\ast}\otimes V\to (V_{\alpha}^{\ast}\otimes V)\otimes \R[T]$.
Then the $T$-fixed part $(V_{\alpha}^{\ast}\otimes V)^{T}$ is the kernel of the map
\[V_{\alpha}^{\ast}\otimes V\to (V_{\alpha}^{\ast}\otimes V)\otimes \R[T]/(V_{\alpha}^{\ast}\otimes V)\otimes\langle 1\rangle .
\]
By the assumption, the map 
\[V_{\alpha}^{\ast}\otimes V\to (V_{\alpha}^{\ast}\otimes V)\otimes \R[T]/(V_{\alpha}^{\ast}\otimes V)\otimes\langle 1\rangle \]
is a morphism of $\R$-mixed Hodge structures.
Thus $(V_{\alpha}^{\ast}\otimes V)^{T}$ is an $\R$-mixed Hodge substructure  of $V_{\alpha}^{\ast}\otimes V$ (see \cite[Section 3.1]{PS}).

\end{proof}
As in Proposition \ref{spitto}, for some $b\in {\rm Aut}_{1}(V_{\C},W_{\ast})$ we can write $F^{\ast}=b^{-1}F_{sp}^{\ast}$ such that $(W_{\ast},F^{\ast}_{sp})$ is an $\R$-split $\R$-mixed Hodge structure.
Now we should remark that the $V$'s occurring  here are not $\V$'s  that were  used in the construction of ${\mathcal M}^{\ast}$.
In particular this $b$ is different from the one in the previous section.

We consider the canonical  $1$-minimal model ${\mathcal M}^{\ast}$ and the dual Lie algebra ${\frak u}$ with the $\R$-mixed Hodge structures $(W_{\ast}, F^{\ast})$ as in the last section.

\begin{definition}\label{mixuu}
Let $V$ be a finite-dimensional $\R$-vector space with  an $\R$-mixed Hodge structure $(W_{\ast},F^{\ast})$.
Let  $\Omega:{\frak u}\to {\rm End}(V)$ be a representation.
The representation  $\Omega$ is 
 called {\em mixed Hodge}  if $\Omega:{\frak u}\to {\rm End}(V)$ is a  morphism of $\R$-mixed Hodge structures.
 \end{definition}

For a finite-dimensional $\R$-vector space $V$ with  an $\R$-mixed Hodge structure $(W_{\ast},F^{\ast})$, ${\frak n}=
W_{-1}({\rm End}(V))$ is a nilpotent Lie algebra.
By $W_{-1}{\frak u}=\frak u$, if $\Omega:{\frak u}\to {\rm End}(V)$ is compatible with the weight filtrations $W_{\ast}$, then  $\Omega(\frak u)\subset {\frak n}$.
A representation $\Omega:{\frak u}\to {\rm End}(V)$ is identified with an element $\Omega\in {\mathcal M}^{1}\otimes {\rm End}(V)$ satisfying the Maurer-Cartan equation $d\Omega+\frac{1}{2}[\Omega,\Omega]=0$.
For this identification,  if $\Omega:{\frak u}\to {\rm End}(V)$ is compatible with the weight filtrations $W_{\ast}$, then  $\Omega\in {\mathcal M}^{1}\otimes  {\frak n}$.

\begin{remark}\label{spbig}
Let $V$   a finite-dimensional $\R$-vector space with  an $\R$-mixed Hodge structure $(W_{\ast},F^{\ast})$.
As in Proposition \ref{spitto}, for some $b\in {\rm Aut}_{1}(V_{\C},W_{\ast})$ we can write $F^{\ast}=b^{-1}F_{sp}^{\ast}$ such that $(W_{\ast},F^{\ast}_{sp})$ is an $\R$-split
 $\R$-mixed Hodge structure on $V$.
 Consider the bigrading $V_{\C}=\bigoplus V^{p,q}$ for this $\R$-split $\R$-mixed Hodge structure and write $V_{r}=\left(\bigoplus_{p+q=r} V^{p,q}\right)\cap V$.
 Then, a representation  $\Omega:{\frak u}\to {\rm End}(V)$ is a morphism of $\R$-mixed Hodge structures if 
 \[\Omega \left(\bigoplus_{i\le r} \V_{i}^{\ast}\right) \left(\bigoplus_{i\le s} V_{i}\right)\subset \bigoplus_{i\le r+s}V_{i}
 \]
 and 
 \[ {\mathcal I}(\Omega )\left(\bigoplus_{-P\ge r} (\V^{P,Q})^{\ast} \right) b^{-1} \left(\bigoplus_{p\ge s} V^{p,q}\right)\subset   b^{-1}\left( \bigoplus_{p\ge r+s} V^{p,q}\right).
 \]
\end{remark}

\begin{definition}
{\em A mixed Hodge $(T,\frak u)$-representation} is  $(V,W_{\ast}, F^{\ast}, \Omega)$ so that:
\begin{enumerate}
\item  $V$ is  a finite-dimensional $\R$-vector space with  an $\R$-mixed Hodge structure $(W_{\ast},F^{\ast})$.
\item $V$ is a $T$-module and it is mixed Hodge as in Definition \ref{Tmh}.
\item $\Omega$ is a mixed Hodge representation as in Definition \ref{mixuu}.

\item  $\Omega:{\frak u}\to {\rm End}(V)$ is $T$-equivariant.
\end{enumerate}
\end{definition}

\subsection{Flat bundles associated with a mixed Hodge $(T,\frak u)$-representation}
For a  mixed Hodge $(T,\frak u)$-representation
$(V,W_{\ast}, F^{\ast}, \Omega)$, we consider the following  flat bundles.

\begin{itemize}
\item Define the ${\mathcal C}^{\infty}$-vector bundle ${\bf E}=\bigoplus_{\alpha} (V_{\alpha}^{\ast}\otimes V)^{T}\otimes {\bf E}_{\alpha}$ with the flat connection $D=\bigoplus_{\alpha} D_{\alpha}$.
\begin{itemize}
\item  We can identify $\left(A^{\ast}(M,{\mathcal O}_{\rho})\otimes {\rm End}(V)\right)^{T}$ with $A^{\ast}(M, {\rm End}(\bf E))$.
\item By the arguments after Definition \ref{mixuu}, we can regard $\Omega \in ({\mathcal M}^{1}\otimes {\rm End}(V))^{T}$ satisfying the Maurer-Cartan equation.
\item By the maps $\phi:{\mathcal M}^{\ast}\to  A^{\ast}(M,{\mathcal O}_{\rho})$ and $\varphi \circ {\mathcal I}: {\mathcal M}^{\ast}_{\C}\to  A^{\ast}(M,{\mathcal O}_{\rho}\otimes \C)$,
we obtain the Maurer-Cartan elements $\Omega_{\phi}=\phi(\Omega)\in A^{1}(M, {\rm End}(\bf E))$ and  $\Omega_{\varphi}=\varphi(I(\Omega))\in A^{1}(M, {\rm End}(\bf E_{\C}))$. 
\end{itemize}
\item Define the flat bundle ${\bf E}_{\Omega_{\phi}}$ as the vector bundle  ${\bf E}$ with the flat connection $D+\Omega_{\phi}$.
\item Define the flat bundle ${\bf E}_{\Omega_{\varphi}}$ as the vector bundle  ${\bf E}_{\C}$ with the flat connection $D+\Omega_{\varphi}$.
\end{itemize}

By Lemma \ref{irdho}, each  $(V_{\alpha}^{\ast}\otimes V)^{T}$
admits an $\R$-mixed Hodge structure.
Each ${\bf E}_{\alpha}$ is an $\R$-VHS by ${\bf E}_{\alpha}={\mathbb S}_{\lambda}{\bf E}_{0}\cap {\bf E}^{[d]}_{0}$.
Thus, we obtain the increasing filtration ${\bf W}_{\ast}$ on ${\bf E}$ induced by the weight filtrations of $(V_{\alpha}^{\ast}\otimes V)^{T}$ and the weights of ${\bf E}_{\alpha}$ 
and decreasing filtration ${\bf F}^{\ast}$ on ${\bf E}_{\C}$ induced by Hodge filtrations on $(V_{\alpha}^{\ast}\otimes V)^{T}$ and ${\bf E}_{\alpha}$.

\begin{lemma}\label{WETTT}
On any one of the flat bundles ${\bf E}$, ${\bf E}_{\Omega_{\phi}}$ and ${\bf E}_{\Omega_{\varphi}}$,
 the filtration ${\bf W}_{\ast}$  is a  filtration of flat bundles.
Moreover, for any $i$,  the identity map on ${\bf E}$ induces isomorphisms of flat bundles
$Gr_{i}^{\bf W}{\bf E} \cong Gr_{i}^{\bf W}{\bf E}_{\Omega_{\phi}}$ and 
$Gr_{i}^{\bf W}{\bf E}_{\C}\cong  Gr_{i}^{\bf W}{\bf E}_{\Omega_{\varphi}}$.
\end{lemma}
\begin{proof}
On ${\bf E}$, the first assertion is obvious.

By the arguments after Definition \ref{mixuu}, we have $\Omega \in ({\mathcal M}^{1}\otimes {\frak n})^{T}$ where ${\frak n}=
W_{-1}({\rm End}(V))$.
By this, we can say that 
$\Omega_{\phi}\wedge {\bf W}_{i}\subset A^{1}(M, {\bf W}_{i-1})$ and $\Omega_{\varphi}\wedge {\bf W}_{i}\subset A^{1}(M, {\bf W}_{i-1})$.
This implies the lemma.
\end{proof}

\begin{proposition}\label{Gritr}{\rm (cf. \cite[Lemma 3.11]{ES})}
\[(D+\Omega_{\varphi})^{1,0}{\bf F}^{r}\subset A^{1,0}(M,{\bf F}^{r-1})
\]
and
\[(D+\Omega_{\varphi})^{0,1}{\bf F}^{r}\subset  A^{0,1}(M,{\bf F}^{r}).
\]
Thus the filtration ${\bf F}^{\ast}$ is a filtration on the holomorphic vector bundle ${\bf E}_{\Omega_{\varphi}}$ and the Griffiths transversality holds.
\end{proposition} 
\begin{proof}
It is sufficient to show
\[ (\Omega_{\varphi})^{1,0}\wedge {\bf F}^{r}\subset A^{1,0}(M,{\bf F}^{r-1}) \qquad {\rm  and}  \qquad (\Omega_{\varphi})^{0,1}\wedge{\bf F}^{r}\subset  A^{0,1}(M,{\bf F}^{r}).\]

As in Proposition \ref{spitto}, we have $b_{\alpha}\in {\rm Aut}_{1}((V_{\alpha}^{\ast}\otimes V)^{T},W_{\ast})$ which 
makes the  $\R$-mixed Hodge structure on $(V_{\alpha}^{\ast}\otimes V)^{T}$ $\R$-split.
Take $b=\sum b_{\alpha}\otimes {\rm id}_{V_{\alpha}}\in  {\rm Aut}_{1}(V,W_{\ast})$.
Then $(W_{\ast}, bF^{\ast})$ is an $\R$-split $\R$-mixed Hodge structure on $V$ and $b$ commutes with the $T$-action.
We take the bigrading $V_{\C}=\bigoplus V^{p,q}$ for the $\R$-split  $\R$-mixed Hodge structure  $(W_{\ast}, bF^{\ast})$ on $V$
and the bigrading ${\rm End}(V_{\C})=\bigoplus {\rm End}(V)^{p,q}$ which is induced by $V_{\C}=\bigoplus V^{p,q}$.

Let $\Omega^{\prime}=b{\mathcal I}(\Omega)b^{-1}$ and $\Omega^{\prime}_{\varphi}=\varphi(\Omega^{\prime})$.
By the $\R$-split $\R$-mixed Hodge structures on   $(V_{\alpha}^{\ast}\otimes V)^{T}$ and $\R$-VHSs ${\bf E}_{\alpha}$,
we take the bigrading ${\bf E}_{\C}=\bigoplus {\bf E}^{P,Q}$.
Then our goal is to show
\[ (\Omega^{\prime}_{\varphi})^{1,0}\wedge \left(\bigoplus_{P\ge r} {\bf E}^{P,Q}\right)\subset \bigoplus_{P\ge r-1} A^{1,0}(M,{\bf E}^{P,Q}) \]
and 
\[ (\Omega^{\prime}_{\varphi})^{0,1}\wedge\left( \bigoplus_{P\ge r} {\bf E}^{P,Q}\right)\subset \bigoplus_{P\ge r}  A^{0,1}(M,{\bf E}^{P,Q}).\]
We take the bigrading
\[ A^{\ast}(M, {\rm End} ({\bf E}_{\C}))=\bigoplus A^{\ast}(M, {\rm End} ({\bf E}))^{P,Q}
\]
as in Section \ref{VHSDEC}.
Then it is sufficient to show $\Omega^{\prime}_{\varphi}\in \bigoplus_{P\ge 0} A^{\ast}(M, {\rm End} ({\bf E}))^{P,Q}$.

Write $\Omega^{\prime}=\omega_{1}^{\prime}+\dots+\omega_{l}^{\prime}$ such that $\omega^{\prime}_{k}\in \bigoplus_{P+Q=k}({\mathcal N}^{\ast})^{P,Q}\otimes {\rm End}(V)$.
Then we have $d\omega^{\prime}_{k}=-\frac{1}{2}\sum_{i+j=k}[\omega^{\prime}_{i},\omega^{\prime}_{j}]$.
Thus we obtain \[\varphi(\omega^{\prime}_{k})=-\frac{1}{2}\sum_{i+j=k}D^{\prime}F^{\prime}_{g}[\varphi(\omega^{\prime}_{i}),\varphi(\omega^{\prime}_{j})].\]
Since $D^{\prime}F^{\prime}_{g}: A^{\ast}(M, {\rm End} ({\bf E}))^{P,Q}\to A^{\ast}(M, {\rm End} ({\bf E}))^{P,Q-1}$,
if $\varphi (\omega_{1}^{\prime})\in  \bigoplus_{P\ge 0} A^{\ast}(M, {\rm End} ({\bf E}))^{P,Q}$, then we can say
 $\varphi (\omega_{k}^{\prime})\in  \bigoplus_{P\ge 0} A^{\ast}(M, {\rm End} ({\bf E}))^{P,Q}$ for each $k$ inductively.
Thus it is sufficient to show $\varphi (\omega_{1}^{\prime})\in  \bigoplus_{P\ge 0} A^{\ast}(M, {\rm End} ({\bf E}))^{P,Q}$.

As in  Remark \ref{spbig}, we have 
 \[\Omega^{\prime}\left(\bigoplus_{-P\ge r} (\V^{P,Q})^{\ast} \right)  \left(\bigoplus_{p\ge s} V^{p,q}\right)\subset    \bigoplus_{p\ge r+s} V^{p,q}.
 \]
By this, for any $P,Q$, we have
\[\omega_{P+Q}^{\prime} (\V^{P,Q})=\Omega^{\prime}\left((\V^{P,Q})^{\ast} \right) \subset \bigoplus_{s\ge p-P} (V^{p,q})^{\ast}\otimes V^{s,t}\subset \bigoplus_{p\ge -P} {\rm End}(V)^{p,q}
\]
and hence
\[\omega_{k}^{\prime} \in    \bigoplus_{p+P\ge 0} ({\mathcal N}^{1})^{P,Q} \otimes {\rm End}(V)^{p,q}.
\]
By the construction of $\varphi:{ \mathcal N}^{\ast}\to A^{\ast}(M,{\mathcal O}_{\rho}\otimes \C)$, 
if $P+Q=1$, then $\varphi(\V^{P,Q})\subset A^{\ast}(M,{\mathcal O}_{\rho}\otimes \C)^{P,Q}$
and hence we have 
\[\varphi(\omega_{1}^{\prime})\in \bigoplus_{P\ge 0}A^{\ast}(M, {\rm End} ({\bf E}))^{P,Q}.
\]
Thus the proposition follows.
\end{proof}

\begin{proposition}\label{gau}
There exists a weight preserving gauge transformation $a$ of ${\bf E}_{\C}$ so that:
\begin{itemize}
\item $a(\Omega_{\phi})=\Omega_{\varphi}$ where $a(\Omega_{\phi})=a^{-1}Da+a^{-1}\Omega_{\phi}a$.
\item $a$ induces the identity map on each $Gr_{i}^{\bf W}{\bf E}$.

\end{itemize}
Moreover, such transformation $a$ can be determined by $H(\Omega)\in  A^{\ast}(M,{\mathcal O}_{\rho}\otimes \C)\otimes [t,dt]$ where $H:{\mathcal M}^{\ast}_{\C}\to A^{\ast}(M,{\mathcal O}_{\rho}\otimes \C)\otimes [t,dt]$ is the  $T$-equivariant homotopy $H:{\mathcal M}^{\ast}_{\C}\to A^{\ast}(M,{\mathcal O}_{\rho}\otimes \C)\otimes [t,dt]$ as in Subsection \ref{isho}.
We call this $a$ canonical.
\end{proposition}
\begin{proof}
In \cite[Section 5]{Man}, it is shown that on a nilpotent DGLA (differential graded Lie algebra) $L^{\ast}$, for two Maurer-Cartan elements $x,y\in L^{\ast}$, the following two conditions are equivalent:
\begin{itemize}
\item $x$ and $y$ are homotopy equivalent  i.e. there exists Maurer-Cartan element $x(t)\in  L^{\ast}\otimes[t,dt]$ so that $x(0)=x$ and $x(1)=y$.
\item $x$ and $y$ are gauge equivalent i.e. there exists $A\in L^{0}$ so that $y=\exp(A) \ast x$ (see Subsection \ref{DEFLA}).
\end{itemize}

Consider the DGLA $(A^{\ast}(M,{\mathcal O}_{\rho}\otimes \C)\otimes  \frak n_{\C})^{T}$.
We can say the  homotopy equivalence of $\Omega_{\phi}$ and  $\Omega_{\varphi}$ by $H(\Omega)$.
Hence we obtain $A\in (A^{0}(M,{\mathcal O}_{\rho}\otimes \C)\otimes  \frak n_{\C})^{T}$ so that $\exp(A)\ast (\Omega_{\phi})=\Omega_{\varphi}$.
By ${\frak n}=
W_{-1}({\rm End}(V))$, regarding  $A\in {\rm End}({\bf E}_{\C})$, we have $A({\bf W}_{i})\subset {\bf W}_{i-1}$.
Thus $a=\exp(A)$ is a desired gauge transformation.

We will find more explicit $A$.
Let $H(\Omega)=\alpha(t)+\beta(t)dt$.
We use the techniques in  \cite[Lemma 5.6, Proposition 5.7 and the proof of Theorem 5.5]{Man}.
We will find $A(t)\in (A^{0}(M,{\mathcal O}_{\rho}\otimes \C)\otimes  \frak n_{\C})^{T}\otimes [t,dt]$
such that $\alpha(t)=\exp(A(t))(\Omega_{\phi})$.
As  explained around  \cite[Proposition 5.7]{Man}, it is sufficient to solve the differential equation
\[A^{\prime}(t)+\gamma^{A}(t)=\beta(t)
\]
with $A(0)=0$ for certain $\gamma^{A}(t) \in(A^{0}(M,{\mathcal O}_{\rho}\otimes \C)\otimes  \frak n_{\C})^{T}\otimes [t,dt]$.
We will explain $\gamma^{A}(t) $ in  detail.
By using  the Baker-Campbell-Hausdorff formula $\exp (X)\exp(Y)=\exp(X+Y+\frac{1}{2}[X,Y]+\frac{1}{12}([X,[X,Y]]+[Y,[Y,X]])\dots)$ and the Taylor expansion $A(t+h)=A(t)+A^{\prime}(t)h+\dots$, we have
\[\exp\left(A(t+h)\right)\exp\left(-A(t)\right)=\exp\left((A^{\prime}(t)+\gamma^{A}(t))h+\delta(t,h)h^{2}\right).
\]

Since the filtration $W_{\ast}$ on $V$ is a filtration of $T$-module and $T$ is reductive, we can take a decomposition $V=\bigoplus_{i} V_{i}$ of $T$-modules so that  $\bigoplus_{i\le k}V_{i}=W_{k}((V))$.
For ${\frak n}=
W_{-1}({\rm End}(V))$,   this decomposition induces the decomposition ${\frak n}=\bigoplus_{i>0}{\frak n}_{-i}$ of $T$-modules such that $\bigoplus_{i\ge k}{\frak n}_{-i}=W_{-k}({\rm End}(V))$ and $[{\frak n}_{-i},{\frak n}_{-j}]\subset {\frak n}_{-i-j}$.
We write $A(t)=\sum_{i}A_{-i}(t)$, $\gamma^{A}(t) =\sum_{i}\gamma^{A}_{-i}(t)$ and $\beta(t)=\sum_{i}\beta_{-i}(t)$ associated with this decomposition.
Then each $\gamma^{A}_{-i}(t)$ is a linear combination of iterated products of $A_{-j}(t)$ and $A^{\prime}_{-j}(t)$ with $j<i$.
Thus, inductively, we can solve $A(t)=\sum_{i}A_{-i}(t)$ as
\[A_{-i}(t)=\int_{0}^{t}\left(\beta_{-i}(t)-\gamma^{A}_{-i}(t)\right).
\]
\end{proof}

\subsection{Main construction}
For a  mixed Hodge $(T,\frak u)$-representation
${\frak V}=(V,W_{\ast}, F^{\ast}, \Omega)$, we construct an $\R$-VMHS $({\bf E}_{\frak V}, {\bf W}_{\frak V \ast},{\bf F}^{\ast}_{\frak V})$.
We take:
\begin{itemize}
\item ${\bf E}_{\frak V}={\bf E}_{\Omega_{\phi}}$.
\item ${\bf W}_{\frak V\ast}$ is the increasing filtration ${\bf W}$ on the $\C^{\infty}$-vector bundle ${\bf E}$.
\item  ${\bf F}_{\frak V}^{\ast}=a {\bf F}^{\ast}$ where $a$ is the canonical weight preserving gauge transformation as in Proposition \ref{gau}.
\end{itemize}

\begin{theorem}
 $({\bf E}_{\frak V}, {\bf W}_{\frak V \ast},{\bf F}^{\ast}_{\frak V})$ is an $\R$-VMHS.
\end{theorem}
\begin{proof}
By Lemma \ref{WETTT}, ${\bf W}_{\frak V \ast}$ is a filtration of the local system ${\bf E}_{\frak V}$.
By Proposition \ref{Gritr}, ${\bf F}^{\ast}_{\frak V}$ is a filtration of the holomorphic vector bundle ${\bf E}_{\frak V}\otimes {\mathcal O}_{M}$ and the Griffiths transversality holds.
We show that $Gr^{\bf W}_{k}({\bf E}_{\frak V})$ with the filtration induced by $ {\bf F}^{\ast}_{\frak V}$ is an $\R$-VHS of weight $k$.
We notice that $({\bf E},{\bf W}_{\ast}, {\bf F}^{\ast})$ is an $\R$-VMHS as in Example \ref{triVMHSS}.
By Lemma  \ref{WETTT}, as a local system, we have $Gr^{\bf W}_{k}({\bf E}_{\frak V})=Gr^{\bf W}_{k}({\bf E})$.
By Proposition \ref{gau},
${\bf F}_{\frak V}^{\ast}=a {\bf F}^{\ast}$ induces $\R$-VHS on $Gr^{\bf W}_{k}({\bf E}_{\frak V})$.
Hence the theorem follows.
\end{proof}

\begin{remark}
Our $\R$-VMHS $({\bf E}_{\frak V}, {\bf W}_{\frak V \ast},{\bf F}^{\ast}_{\frak V})$ is not necessarily graded-polarizable.
Since each ${\bf E}_{\alpha}$ is polarized $\R$-VHS, by the construction, if the mixed Hodge structure on $(V_{\alpha}^{\ast}\otimes V)^{T}$ is graded-polarizable for every $\alpha$, then the $\R$-VMHS $({\bf E}_{\frak V}, {\bf W}_{\frak V \ast},{\bf F}^{\ast}_{\frak V})$ is graded-polarizable.
\end{remark}

For a simple case, we explain a way of describing a gauge transformation $a$  as in Proposition \ref{gau} explicitly.

\begin{example}\label{len22}
We assume that  the weight filtration $W_{\ast}$ on $V$ is of length $2$ i.e. for some $k$, $W_{k-2}(V)=0$ and $W_{k}(V)=V$.
Then, by $W_{-3}({\rm End}(V))=0$, ${\frak n}$ is $2$-step i.e. $[\frak n,[\frak n,\frak n]]=0$ and we have $\Omega(W_{-3}({\frak u}))=0$.
By this, we have $\Omega=\omega_{1}+\omega_{2}$ such that $\omega_{1}\in \V_{1}\otimes {\frak n}$ and $\omega_{2}\in \V_{2}\otimes W_{-2}({\rm End}(V))$.
By the Maurer-Cartan equation $d\Omega+\frac{1}{2}[\Omega,\Omega]=0$, we have  $d\omega_{2}+\frac{1}{2}[\omega_{1},\omega_{1}]=0$.
By constructions of the maps $\phi$ and $\varphi$, we have $\Omega_{\phi}=\omega_{1}+D^{c}A$ and $\Omega_{\varphi}=\omega_{1}-2\sqrt{-1}D^{\prime}A$ where $A=-\frac{1}{2}F_{g}[\omega_{1},\omega_{1}]$.
In this case, a gauge transformation as in Proposition \ref{gau} is $a=\exp(\sqrt{-1}A)$.
\end{example}

\begin{remark}
Replace $A^{\ast}(M,{\mathcal O}_{\rho})$, $\left(A^{\ast}(M,{\mathcal O}_{\rho}\otimes \C)^{P,Q},D^{\prime}, D^{\prime\prime}\right)$ and $T$ by the usual de Rham complex $A^{\ast}(M)$, usual Dolbeault complex $(A^{\ast,\ast}(M),\partial, \bar\partial)$ and the trivial group respectively.
By the arguments as in Section \ref{1-mKal} and \ref{MAIVM},
we obtain Theorem (Prototype).
In this case, each $\R$-VHS on $Gr^{\bf W}_{k}({\bf E}_{\frak V})$ is constant and so the $\R$-VMHS $({\bf E}_{\frak V}, {\bf W}_{\frak V \ast},{\bf F}^{\ast}_{\frak V})$ is unipotent in the sense of Hain-Zucker as in \cite{HZ}.
In this case, we can avoid the argument \ref{DGKAA} and so we do not need a base point.
\end{remark}

\begin{remark}
In \cite{Pri},
Hain-Zucker's construction is extended by very interesting but complicated techniques.
\end{remark}

\subsection{Functoriality}
Let $(M_{1},g_{1})$ and $(M_{2},g_{2})$ be compact K\"ahler manifolds
and  $f:M_{2}\to M_{1}$ a holomorphic map satisfying the condition (V) as in Subsection \ref{Funcmx}.
We consider the same situation as in Subsection \ref{Funcmx}.
Let $\,_{1}{\frak V}=(V,W_{\ast}, F^{\ast}, \Omega)$ be a  mixed Hodge $(T,\,_{1}\frak u)$-representation where $\,_{1}\frak u$ (resp. $\,_{2}\frak u$)  is the pro-nilpotent Lie algebra as in Subsection \ref{proni} for $\,_{1}{\mathcal M}^{\ast}$ (resp. $\,_{2}{\mathcal M}^{\ast}$).
Now we obtain the $\R$-VMHS $({\bf E}_{\,_{1}\frak V}, {\bf W}_{\,_{1}\frak V \ast},{\bf F}^{\ast}_{\,_{1}\frak V})$ by the above way.
Since $f$ is holomorphic, we have the pull-back $\R$-VMHS $(f^{\ast}{\bf E}_{\,_{1}\frak V}, f^{\ast}{\bf W}_{\,_{1}\frak V \ast},f^{\ast}{\bf F}^{\ast}_{\,_{1}\frak V})$.
Otherwise, by the result in Subsection \ref{Funcmx}, $\,_{2}{\frak V}=(V,W_{\ast}, F^{\ast}, f_{\mathcal M}(\Omega))$ is a mixed Hodge $(T,\,_{2}\frak u)$-representation and hence we obtain the $\R$-VMHS $({\bf E}_{\,_{2}\frak V}, {\bf W}_{\,_{2}\frak V \ast},{\bf F}^{\ast}_{\,_{2}\frak V})$ by the above way.
\begin{proposition}
\[(f^{\ast}{\bf E}_{\,_{1}\frak V}, f^{\ast}{\bf W}_{\,_{1}\frak V \ast},f^{\ast}{\bf F}^{\ast}_{\,_{1}\frak V})=({\bf E}_{\,_{2}\frak V}, {\bf W}_{\,_{2}\frak V \ast},{\bf F}^{\ast}_{\,_{2}\frak V}).\]

\end{proposition}
\begin{proof}
Now we have $\phi_{2}(f_{\mathcal M}(\Omega))=f^{\ast}(\phi_{1}(\Omega))$ (see Subsection \ref{Funcmx}).
Thus we have $(f^{\ast}{\bf E}_{\,_{1}\frak V}, f^{\ast}{\bf W}_{\,_{1}\frak V \ast})=({\bf E}_{\,_{2}\frak V}, {\bf W}_{\,_{2}\frak V \ast})$. Now we write ${\bf F}_{\,_{1}\frak V}^{\ast}=\,_{1}a \,_{1}{\bf F}^{\ast}$ and ${\bf F}_{\,_{2}\frak V}^{\ast}=\,_{2}a \,_{2}{\bf F}^{\ast}$ as above. Then we can easily check that $\,_{2}{\bf F}^{\ast}=f^{\ast}\,_{1}{\bf F}^{\ast}$ by the construction.
Thus it is sufficient to show  $\,_{2}a =f^{\ast} (\,_{1}a) $.
We write $\,_{1}a=e^{\,_{1}A}$ and $\,_{2}a=e^{\,_{2}A}$ as in the proof of  Proposition \ref{gau}.
We prove $\,_{2}A=f^{\ast}(\,_{1}A)$.
Write $\,_{1}H(\Omega)=\,_{1}\alpha(t)+\,_{1}\beta(t)dt$ and  $\,_{2}H(f_{\mathcal M}(\Omega))=\,_{2}\alpha(t)+\,_{2}\beta(t)dt$ .
Then, $\,_{1}A=\,_{1}A(1)$ and $\,_{2}A=\,_{2}A(1)$ for the solutions $\,_{1}A(t)$ and $\,_{2}A(t)$ of  the differential equations
\[\,_{1}A^{\prime}(t)+\gamma^{\,_{1}A}(t)=\,_{1}\beta(t)
\]
and 
\[\,_{2}A^{\prime}(t)+\gamma^{\,_{2}A}(t)=\,_{2}\beta(t)
\]
with $\,_{1}A(0)=0$ and  $\,_{2}A(0)=0$ respectively (see the proof  of  Proposition \ref{gau}).
By the result of Subsection \ref{Funcmx}, we have $\,_{2}H(f_{\mathcal M}(\Omega))=f^{\ast}\otimes {\rm id}_{ [t,dt]} (\,_{1}H(\Omega))$.
Thus we have $\,_{2}\beta(t)=f^{\ast}\otimes {\rm id}_{ [t,dt]} (\,_{1}\beta(t))$ and this implies  $\,_{2}A(t)=f^{\ast}\otimes {\rm id}_{ [t,dt]}(\,_{1}A(t))$.
Hence we obtain $\,_{2}A=f^{\ast}(\,_{1}A)$.
\end{proof}

\section{Constructing VMHSs by substructures of $1$-minimal models}\label{SUBSMMM}
\subsection{Representations of nilpotent Lie algebras}\label{nilf}
The reference of this subsection is \cite{Re}.

Let $\frak n$ be a nilpotent Lie algebra and ${\frak n}={\frak n}_{1}\supset {\frak n}_{2}\supset {\frak n}_{3}\dots$ the lower central series of $\frak n$ (i.e. $[\frak n,\frak n_{i}]=\frak n_{i+1}$).
It is known that $[{\frak n}_{i},{\frak n}_{j} ]\subset {\frak n}_{i+j}$.
We say that $\frak n$ is $k$-step if ${\frak n}_{k}\not=0$ and ${\frak n}_{k+1}=0$.
Define the increasing filtration $W_{\ast}$ of ${\frak n}$ so that $W_{-k}={\frak n}_k$ for $k>0$.

Let $U(\frak n)$ be the universal enveloping algebra.
That is $U({\frak n})=T({\frak n})/I$ where $T({\frak n})$ is the tensor algebra of $\frak n$ and $I$ is the ideal which is generated by \[\{X\otimes Y-Y\otimes X-[X,Y]\vert X,Y\in \frak n\}.\]
Then  we have the natural increasing filtration $W_{\ast}$ of $U(\frak n)$ induced by the above increasing filtration  of ${\frak n}$.
This filtration is compatible with the multiplication of $U(\frak n)$.
We suppose $\frak n$ is $k$-step.
Let $J=W_{-k}(U(\frak n))$.
Then $J$  is an ideal and the map ${\frak n}\to U(\frak n)/J$ induced by the natural inclusion ${\frak n}\hookrightarrow U(\frak n)$ is an injection.
Thus, for the endomorphisms ${\rm End}(U(\frak n)/J)$ of the finite-dimensional vector space $U(\frak n)/J$, 
 we have a finite-dimensional faithful representation $\tau:{\frak n}\to {\rm End}(U(\frak n)/J)$.
For any $x\in \frak n$ and any integer $i$, \[\tau(x) (W_{-i}(U(\frak n)/J))\subset W_{-i-1}(U(\frak n)/J)\] and so $\tau$ is a nilpotent representation.

\subsection{Sub-structures of $1$-minimal models}\label{subs}
\begin{definition}\label{subst}
A {\em $k$-step sub-structure}  of ${\mathcal M}^{\ast}$ is a sub-vector space ${\mathcal  X}\subset {\mathcal M}^{1}$ so that\begin{itemize}
\item ${\mathcal X}={\mathcal X}_{1}\oplus\dots\oplus{\mathcal  X}_{k}$ such that for each $1\le i\le k$,  ${\mathcal X}_{i}\not=0$ and
${\mathcal X}_{i}\subset \V_{i}$.
\item  $d:{\mathcal X}_{r}\to \sum_{i+j=r} {\mathcal X}_{i}\wedge {\mathcal  X}_{j}$  (thus $\bigwedge {\mathcal X}$ is a sub-DGA of ${\mathcal M}^{\ast}$).
\item $\mathcal X$  is an $\R$-mixed Hodge substructure of ${\mathcal M}^{1}$.
\item $\mathcal X$ is a $T$-submodule of ${\mathcal M}^{1}$.

\end{itemize}

\end{definition}

We assume that a $k$-step sub-structure $\mathcal X$ is finite-dimensional.
Consider the nilpotent Lie algebra ${\mathcal X}^{\ast}$ which is dual to the DGA $\bigwedge {\mathcal X}$.
Take the dual $\R$-mixed Hodge structure $(W_{\ast},F^{\ast})$ on ${\mathcal X}^{\ast}$.
Then, for  $\X^{\ast}=\X^{\ast}_{1}\oplus\dots\oplus \X^{\ast}_{k}$, the bracket on $\X^{\ast}$ is dual to the differential $d:\X_{r}\to \sum_{i+j=r} \X_{i}\wedge \X_{j}$.
Hence, we can easily  check that  $\X^{\ast}=\X^{\ast}_{1}\oplus\dots\oplus \X^{\ast}_{k}$ is a graded nilpotent Lie algebra so that the filtration $W_{-n}(\X^{\ast})=\bigoplus_{i\ge n} \X^{\ast}_{i}$ is the natural filtration which is given by the lower central series of $\X^{\ast}$.

Consider the universal enveloping algebra $U(\X^{\ast})$ of the nilpotent Lie algebra $\X^{\ast}$.
Then, we obtain the $\R$-mixed Hodge structure $(W_{\ast},F^{\ast})$ on $U(\X^{\ast})$ which is induced by the $\R$-mixed Hodge structure on $\X^{\ast}$.
By the above argument, the weight  filtration $W_{\ast}$ is the natural filtration induced by  the lower central series of $\X^{\ast}$.
Hence, for $J=W_{-k}(U(\X))$, we obtain the $\R$-mixed Hodge structure $(W_{\ast},F^{\ast})$ on the quotient space $U(\X^{\ast})/J$.

Consider the faithful representation $\X^{\ast}\to {\rm End}(U(\X^{\ast})/J) $ as in Subsection \ref{nilf}.
Since the multiplication on $U(\X^{\ast})$ is a  morphism of  $\R$-mixed Hodge structures,  we can easily show that the representation $\X^{\ast}\to {\rm End}(U(\X^{\ast})/J) $ is a morphism of  $\R$-mixed Hodge structure.
Consider the composition $\Omega:{\frak u}\to  {\rm End}(U(\X^{\ast})/J) $ of the surjection ${\frak u}\to \X^{\ast}$ which is dual to the inclusion $ \X\subset {\mathcal M}^{1}$ and the faithful representation $\X^{\ast}\to {\rm End}(U(\X^{\ast})/J) $.
Since $\X$ is a $T$-submodule of ${\mathcal M}^{1}$, we can say that
 \[(U(\X^{\ast})/J, W_{\ast},F^{\ast}, \Omega)\] 
is a mixed Hodge $(T,\frak u)$-representation.
Hence, we obtain an $\R$-VMHS.
More precisely, taking a basis $x_{i1},\dots, x_{il_{i}}$ of each $\X_{i}$ and the dual basis $\chi_{i1},\dots, \chi_{il_{i}}$ of $\X^{\ast}_{i}$, we have $\Omega=\sum_{ij} x_{ij}\otimes \chi_{ij}$ and so we can write $\Omega_{\phi}=\sum_{ij} \phi(x_{ij})\otimes \chi_{ij}$.

\begin{remark}\label{nofi}
Each $(\V_{1}\oplus \dots \oplus \V_{n})$ is a $n$-step sub-structure of ${\mathcal M}^{\ast}$ as in Definition \ref{subst}.
However, it is not finite-dimensional in general.
By the construction, if $\V_{1}$ is finite-dimensional, then all $\V_{i}$ are also finite-dimensional.

Since we have
\[\V_{1}\cong H^{1}(A^{\ast}(M,{\mathcal O}_{\rho}))\cong \bigoplus_{\alpha} H^{1}(M, {\bf E}_{\alpha}^{\ast})\otimes V_{\alpha},
\]
if  the group cohomology
\[ \bigoplus_{\alpha}  H^{1}(\pi_{1}(M,x), V_{\alpha}^{\ast})\otimes V_{\alpha}.
\]
is finite-dimensional, then 
$\V_{1}$ is finite-dimensional.
\end{remark}

\begin{proposition}\label{gomh}
We suppose the following conditions:
\begin{itemize}
\item ${\rm im}\rho $ is a co-compact discrete subgroup in $T$.
\item $T\not\cong O(m,1)$ for any $m$.
\item The group cohomology $H^{1}({\rm ker}\rho,\R)$ is finite-dimensional.
\end{itemize}
Then  the group cohomology
\[ \bigoplus_{\alpha}  H^{1}(\pi_{1}(M,x), V_{\alpha}^{\ast})\otimes V_{\alpha}
\]
is finite-dimensional.
Hence, in this case,  $(\V_{1}\oplus \dots \oplus \V_{n})$ is a finite-dimensional $n$-step sub-structure of ${\mathcal M}^{\ast}$ and so
we obtain the  $\R$-VMHS associated with $(V_{1}\oplus \dots \oplus V_{n})$ for each $n$.
\end{proposition}

\begin{proof}
Let $T_{0}$ be the identity component of $T$ and $\Gamma={\rm im}\rho\cap T_{0}$.
Then $\Gamma$ is a co-compact discrete subgroup of the connected semi-simple Lie group $T_{0}$.
We notice that  $\Gamma$ is a finite-index normal subgroup of ${\rm im}\rho$.
For the extension
\[\xymatrix{
1\ar[r]& {\rm ker}\rho\ar[r]&\pi_{1}(M,x)\ar[r]&{\rm im}\rho\ar[r]&1
 },\]
we have the spectral sequence $E^{p,q}_{r}$ so that 
\[E^{p,q}_{2}=\bigoplus_{\alpha}  H^{p}({\rm im}\rho, H^{q}({\rm ker}\rho,\R)\otimes V_{\alpha}^{\ast})\otimes V_{\alpha}
\]
and it converges to $\bigoplus_{\alpha}  H^{p+q}(\pi_{1}(M,x), V_{\alpha}^{\ast})\otimes V_{\alpha}$.
It is sufficient to show that $E_{2}^{1,0}$ and $E^{0,1}_{2}$ are finite-dimensional.
By Raghunathan's result in \cite{Rag}, for non-trivial $V_{\alpha}$, we have $H^{1}(\Gamma, V_{\alpha})=0$.
Since $\Gamma$ is a finite-index normal subgroup of ${\rm im}\rho$, we have $H^{1}({\rm im}\rho, V_{\alpha})=0$.
Thus
\[E_{2}^{1,0}=\bigoplus_{\alpha}  H^{1}({\rm im}\rho, H^{0}({\rm ker}\rho,\R)\otimes V_{\alpha}^{\ast})\otimes V_{\alpha}
\]
is finite-dimensional.
We have 
\[E_{2}^{0,1}=\bigoplus_{\alpha}  H^{0}({\rm im}\rho, H^{1}({\rm ker}\rho,\R)\otimes V_{\alpha}^{\ast})\otimes V_{\alpha}
=\bigoplus_{\alpha} (H^{1}({\rm ker}\rho,\R)\otimes V_{\alpha}^{\ast})^{\rm im\rho}\otimes V_{\alpha}.
\]
 Since ${\rm im\rho}$ is Zariski-dense in $T$, we have
\[\bigoplus_{\alpha} (H^{1}({\rm ker}\rho,\R)\otimes V_{\alpha}^{\ast})^{\rm im\rho}\otimes V_{\alpha}= \bigoplus_{\alpha} (H^{1}({\rm ker}\rho,\R)\otimes V_{\alpha}^{\ast})^{T}\otimes V_{\alpha}\cong H^{1}({\rm ker}\rho,\R).
\]
Since  $H^{1}({\rm ker}\rho,\R)$ is finite-dimensional, we can say that $E_{2}^{0,1}$ is finite-dimensional.
Hence the proposition follows.
\end{proof}

\subsection{Lower step   $\R$-VMHS}\label{lowers}
By using sub-structures of $1$-minimal models,
we can obtain explicit $\R$-VMHSs whose weight filtration is of length $1$ or $2$. 
\subsubsection{$1$-step $\R$-VMHS}
Let $\{V_{\alpha_{1}},\dots, V_{\alpha_{l}}\}$ be a finite set of irreducible representations of $T$.
Take $\X_{1}=\bigoplus_{i} {\mathcal H}^{1}(M, {\bf E}_{\alpha_{i}})\otimes V_{\alpha_{i}}^{\ast}$.
Then, regarding  $\X_{1}$ as a subspace in $\V_{1}$, obviously $\X_{1}$ is a finite-dimensional $1$-step sub-structure of $\mathcal M^{\ast}$.
In this case we have $U(\X_{1}^{\ast})/J= \langle 1\rangle \oplus \bigoplus_{i} {\mathcal  H}^{1}(M, {\bf E}_{\alpha_{i}})^{\ast}\otimes V_{\alpha_{i}}$
Take a basis $x^{\alpha_{i}}_{1},\dots x^{\alpha_{i}}_{m_{i}}$ of each ${\mathcal H}^{1}(M, {\bf E}^{\ast}_{\alpha_{i}})$.
We take 
 the dual basis $\chi^{\alpha_{i}}_{1},\dots \chi^{\alpha_{i}}_{m_{i}}$ of each ${\mathcal H}^{1}(M, {\bf E}_{\alpha_{i}})^{\ast}$.
Then, the $\R$-VHMS as in subsection \ref{subs} is given by the flat connection $D+\sum x^{\alpha_{i}}_{j}\otimes \chi^{\alpha_{i}}_{j}$ over the vector bundle $\R\oplus  \bigoplus_{i} { \mathcal H}^{1}(M, {\bf E}^{\ast}_{\alpha_{i}})^{\ast}\otimes {\bf E}_{\alpha_{i}}$.
Precisely, for $(f, \sum \eta_{\alpha_{i}})\in \R\oplus  \bigoplus_{i} {\mathcal H}^{1}(M, {\bf E}_{\alpha_{i}})^{\ast}\otimes {\bf E}_{\alpha_{i}}$,
we have 
\[\left(D+\sum x^{\alpha_{i}}_{j}\otimes \chi^{\alpha_{i}}_{j}\right) \left(f, \sum\eta_{\alpha_{i}}\right)=\left(df, \sum D_{\alpha_{i}}\eta_{\alpha_{i}}+f\sum x^{\alpha_{i}}_{j}\otimes \chi^{\alpha_{i}}_{j}\right)
\]
The weight filtration $\bf W_{\ast}$ is given by 
\[{\bf W}_{0}=\R\oplus  \bigoplus_{i}  H^{1}(M, {\bf E}_{\alpha_{i}})^{\ast}\otimes {\bf E}_{\alpha_{i}},\]
 \[{\bf W}_{-1}= \bigoplus_{i}  H^{1}(M, {\bf E}_{\alpha_{i}})^{\ast}\otimes {\bf E}_{\alpha_{i}}\]
  and ${\bf W}_{-2}=0$.
The Hodge filtration ${\bf F}^{\ast}$ on $(\R\oplus  \bigoplus_{i} {\mathcal H}^{1}(M, {\bf E}_{\alpha_{i}})^{\ast}\otimes {\bf E}_{\alpha_{i}})\otimes\C$ is given by the dual Hodge structures on ${\mathcal H}^{1}(M, {\bf E}_{\alpha_{i}})^{\ast}$  and the $\R$-VHSs ${\bf E}_{\alpha_{i}}$.
Similar construction is given in \cite[Example 2.2.2]{EKPR}.

\subsubsection{$2$-step $\R$-VMHS}
Take $\X_{1}=\bigoplus_{i} {\mathcal H}^{1}(M, {\bf E}_{\alpha_{i}})\otimes V_{\alpha_{i}}^{\ast}$. as the above argument.
Let $\X_{2}=d^{-1}( \X_{1}\wedge \X_{1})\subset \V_{2}$.
As in Remark \ref{nonsp}, we have ${\mathcal I}(\X_{2}\otimes \C)=d^{-1}( {\mathcal I}(\X_{1}\otimes \C)\wedge  {\mathcal I}(\X_{1}\otimes \C))$.
Hence, $\X_{1}\oplus \X_{2}$ is a finite-dimensional sub-structure of $\mathcal M^{\ast}$ and so we can construct the $\R$-VMHS associated with  $\X_{1}\oplus \X_{2}$.
Thus, we obtain the $\R$-VMHS $({\bf E}_{\frak V}, {\bf W}_{\frak V \ast},{\bf F}^{\ast}_{\frak V})$ associated with the mixed Hodge $(T,\frak u)$-module
$\frak V=(U(\X^{\ast})/J, W_{\ast},F^{\ast}, \Omega)$.
Since the weight filtration $W_{\ast}$ is of length $2$,
we can write the Hodge filtration ${\bf F}^{\ast}_{\frak V}$ as in Example \ref{len22}. 

For this construction, it is necessary that the kernel of the cup product 
\[\left(\bigoplus_{i} {\mathcal H}^{1}(M, {\bf E}_{\alpha_{i}})\otimes V_{\alpha_{i}}^{\ast}\right)\wedge \left(\bigoplus_{i} {\mathcal H}^{1}(M, {\bf E}_{\alpha_{i}})\otimes V_{\alpha_{i}}^{\ast}\right)\to H^{2}(M,{\mathcal O}_{\rho})\]
 is non-trivial.
On Example \ref{reimsu}, by $\dim M=2$, we have
\[H^{2}(M,{\mathcal O}_{\rho})=\bigoplus_{k=0}^{\infty} H^{2}(M,S^{k}{\bf E}^{\ast}_{0})\otimes S^{k}V_{0}\cong \R.
\]
By the  Euler number $\chi (M)=2-2g$ and $\dim V_{0}=2$, we have $\dim H^{1}(M, {\bf E}^{\ast}_{0})\otimes V_{0}= 8g-8\ge 8$.
Thus the kernel of  the cup product 
\[\left({\mathcal H}^{1}(M, {\bf E}^{\ast}_{0})\otimes V_{0}\right)\wedge\left( {\mathcal H}^{1}(M, {\bf E}^{\ast}_{0})\otimes V_{0}\right) \to H^{2}(M,{\mathcal O}_{\rho})\]
is non-trivial.

\section{Constructing $\R$-VMHS from Deformation theory}\label{DGLMH}
\subsection{Functors of Artinian algebras}
We briefly review the theory of Schlessinger's hull (\cite{Sch}).
Let ${\mathbb K}=\R$ or $\C$.
For a local  ${\mathbb K}$-algebra $R$, we denote by ${\frak m}_{R}$ the maximal ideal of $R$.
We define:
\begin{itemize}
\item $Art$ is the category of Artinian local $\mathbb K$-algebras.
\item $\overline{Art}$ is the category of complete Noetherian local $\mathbb K$-algebra $R$ so that $R/ {\frak m}_{R}^{k}\in Art$ for any $k$.
\item A functor $F$ of $Art$ is a covariant functor from $Art$ to the category of sets so that $F(\mathbb K)$ is a $1$-point set.
\item For  a functor $F$ of $Art$, we define $t_{F} =F(\mathbb K[t]/(t^{2}))$.
\item For $R\in \overline{Art}$, we define the functor $h_{R}$ of $Art$ as $h_{R}(A)=Hom(R,A)$.
\item For a functor $F$ of $Art$, $R\in  \overline{Art}$ and $\in \xi\in F(R)$, we define the morphism $T_{\xi}:h_{R}\to F$ of functors such that $h_{R}(A)\ni u\mapsto F(u)(\xi)\in F(A)$
where we extend $F$ to $\overline{Art}$.
\item A morphism $F\to G$ of functors of $Art$ is {\em smooth} if for any surjection $B\to A$ in $Art$ the map
\[F(B)\to F(A)\times_{G(A)} G(B)
\]
is surjective.
\item A morphism $F\to G$ of functors of $Art$ is {\em \'etale} if it is smooth and the induced map $t_{F}\to t_{G}$ is bijective.
\item For a functor $F$ of $Art$, $R\in  \overline{Art}$ and $ \xi\in F(R)$, a pair $(R,\xi)$ is a {\em hull} if the morphism $T_{\xi}:h_{R}\to F$ is \'etale.
\end{itemize}
The following uniqueness is important.
\begin{proposition}[\cite{Sch}]\label{hulluni}
 For a functor $F$ of $Art$, if two pairs $(R_{1},\xi_{1})$ and $(R_{2},\xi_{2})$ are hulls, then we have an isomorphism $u:R_{1}\to R_{2}$ such that $F(u)(\xi_{1})=\xi_{2}$.
\end{proposition}
\subsection{Deformation theory of DGLA}\label{DEFLA}
We briefly review the Kuranishi theory of a DGLA (differential graded Lie algebra) (\cite{Man}).
Let $L^{\ast}$ be a DGLA over $\mathbb K$ with a differential $d$.
We define $MC(L^{\ast})=\{x\in L^{1}\vert dx+\frac{1}{2}[x,x]=0\}$.
If $L^{\ast}$ is nilpotent, then for any $B\in L^{0}$, we can define the gauge transformation $\exp(B)\ast $ on $L^{1}$ by the exponential of the affine transformation $x\mapsto [B,x]-dB$.
Moreover this action preserves $MC(L^{\ast})$ and so we can define gauge equivalence  on $MC(L^{\ast})$ so that $x,y\in MC(L^{\ast})$ are gauge equivalent if for some $B\in L^{0}$, $y= \exp(B)\ast (x)$.

For any DGLA $L^{\ast}$ and $A\in Art$, the DGLA $L^{\ast}\otimes {\frak m}_{A}$ is nilpotent and we define $Def_{L^{\ast}}(A)$ as the set of gauge equivalent classes of $MC(L^{\ast}\otimes  {\frak m}_{A})$.
Consider the functor $Def_{L^{\ast}}: A\mapsto Def_{L^{\ast}}(A)$ of $Art$.
We notice that for $R\in \overline{Art}$ we can also define the gauge transformation $\exp (B)\ast$ of $B\in L^{0}\otimes {\frak m}_{S}$ by using formal power series and we can regard $Def_{L^{\ast}}(S)$ as  the set of gauge equivalent classes of $MC(L^{\ast}\otimes  {\frak m}_{S})$.

Let $L^{\ast}$ be a DGLA such that the cohomologies $H^{0}(L^{\ast})$, $H^{1}(L^{\ast})$ and $H^{2}(L^{\ast})$ are finite-dimensional.
We will  define the Kuranishi functor of $L^{\ast}$ as in  \cite[Section 4]{Man}.
For our application, we only consider the special case.
A {\em special grading} is a grading $L^{\ast}=\bigoplus_{k\ge 0}L_{k}^{\ast}$ on the vector space $L^{\ast}$ so that:
\begin{itemize}
\item $d L^{\ast}_{k}\subset L^{\ast}_{k}$ and $[L^{\ast}_{k_{1}},L^{\ast}_{k_{2}}]\subset L^{\ast}_{k_{1}+k_{2}}$.
\item $L_{0}^{\ast}=L^{0}$.
\item $L^{1}=\bigoplus_{k\ge 1}L_{k}^{1}$ and ${\rm ker}\, d_{\vert L^{1}}=L_{1}^{1}$.
\item $L^{2}=\bigoplus_{k\ge 2}L_{k}^{2}$.
\item $d: L^{1}_{2}\to {\rm ker}\, d_{\vert L^{2}_{2}}$ is injective and $d: L^{1}_{k}\to {\rm ker}\, d_{\vert L^{2}_{k}}$ is bijective for any $k\ge 3$.
\end{itemize}
This grading gives a special case of decomposition as in \cite[Section 4]{Man}.

We assume that $L^{\ast}$ admits a special grading.
For $A\in Art$, we define the set
\[Kur_{L^{\ast}}(A)=\{x_{1}\in L_{1}^{1}\otimes {\frak m}_{A}\vert [x_{1},x_{1}]\equiv 0 \in H^{2}(L^{\ast})\otimes  {\frak m}_{A}\}
\]
and the map $Kur_{L^{\ast}}(A)\ni x_{1}\mapsto \sum_{i}x_{i}\in MC(L^{\ast}\otimes{\frak m}_{A})$ so that for $k\ge 2$, \[dx_{k}=-\frac{1}{2}\sum _{i+j=k, i>0,j>0}[x_{i},x_{j}].\]
By the assumptions, each $x_{k}$ is uniquely determined by $x_{1}$.
We obtain the functor $Kur_{L^{\ast}}:A\mapsto Kur_{L^{\ast}}(A)$ and the morphism $Kur_{L^{\ast}}\to Def_{L^{\ast}} $ of functors so that  $Kur_{L^{\ast}}(A)\ni x_{1}\mapsto [ \sum_{i}x_{i}]\in Def_{L^{\ast}}(A)$.
\begin{theorem}\rm{(\cite[Theorem 4.7]{Man})}\label{manet}
The morphism $Kur_{L^{\ast}}\to Def_{L^{\ast}} $ is \'etale.
\end{theorem}
We can easily check that $Kur_{L^{\ast}}=h_{R}$ so that 
$R={\mathbb K}[[(L^{1}_{1})^{\ast}]]/I$ where ${\mathbb K}[[(L^{1}_{1})^{\ast}]]$ is the algebra of formal power series on $(L^{1}_{1})^{\ast}$ and 
$I$ is the ideal generated by the quadratic polynomials on $(L^{1}_{1})^{\ast}$ associated with $L_{1}^{1}\ni x\mapsto [x,x]\in H^{2}(L^{\ast})$.
Take $\xi_{1}\in L^{\ast}\otimes R$ which is the extension of the identity map $I\in L^{1}_{1}\otimes (L^{1}_{1})^{\ast}$.
Then, by elementary arguments, we have $[\xi_{1},\xi_{1}]\equiv 0\in H^{2}(L^{\ast})\otimes {\frak m}_{R}$ (see \cite[Lemma 3.4]{ES}).   
Take the formal power series $\xi=\sum_{i=1}^{\infty} \xi_{i}\in MC(L^{\ast}\otimes {\frak m}_{R})$ so that 
\[d\xi_{k}=-\frac{1}{2}\sum _{i+j=k, i>0,j>0}[\xi_{i},\xi_{j}].\]
By Theorem \ref{manet}, we have:
\begin{corollary}
For the functor $Def_{L^{\ast}}$ of $Art$ and 
 the gauge equivalent class $[\xi]\in Def_{L^{\ast}}(R)$ of $\xi$, the pair $(R,[\xi])$ is a hull.

\end{corollary}
We notice that the definition of $R$ is independent of the choice of a special  grading $L^{\ast}=\bigoplus_{k\ge 0}L_{k}^{\ast}$ 
as above but $\xi$ varies for  the choice of a special grading.
By Proposition \ref{hulluni},  we have:
\begin{corollary}\label{hullgau}
For two Maurer-Cartan elements $\xi, \xi^{\prime}\in   MC(L^{\ast}\otimes {\frak m}_{R})$ constructed as above associated with two special gradings of $L^{\ast}$,
there exists an automorphism $u:R\to R$ of $R$ and $B\in L^{0}\otimes {\frak m}_{R}$ such that
\[\xi^{\prime}=\exp(B)\ast u(\xi)=\exp({\rm ad}_{B})\circ u(\xi).
\]

\end{corollary}
\subsection{Mixed Hodge $(T,\frak u)$-representation associated with deformation theory}
We assume the same settings and use the same notations as in Section \ref{1-mKal} and \ref{MAIVM}.
Let $U$ be a finite-dimensional rational $T$-representation with an irreducible decomposition $U=\bigoplus V_{\gamma_{i}}$.
Now we have the polarized $\R$-Hodge structure on each $V_{\gamma_{i}}$ induced by the polarized $\R$-Hodge structure on $V_{0}$.
We assume that all such $\R$-Hodge structures have same weight.
Consider the polarized $\R$-VHS ${\bf E}_{\gamma_{i}}$ corresponding to each irreducible representation $\gamma_{i}$.
Then, we have the $\R$-VHS ${\bf E}_{U}=\bigoplus {\bf E}_{\gamma_{i}}$.

We consider the DGLA $L^{\ast}=({\mathcal M}^{\ast}\otimes {\rm End} (U))^{T}$ over $\R$.
Then we have the grading 
\[L^{\ast}=\bigoplus_{k}({\mathcal M}_{k}^{\ast}\otimes {\rm End} (U))^{T}.
\]
This grading is a special grading and $L^{1}_{1}=({\mathcal V}_{1}\otimes {\rm End}(U))^{T}={\mathcal H}^{1}(M, {\rm End}({\bf E}_{U}))$.
We study $R={\mathbb K}[[(L^{1}_{1})^{\ast}]]/I$ as in the last subsection.
Since the map $\phi:{\mathcal M}^{\ast}\to A^{\ast}(M,{\mathcal O}_{\rho})$ induces an injection on the second cohomology, the map $L^{\ast}\to  (A^{\ast}(M,{\mathcal O}_{\rho})\otimes {\rm End}(U))^{T}\cong A^{\ast}(M,{\rm End}({\bf E}_{U}))$ induces an injection $H^{2}(L^{\ast})\to H^{2}(M,{\rm End}({\bf E}_{U}))$.
Thus $I$ is the ideal generated by the quadratic polynomials on ${\mathcal H}^{1}(M, {\rm End}({\bf E}_{U}))^{\ast}$ associated with ${\mathcal H}^{1}(M, {\rm End}({\bf E}_{U}))\ni x\mapsto [x,x]\in H^{2}(M,{\rm End}({\bf E}_{U}))$.
Define $I_{2}\subset S^{2}{\mathcal H}^{1}(M, {\rm End}({\bf E}_{U}))^{\ast}$ by the image of the dual 
\[H^{2}(M,{\rm End}({\bf E}_{U}))^{\ast}\to S^{2}{\mathcal H}^{1}(M, {\rm End}({\bf E}_{U}))^{\ast}\] of the cup bracket
\[[,]: {\mathcal H}^{1}(M, {\rm End}({\bf E}_{U}))\times {\mathcal H}^{1}(M, {\rm End}({\bf E}_{U}))\to  H^{2}(M,{\rm End}({\bf E}_{U})).
\]
We have
\[R/{\frak m}_{R}^{k}\cong \bigoplus_{i=1}^{k-1} S^{i}{\mathcal H}^{1}(M, {\rm End}({\bf E}_{U}))^{\ast}/ S^{i-2}{\mathcal H}^{1}(M, {\rm End}({\bf E}_{U}))^{\ast}\cdot I_{2}.
\]
It is known that the cup bracket
\[[,]: {\mathcal H}^{1}(M, {\rm End}({\bf E}_{U}))\times {\mathcal H}^{1}(M, {\rm End}({\bf E}_{U}))\to  H^{2}(M,{\rm End}({\bf E}_{U}))
\]
is a homomorphism of $\R$-Hodge structures for the $\R$-Hodge structures  on ${\mathcal H}^{1}(M, {\rm End}({\bf E}_{U}))$ and $H^{2}(M,{\rm End}({\bf E}_{U}))\cong{\mathcal H}^{2}(M, {\rm End}({\bf E}_{U}))$ as in Section \ref{HOK}.
By this, we can say that 
\[R/{\frak m}_{R}\leftarrow R/{\frak m}^{2}_{R}\leftarrow R/{\frak m}^{3}_{R} \cdots
\]
is a inverse system of $\R$-split $\R$-mixed Hodge structures $(W_{\ast}, F^{\ast}_{sp})$ such that
$W_{-i}(R/{\frak m}_{R}^{k})={\frak m}_{R}^{i}/{\frak m}_{R}^{k}$ and the multiplication on $R$ is compatible with these structures.

Take the formal power series $\xi=\sum_{i=1}^{\infty}\xi_{i}\in L^{\ast}\otimes {\frak m}_{R} $ associated with the above special grading as in the last subsection.
For the complexification $L^{\ast}_{\C}=({\mathcal M}^{\ast}_{\C}\otimes {\rm End} (U_{\C}))^{T}$,
we have another special grading
\[L^{\ast}_{\C}=\bigoplus_{k}\left(\bigoplus_{P+Q=k} {\mathcal I}^{-1}\left(({\mathcal N}^{\ast})^{P,Q}\right)\otimes {\rm End} (U)\right)^{T}.
\]
Take the formal power series $\xi^{\prime}=\sum_{i=1}^{\infty}\xi^{\prime}_{i}\in L^{\ast}_{\C}\otimes {\frak m}_{R_{\C}} $ associated with this new special grading.
Then, by Corollary \ref{hullgau}, 
there exists an automorphism $u:R_{\C}\to R_{\C}$ and $B\in L^{0}_{\C}\otimes {\frak m}_{R_{\C}}$ such that
\[\xi^{\prime}=\exp({\rm ad}_{B})\circ u(\xi).
\]
We notice that $u$ induces the identity map on ${\frak m}_{R_{\C}}/{\frak m}_{R_{\C}}^{2}$, 
since such induced map sends  the identity map on ${\mathcal H}^{1}(M, {\rm End}({\bf E}_{U_{\C}}))$ to itself by the constructions of  $\xi$ and  $\xi^{\prime}$.
Since $u$ is a ring homomorphism,  $u$ induces the identity map on ${\frak m}_{R_{\C}}^{k}/{\frak m}_{R_{\C}}^{k+1}$ for any $k$.
We consider the map $\iota: R\to {\rm End}(R)$ associated with the multiplication on $R$.
Let 
\[\Omega=\iota(\xi)\in L^{\ast}\otimes {\rm End}(R)=({\mathcal M}^{\ast}\otimes {\rm End} (U))^{T}\otimes {\rm End}(R)\]
 and 
\[\Omega^{\prime}=\iota(\xi^{\prime})\in L^{\ast}_{\C}\otimes {\rm End}(R_{\C})=({\mathcal M}^{\ast}_{\C}\otimes {\rm End} (U_{\C}))^{T}\otimes {\rm End}(R_{\C}).\]
Then we have
\[\Omega^{\prime}=b^{-1}\Omega b.
\]
where $b=u^{-1}e^{-B}$.

Consider each quotient $q_{k}:R\to R/{\frak m}_{R}^{k}$.
Take $\xi (k)=q_{k}( \xi)$, $\xi^{\prime}(k)=q_{k}( \xi^{\prime})$, $B_{k}=q_{k}(B)$ and the reduction $u_{k}:R_{\C}/{\frak m}_{R_{\C}}^{k}\to R_{\C}/{\frak m}_{R_{\C}}^{k}$ of $u:R_{\C}\to R_{\C}$.
We have $\xi^{\prime}(k)=\exp({\rm ad}_{B_{k}})\circ u_{k}(\xi(k))$.
By the construction, we have $\xi (k)\equiv \sum_{i=1}^{k-1} \xi_{i}$ and  $\xi^{\prime} (k)\equiv \sum_{i=1}^{k-1} \xi^{\prime}_{i}$.
For the map   $\iota_{k}: R/{\frak m}_{R}^{k}\to {\rm End}(R/{\frak m}_{R}^{k})$ associated with the multiplication on $R/{\frak m}_{R}^{k}$, let $\Omega_{k}=\iota_{k}(\xi(k))$ and $\Omega_{k}^{\prime}=\iota_{k}(\xi^{\prime}(k))$. 
We have
\[\Omega^{\prime}_{k}=b^{-1}_{k}\Omega_{k} b_{k}.
\]
where $b_{k}=u^{-1}_{k}e^{-B_{k}}$.
Consider the $\R$-split $\R$-mixed Hodge structure $(W_{\ast}, F^{\ast}_{sp})$ on $U\otimes R/{\frak m}_{R}^{k}$ induced by the $\R$-split $\R$-mixed Hodge structure on $R/{\frak m}_{R}^{k}$ as above and the $\R$-Hodge structure on $U=\bigoplus V_{\gamma_{i}}$.
We regard $\Omega$ as a $T$-equivariant Lie algebra homomorphism ${\frak u}\to {\rm End} (U)\otimes  {\rm End}(R/{\frak m}_{R}^{k})$.
By $b_{k}\in {\rm Aut}_{1}(U\otimes R/{\frak m}_{R}^{k}, W)$,  $(W_{\ast}, F^{\ast})=(W_{\ast}, b_{k}^{-1}F^{\ast}_{sp})$ is an $\R$-mixed Hodge structure.
We have:
\begin{proposition}
$(U\otimes R/{\frak m}_{R}^{k}, W_{\ast}, F^{\ast}, \Omega_{k})$ is a mixed Hodge $(T,{\frak u}) $-representation.

\end{proposition}
\begin{proof}
By the above arguments, it is sufficient to prove the following two claims.
\begin{itemize}
\item $\Omega_{k}:{\frak u}\to {\rm End} (U)\otimes  {\rm End}(R/{\frak m}_{R}^{k})$ is compatible with the weight filtrations $W_{\ast}$.
\item $\Omega_{k}: {\frak u}_{\C}\to {\rm End} (U_{\C})\otimes  {\rm End}(R_{\C}/{\frak m}_{R_{\C}}^{k})$ is compatible with the Hodge filtrations $F^{\ast}$.
\end{itemize}
Consider the sum  $\xi (k)\equiv \sum_{i=1}^{k-1} \xi_{i}$.
By the construction, we have $\xi_{i}\in {\mathcal V}_{i}\otimes  {\rm End} (U)\otimes  {\frak m}_{R}^{i}$.
Thus,  we have
\[\xi (k)({\mathcal V}^{\ast}_{r})\cdot (U\otimes {\frak m}_{R}^{s}/{\frak m}_{R}^{k})\subset  U\otimes {\frak m}_{R}^{r+s}/{\frak m}_{R}^{k}.
\]
This implies the first claim.

Consider the sum $\xi^{\prime} (k)\equiv \sum_{i=1}^{k-1} \xi^{\prime}_{i}$.
By the splitting 
\[R/{\frak m}_{R}^{k}\cong \bigoplus_{i=1}^{k-1} S^{i}{\mathcal H}^{1}(M, {\rm End}({\bf E}_{U}))^{\ast}/ S^{i-2}{\mathcal H}^{1}(M, {\rm End}({\bf E}_{U}))^{\ast}\cdot I_{2},
\]
we took $\xi^{\prime}_{1}\in $ as an identity  map on ${\mathcal H}^{1}(M, {\rm End}({\bf E}_{U}))\otimes \C$.
Thus for the bigrading of the $\R$-split $\R$-mixed Hodge structure $(W_{\ast}, F^{\ast}_{sp})$ on $R/{\frak m}_{R}^{k}$ and the bigrading ${\mathcal M}_{\C}^{\ast}={\mathcal I}^{-1}(\bigoplus (\mathcal N^{\ast})^{P,Q})$, $\xi^{\prime}_{1}\in {\mathcal M}^{\ast}_{\C}\otimes {\rm End}(U)\otimes R/{\frak m}_{R}^{k}$  is of type $(0,0)$.
Since we have $d\xi_{l}=-\frac{1}{2}\sum _{i+j=l, i>0,j>0}[\xi_{i},\xi_{j}]$ for each $l$, we can say that each  $\xi_{l}$ is also of type $(0,0)$ inductively.
Thus the sum $\xi^{\prime} (k)= \sum_{i=1}^{k-1} \xi^{\prime}_{i}$ is of type $(0,0)$.
By this, we can say
\begin{multline*}
\Omega_{k}(F^{r}({\frak u}_{\C}))(F^{s}(U_{\C}\otimes R_{\C}/{\frak m}_{R_{\C}}^{k}))
=b^{-1}\Omega^{\prime}_{k}(F^{s}_{sp}(U_{\C}\otimes R_{\C}/{\frak m}_{R_{\C}}^{k}))\\
=b^{-1}\xi^{\prime} (k)(F^{r}({\frak u}_{\C}))\cdot F^{s}_{sp}(U_{\C}\otimes R_{\C}/{\frak m}_{R_{\C}}^{k})
\subset b^{-1}F^{r+s}_{sp} (U_{\C}\otimes R_{\C}/{\frak m}_{R_{\C}}^{k})\\
=F^{r+s}(U_{\C}\otimes R_{\C}/{\frak m}_{R_{\C}}^{k})
\end{multline*}
and hence
the second claim follows.
Thus the proposition follows.
\end{proof}
Thus we can construct the $\R$-VMHSs associated with mixed Hodge $(T,{\frak u}) $-representations $(U\otimes R/{\frak m}_{R}^{k}, W_{\ast}, F^{\ast}, \Omega_{k})$.
The idea of this construction is inspired by  Eyssidieux-Simpson's work in \cite{ES}.
In \cite[Theorem 3.15]{ES}, Eyssidieux and  Simpson construct $\R$-VMHSs  starting from an $\R$-VHS $\bf E$,  by using Goldman-Millson's theory in \cite{Gold} and \cite{GMi2}.
The construction of this section is very similar to Eyssidieux-Simpson's construction.
They also use the $\R$-split $\R$-mixed Hodge structure on
 \[\bigoplus_{i=1}^{k-1} S^{i}{\mathcal H}^{1}(M, {\rm End}({\bf E}))^{\ast}/ S^{i-2}{\mathcal H}^{1}(M, {\rm End}({\bf E}))^{\ast}\cdot I_{2}
\]
(see \cite[Subsection 2.3]{ES}).
But they do not use $1$-minimal model  and  it is not clear that the two constructions are same.
This matter is left for future work.

\section{Unipotent VMHS without base points}\label{HZZZ}
The main statement of this section is the following.
\begin{theorem}\label{UNIPMM}
On a compact K\"ahler manifold $M$, 
any unipotent $\R$-VMHS over $M$ is isomorphic to the $\R$-VMHS 
$({\bf E}_{\frak V}, {\bf W}_{\frak V \ast},{\bf F}^{\ast}_{\frak V})$ associated with  a mixed Hodge $\frak u$-representation  ${\frak V}$ as in Theorem (Prototype).
\end{theorem}
This result may be a counterpart of the construction of mixed Hodge representations of the fundamental group from unipotent VMHSs as in  \cite{HZ}.
On the construction of Hain and Zucker in \cite{HZ},  the monodromy  representation associated with a base point $x\in M$ play a central role.
Certainly, the monodromy representation of the fiber   at a base point is a very useful method for studying flat bundles. 
But, we consider alternative techniques for studying nilpotent flat bundles without using base points.
In fact, for proving Theorem \ref{UNIPMM}, we need "global trivializations" of unipotent $\R$-VMHSs but we never take fibers of them.

\subsection{Algebraic model for graded nilpotent flat bundles}
Let $M$ be a connected manifold.
Let $\bf E$ be a flat bundle over $M$ and ${\bf W}_{\ast}$ a increasing filtration  of the flat bundle $\bf E$
such that each $Gr_{k}^{\bf W}(\bf E)$ is the trivial local system.
We  fix a global flat frame of each $Gr_{k}^{\bf W}(\bf E)$.
By  splittings of 
\[\xymatrix{
0\ar[r]& {\bf W}_{k-1}(\bf E)\ar[r]&{\bf W}_{k}(\bf E)\ar[r]&Gr_{k}^{\bf W}(\bf E)\ar[r]&0
 },\]
 we obtain a global $\mathcal C^{\infty}$-frame $e^{m_{0}}_{1},\dots, e^{m_{0}}_{i_{m_{0}}},e^{m_{0}+1}_{1},\dots , e^{m_{0}+1}_{i_{m_{0}+1}},\dots, e^{m_{1}}_{1},\dots ,e^{m_{1}}_{i_{m_{1}}}$ of $\bf E$
 such that $e^{m_{0}}_{1},\dots,e^{k}_{i_{k}}$  is a global $\mathcal C^{\infty}$-frame of ${\bf W}_{k}(\bf E)$ and each $e^{k}_{1},\dots ,e^{k}_{i_{k}}$ induces the fixed  global flat frame of $Gr_{k}^{\bf W}(\bf E)$.
 By this global $\mathcal C^{\infty}$-frame, the flat connection on $\bf E$ is represented by a connection form $\omega\in A^{\ast}(M)\otimes W_{-1}({\rm End}(V))$
 where $V=\langle e^{k}_{i}\rangle_{i,k}$ with a filtration $W_{k}(V)=\langle e^{m_{0}}_{1},\dots, e^{k}_{i_{k}}\rangle$.
 (We do not regard $V$ as any fiber of $\bf E$.)
For another such global frame  $f^{m_{0}}_{1},\dots, f^{m_{0}}_{i_{m_{0}}},f^{m_{0}+1}_{1},\dots , f^{m_{0}+1}_{i_{m_{0}+1}},\dots, f^{m_{1}}_{1},\dots ,f^{m_{1}}_{i_{m_{1}}}$ of $\bf E$,
for the representation $\omega^{\prime}$, we have the gauge equivalence $a\in {\rm Id}_{\vert V}+A^{0}(M)\otimes  W_{-1}({\rm End}(V))$ such that 
\[a^{-1}da+a^{-1}\omega a=\omega^{\prime}.
\]
We will control these flat bundles by the following way.

Let $A^{\ast}$ be a cohomologically connected DGA over ${\mathbb K}=\R$ or $\C$ and $(V, W_{\ast})$ a finite-dimensional ${\mathbb K}$-vector space with an increasing filtration $W_{\ast}$.
We define the category $F^{nil}(A^{\ast},V,W_{\ast})$ such that 
objects are $\omega\in A^{1}\otimes W_{-1}({\rm End}(V))$ satisfying the Maurer-Cartan equation 
\[d\omega+\frac{1}{2}[\omega,\omega]=0
\]
and for $\omega_{1},\omega_{2}\in {\rm Ob}(F^{nil}(A^{\ast},V,W_{\ast}))$, morphisms from $\omega_{1}$ to $\omega_{2}$ are $a\in A^{0}\otimes  W_{0}({\rm End}(V))$ satisfying 
\[da+\omega_{1}a-a\omega_{2}=0.
\]
For $\omega_{1},\omega_{2}\in {\rm Ob}(F^{nil}(A^{\ast},V,W_{\ast}))$,
we define the differential $d_{\omega_{1},\omega_{2}}$ on $ A^{\ast}\otimes  W_{0}({\rm End}(V))$
so that for $\eta\in A^{r}\otimes  W_{0}({\rm End}(V))$, 
\[d_{\omega_{1},\omega_{2}}\eta=d\eta+\omega_{1}\eta-(-1)^{r}\eta\omega_{2}.
\]
Then the set of morphisms from $\omega_{1}$ to $\omega_{2}$ is identified with the $0$-th cohomology of the complex 
\[(A^{\ast}\otimes  W_{0}({\rm End}(V)), d_{\omega_{1},\omega_{2}}).
\]
This complex admits the filtration
\[A^{\ast}\otimes  W_{0}({\rm End}(V))\supset A^{\ast}\otimes  W_{-1}({\rm End}(V))\supset A^{\ast}\otimes  W_{-2}({\rm End}(V))\supset\dots
\]
and the $d_{\omega_{1},\omega_{2}}$ induces the differential $d\otimes {\rm id}$ on 
\[Gr_{-k}^{W}(A^{\ast}\otimes  W_{0}({\rm End}(V)))=A^{\ast}\otimes Gr_{-k}^{W}({\rm End}(V)).\]
We also consider the category  $F^{nil}_{gr}(A^{\ast},V,W_{\ast})$ such that 
\[{\rm Ob}(F^{nil}_{gr}(A^{\ast},V,W_{\ast}))={\rm Ob}(F^{nil}(A^{\ast},V,W_{\ast}))\]
and  for $\omega_{1},\omega_{2}\in {\rm Ob}(F^{nil}_{gr}(A^{\ast},V,W_{\ast}))$, 
morphisms from $\omega_{1}$ to $\omega_{2}$ are $a\in {\rm Id}_{\vert V}+A^{0}\otimes  W_{-1}({\rm End}(V))$ satisfying 
\[da+\omega_{1}a-a\omega_{2}=0.
\]
This category is a groupoid i.e. all morphisms are isomorphisms.

Let $\phi:A^{\ast}_{1}\to  A^{\ast}_{2}$ be a morphism between DGAs.
Then this induces the functors $F(\phi): F^{nil}(A^{\ast}_{1},V,W_{\ast})\to F^{nil}(A^{\ast}_{2},V,W_{\ast})$ and $F_{gr}(\phi): F^{nil}_{gr}(A^{\ast}_{1},V,W_{\ast})\to F^{nil}_{gr}(A^{\ast}_{2},V,W_{\ast})$.
\begin{proposition}\label{FFN}
Suppose that $\phi:A^{\ast}_{1}\to  A^{\ast}_{2}$ induces  isomorphisms on $0$th and first cohomologies.
Then the functors $F(\phi): F^{nil}(A^{\ast}_{1},V,W_{\ast})\to F^{nil}(A^{\ast}_{2},V,W_{\ast})$ and $F_{gr}(\phi): F^{nil}_{gr}(A^{\ast}_{1},V,W_{\ast})\to F^{nil}_{gr}(A^{\ast}_{2},V,W_{\ast})$ are fully-faithful.
\end{proposition}
\begin{proof}
On the functor $F(\phi): F^{nil}(A^{\ast}_{1},V,W_{\ast})\to F^{nil}(A^{\ast}_{2},V,W_{\ast})$,
it is sufficient to prove that the morphism \[\phi: (A^{\ast}_{1}\otimes  W_{0}({\rm End}(V)), d_{\omega_{1},\omega_{2}})\to (A^{\ast}_{2}\otimes  W_{0}({\rm End}(V)), d_{\phi(\omega_{1}),\phi(\omega_{2})})\] induces an isomorphism on $0$th cohomology.
This is easily proved by  the five-lemma on  long-exact sequences of extensions 
$A^{\ast}_{i}\otimes  W_{-k}({\rm End}(V))\supset A_{i}^{\ast}\otimes  W_{-k-1}({\rm End}(V))$ for $i=1,2$.

We can easily check that the map $F_{gr}(\phi)(\omega_{1},\omega_{2}): {\rm Hom}(\omega_{1},\omega_{2})\to {\rm Hom}(\phi(\omega_{1}),\phi(\omega_{2}))$ is the restriction of the above isomorphism and this is also an isomorphism.
\end{proof}

 Suppose that $\phi:A^{\ast}_{1}\to  A^{\ast}_{2}$ induces  isomorphisms on the $0$th and first cohomologies and an injection on the second cohomology.
 We take a splitting $V=\bigoplus V_{i}$ of vector space such that $W_{k}(V)=\bigoplus_{i\le k} V_{i}$.
 Corresponding to this splitting, we have the splitting ${\rm End}(V)=\bigoplus U_{i}$ and $U_{i}U_{j}\subset U_{i+j}$.
 For $\omega \in Ob(F^{nil}(A^{\ast}_{2},V,W_{\ast}))$, we write $\omega=\sum \omega _{i}$ with $\omega_{i}\in A^{1}_{2}\otimes U_{i}$. 
 By the  Maurer-Cartan equation, for each $k$, we have
 \[d\omega_{k}=-\sum_{i+j=k}\omega_{i}\wedge \omega_{j}.
 \] 
 We denote $a_{0}={\rm Id}\in {\rm End}(V)$.
\begin{lemma}\label{eqmilll}

For all positive integers $i$,  there exist $\Omega_{i}\in A^{1}_{1}\otimes U_{i}$ and $a_{i} \in A^{0}_{2}\otimes  U_{i}$
such that 
\[d\Omega_{k}=-\sum _{i+j=k}\Omega_{i}\wedge \Omega_{j}
\]
and 
\[da_{k}=\sum_{i+j=k}(-\omega_{i}a_{j}+ a_{i}\phi(\Omega_{j})).
\]
\end{lemma} 
\begin{proof}
We prove the lemma inductively.
Since $\phi:A^{\ast}_{1}\to  A^{\ast}_{2}$  induces an isomorphism on the first cohomology, by $d\omega_{1}=0$,
we have $\Omega_{i}\in A^{1}_{1}\otimes U_{1}$ and $a_{i} \in A^{0}_{2}\otimes U_{1}$ so that
\[da_{1}=-\omega_{1}+\phi(\Omega_{1}).
\]

We suppose that for all $i\le k-1$, we have $\Omega_{i}\in A^{1}_{1}\otimes U_{i}$ and $a_{i} \in A^{0}_{2}\otimes  U_{i}$.
Then by the  equations on $\Omega_{i}$ and $a_{i} $ for $i\le k-1$
  as in the statement, we can easily check
\[d\left(-\sum_{i+j=k} \omega_{i}a_{j} +\sum _{i+j=k, i\ge 1}a_{i}\phi(\Omega_{j})\right)=\sum_{i+j=k}\phi(\Omega_{i})\wedge\phi(\Omega_{j})
\]
and
\[d\left(\sum_{i+j=k}\Omega_{i}\wedge\Omega_{j}\right)=0.
\]
Since $\phi:A^{\ast}_{1}\to  A^{\ast}_{2}$  induces an injection on the second cohomology, 
we have $ \Omega_{k}\in A^{1}_{1}\otimes U_{k}$ so that
\[d \Omega_{k}=-\sum_{i+j=k}\Omega_{i}\wedge\Omega_{j}.
\]
We obtain
\[d\left(-\sum_{i+j=k} \omega_{i}a_{j} +\sum _{i+j=k}a_{i}\phi(\Omega_{j})\right)=d\left(-\sum_{i+j=k} \omega_{i}a_{j} +\sum _{i+j=k, i\ge 1}a_{i}\phi(\Omega_{j})\right)+d \phi(\Omega_{k})=0.
\]
Since $\phi:A^{\ast}_{1}\to  A^{\ast}_{2}$  induces an isomorphism on the first cohomology, we can take $\Omega_{k}$ such that 
\[da_{k}=-\sum_{i+j=k} \omega_{i}a_{j} +\sum _{i+j=k}a_{i}\phi(\Omega_{j})
\]
for some $a_{k} \in A^{0}_{2}\otimes  U_{k}$.
Thus the lemma follows.
\end{proof}
Let $\Omega=\sum\Omega_{i}\in A^{1}_{1}\otimes W_{-1}({\rm End}(V))$ and  $a=\sum a_{i}\in {\rm Id}+A^{0}_{2}\otimes W_{-1}({\rm End}(V))$.
We obtain the equations
\[d\Omega-\Omega\wedge \Omega=0
\]
and
\[da+\omega a-a\phi(\Omega)=0.
\]
Thus, by Proposition \ref{FFN},  we obtain the following statement.
\begin{corollary}\label{eqqfu}
If $\phi:A^{\ast}_{1}\to  A^{\ast}_{2}$ induces  isomorphisms on $0$th and first cohomologies and an injection on the second cohomology, then  the functor $F_{gr}(\phi): F^{nil}_{gr}(A^{\ast}_{1},V,W_{\ast})\to F^{nil}_{gr}(A^{\ast}_{2},V,W_{\ast})$ is an  equivalence.

\end{corollary}

Let $M$ be a connected manifold.
Let $\mathcal M$ be the $1$-minimal model of the DGA $A^{\ast}(M)$ and $\phi:{\mathcal M}^{\ast}\to A^{\ast}(M)$ a map which induces isomorphisms on $0$th and first cohomologies and an injection on the second cohomology.
Then, for the dual Lie algebra $\frak u$ of $\mathcal M$,
each object in $F^{nil}_{gr}({\mathcal M}^{\ast},V,W_{\ast})$
is regarded as a nilpotent representation ${\frak u}\to W_{-1}({\rm End}(V))$ by the Maurer-Cartan equation.
By Lemma \ref{eqmilll}, for a connection form $\omega$ of a filtered nilpotent local system $({\bf E}, {\bf W}_{\ast})$ as above
we obtain $\Omega\in {\rm Ob} F^{nil}_{gr}({\mathcal M}^{\ast},V,W_{\ast})$ and hence a nilpotent representation ${\frak u}\to W_{-1}({\rm End}(V))$.
This construction works on unipotent VMHSs like the monodromy of local systems in the work of Hain and Zucker (\cite{HZ}).

\subsection{Unipotent VMHS}
Let $({\bf E}, {\bf W}_{\ast}, {\bf F}^{\ast})$ be an $\R$-VMHS.
We assume that this is unipotent i.e. the $\R$-VHS on each $Gr_{k}^{\bf W}(\bf E) $ is constant.
Then $({\bf E},{\bf W}_{\ast})$ is a  local system as in the last subsection.
We  fix a global flat frame of $Gr_{k}^{\bf W}(\bf E)$.
We take a global $\mathcal C^{\infty}$-frame $e^{m_{0}}_{1},\dots, e^{m_{0}}_{i_{m_{0}}},e^{m_{0}+1}_{1},\dots , e^{m_{0}+1}_{i_{m_{0}+1}},\dots, e^{m_{1}}_{1},\dots ,e^{m_{1}}_{i_{m_{1}}}$ of $\bf E$ such that $e^{m_{0}}_{1},\dots,e^{k}_{i_{k}}$  is a global $\mathcal C^{\infty}$-frame of ${\bf W}_{k}(\bf E)$ and each $e^{k}_{1},\dots ,e^{k}_{i_{k}}$ induces the fixed  global flat frame of $Gr_{k}^{\bf W}(\bf E)$.
Let $\omega\in A^{\ast}(M)\otimes W_{-1}({\rm End}(V))$ be the connection form for this frame  where $V=\langle e^{k}_{i}\rangle_{i,k}$.
 
We consider the complex  $\mathcal C^{\infty}$-vector bundle  ${\bf E}_{\C}$ and take a splitting ${\bf E}_{\C}=\bigoplus {\bf E}^{k,r}$ such that 
\[{\bf E}^{r,k}\cong Gr_{r}^{\bf F}Gr_{k}^{\bf W}({\bf E}_{\C})\]
as a $\mathcal C^{\infty}$-vector bundle.
By the assumption, each quotient 
\[Gr_{r}^{\bf F}Gr_{k}^{\bf W}({\bf E}_{\C})\]
is a trivial flat bundle.
Thus, we take a global $\mathcal C^{\infty}$-frame $f^{r,k}_{1},\dots, f^{r,k}_{i_{r,k}}$ of ${\bf E}^{k,r}$ which induces a global flat frame of $Gr_{r}^{\bf F}Gr_{k}^{\bf W}({\bf E}_{\C})$.
We take  a global $\mathcal C^{\infty}$-frame $f^{m_{0}}_{1},\dots, f^{m_{0}}_{i_{m_{0}}},f^{m_{0}+1}_{1},\dots , f^{m_{0}+1}_{i_{m_{0}+1}},\dots, f^{m_{1}}_{1},\dots ,f^{m_{1}}_{i_{m_{1}}}$ of ${\bf E}_{\C}$ such that $f^{m_{0}}_{1},\dots,f^{k}_{i_{k}}$  is a global $\mathcal C^{\infty}$-frame of ${\bf W}_{k}(\bf E)$, each $f^{k}_{1},\dots ,f^{k}_{i_{k}}$ induces the fixed  global flat frame of $Gr_{k}^{\bf W}(\bf E)$ and each $f^{k}_{i}$ is a linear combination of $\{f^{r,k}_{i}\}_{r, i}$.
Let $V^{\prime}=\langle f^{k}_{i}\rangle_{i,k}$ and $F^{p}(V^{\prime})=\langle f^{r,k}_{i}\rangle_{r\ge p}$.
Then,  for the connection form $\omega^{\prime}\in A^{\ast}(M)\otimes \C\otimes W_{-1}({\rm End}(V^{\prime}))$  for this frame, the condition of  holomorphicity of $\bf F$ and the Griffiths transversality is equivalent to the condition
\[\omega^{\prime}\in F^{0}\left(A^{\ast}(M)\otimes \C\otimes W_{-1}({\rm End}(V^{\prime}))\right).
\] 
By the identification $V_{\C}=V^{\prime}$, we obtain an $\R$-mixed Hodge structure $(V,W_{\ast}, F^{\ast})$.
We remark that this $\R$-mixed Hodge structure $(V,W_{\ast}, F^{\ast})$ can not be considered as any  fiber of $({\bf E},{\bf W}_{\ast}, {\bf F}^{\ast})$ and can not be determined uniquely.

By these observations and the arguments in the last subsection, for considering unipotent $\R$-VMHSs without a base point, we would like to study the following category.
Let $V$ be  a finite-dimensional $\R$-vector space with an $\R$-mixed Hodge structure $(W_{\ast}, F^{\ast})$.
We define the category $VMHS^{unip}_{\R}(M,V,W_{\ast}, F^{\ast})$ by the following ways:
\begin{itemize}
\item Objects are $(\omega,\omega^{\prime}, a)$ so that:
\begin{itemize}
\item $\omega\in {\rm Ob}(F^{nil}(A^{\ast}(M), V, W_{\ast}))$.
\item  $\omega^{\prime}\in {\rm Ob}(F^{nil}(A^{\ast}(M)\otimes \C, V_{\C}, W_{\ast}))$ satisfying \[\omega^{\prime} \in F^{0}\left(A^{\ast}(M)\otimes \C\otimes W_{-1}({\rm End}(V_{\C}))\right).\]
\item $a\in {\rm Hom}(\omega,\omega^{\prime})$ in the category $F^{nil}_{gr}(A^{\ast}(M)\otimes \C, V_{\C}, W_{\ast})$.
\end{itemize}
\item For $(\omega_{1},\omega_{1}^{\prime}, a_{1}), (\omega_{2},\omega_{2}^{\prime}, a_{2})\in {\rm Ob}(VMHS^{unip}_{\R}(M,V,W_{\ast}, F^{\ast}))$, morphisms from the first one to the second one are $(b,b^{\prime})$ 
so that
\begin{itemize}
\item $b\in {\rm Hom}(\omega_{1},\omega_{2})$.
\item $b^{\prime}\in {\rm Hom}(\omega^{\prime}_{1},\omega^{\prime}_{2})$ satisfying 
\[
b^{\prime}\in F^{0}\left(A^{\ast}(M)\otimes\C \otimes  W_{0}({\rm End}(V))\right) .\]
\item $b^{\prime}\circ a_{1}=a_{2}\circ b$.
\end{itemize}
\end{itemize}
For $(\omega,\omega^{\prime}, a)\in {\rm Ob}(VMHS^{unip}_{\R}(M,V,W_{\ast}, F^{\ast}))$, define $\bf E$ by the  trivial ${\mathcal C}^{\infty}$-vector bundle $M\times V$ with the flat connection $d+ \omega$, ${\bf W}_{\ast}$ by
the filtration of  the vector bundle $M\times V$ induced by the weight filtration $W_{\ast}$ on $V$ and ${\bf F}^{\ast}$ by the filtration of the vector bundle $M\times V$ induced by the Hodge filtration $F^{\ast}$ on $V$.
Then, $({\bf E},{\bf W}_{\ast}, {\bf F}^{\ast})$ is in fact a unipotent $\R$-VMHS.

Suppose that $M$ admits a K\"ahler metric $g$.
We consider the canonical $1$-minimal models $\phi: {\mathcal M}^{\ast}\to  A^{\ast}(M)$ and $\varphi:{\mathcal N}^{\ast}\to  A^{\ast}(M)\otimes \C$ with the trivial $\rho$ and the isomorphism $\mathcal I:{\mathcal M}^{\ast}_{\C}\to {\mathcal N}$ with the homotopy $H$ from $\varphi\circ \mathcal I$ to  $\phi$ as in Section \ref{1-mKal}.
Let $(\omega,\omega^{\prime}, a)\in  {\rm Ob}(VMHS^{unip}_{\R}(M,V,W_{\ast}, F^{\ast}))$.
Then, by Lemma \ref{eqqfu}, we can take $\Omega\in {\rm Ob}(F^{nil}({\mathcal M}^{\ast}, V, W_{\ast})) $ and $b\in {\rm Id}+A^{0}(M)\otimes W_{-1}({\rm End}(V))$ such that $\omega$ is isomorphic to  $\Omega_{\phi}=\phi(\Omega)$ via $b$.
We can also take $\Omega^{\prime}\in {\rm Ob}(F^{nil}({\mathcal N}^{\ast}, V, W_{\ast})) $ and $b^{\prime}\in {\rm Id}+A^{0}(M)\otimes \C\otimes W_{-1}({\rm End}(V_{\C}))$ such that $\omega^{\prime}$ is isomorphic to  $\Omega^{\prime}_{\varphi}=\varphi(\Omega^{\prime})$ via $b^{\prime}$.
\begin{lemma}\label{f000}
We can choose  \[\Omega^{\prime}\in F^{0}\left({\mathcal N}^{\ast}\otimes  W_{-1}({\rm End}(V_{\C}))\right)\] and \[b^{\prime}\in F^{0}\left(A^{\ast}(M)\otimes\C \otimes  W_{0}({\rm End}(V))\right).\]
\end{lemma}
\begin{proof}
Consider the bigrading $V_{\C}=\bigoplus V^{p,q}$ of the $\R$-mixed Hodge structure $(W_{\ast},F^{\ast})$.
Then, by the decoposition  $V_{\C}=\bigoplus_{k} \left(\bigoplus_{p+q=k} V^{p,q}\right)$,
taking $\omega^{\prime}=\sum\omega^{\prime}_{k}$ as in the last subsection,
we can construct
$\Omega^{\prime}=\sum\Omega^{\prime}_{i}$ and  $b^{\prime}=\sum b^{\prime}_{i}$ as in Lemma \ref{eqqfu}.
By \[\omega^{\prime} \in F^{0}\left(A^{1}(M)\otimes \C\otimes W_{-1}({\rm End}(V_{\C}))\right),\]
for each $k$,  we have 
\[\omega^{\prime}_{k}\in F^{0}\left(A^{1}(M)\otimes \C\otimes W_{-k}({\rm End}(V_{\C}))\right).\]
It is sufficient to prove that we can take 
\[\Omega^{\prime}_{k}\in F^{0}\left({\mathcal N}^{1}\otimes  W_{-1}({\rm End}(V_{\C}))\right)\] and \[b^{\prime}_{k}\in F^{0}\left(A^{0}(M)\otimes\C \otimes  W_{-k}({\rm End}(V))\right).\]

We  prove inductively.
By the construction in the proof of Lemma \ref{eqqfu},
$\Omega^{\prime}_{1}$ is the harmonic representative of $\omega^{\prime}_{1}$.
Thus, by 
\[\omega^{\prime}_{1}\in F^{0}\left(A^{\ast}(M)\otimes \C\otimes W_{-1}({\rm End}(V_{\C}))\right),\]
we have $\Omega^{\prime}_{1}\in F^{0}\left(A^{\ast}(M)\otimes \C\otimes W_{-1}({\rm End}(V_{\C}))\right)$.
By the standard argument of  Hodge theory, we can take $b^{\prime}_{1}\in F^{0}\left(A^{0}(M)\otimes\C \otimes  W_{-1}({\rm End}(V))\right)$ such that
\[db_{1}^{\prime}=-\omega^{\prime}_{1}+\varphi(\Omega^{\prime}_{1}) .\]

We assume that for $i\le k-1$ we have taken 
\[\Omega^{\prime}_{i}\in F^{0}\left({\mathcal N}^{\ast}\otimes  W_{-i}({\rm End}(V_{\C}))\right)\] and \[b^{\prime}_{i}\in F^{0}\left(A^{\ast}(M)\otimes\C \otimes  W_{-i}({\rm End}(V))\right).\]
By \[d\Omega^{\prime}_{k}=-\sum _{i+j=k}\Omega^{\prime}_{i}\wedge \Omega^{\prime}_{j},
\]
since the differential on ${\mathcal N}^{\ast}$ is stricktly compatible with the filtration $F^{\ast}$,
we obtain
$\Omega^{\prime}_{k}\in F^{0}\left({\mathcal N}^{\ast}\otimes W_{-k}({\rm End}(V_{\C}))\right)$.
Since  $\varphi:{\mathcal N}^{\ast}\to  A^{\ast}(M)\otimes \C$ is compatible with the filtration $F^{\ast}$,
\[-\sum_{i+j=k} \omega^{\prime}_{i}b^{\prime}_{j} +\sum _{i+j=k}b^{\prime}_{i}\phi(\Omega^{\prime}_{j})\in F^{0}\left(A^{0}(M)\otimes\C \otimes  W_{-k}({\rm End}(V))\right)
\]
and 
we can also take  $b^{\prime}_{k}\in F^{0}\left(A^{0}(M)\otimes\C \otimes  W_{-k}({\rm End}(V))\right)$ such that
\[db^{\prime}_{k}=-\sum_{i+j=k} \omega^{\prime}_{i}b^{\prime}_{j} +\sum _{i+j=k}b^{\prime}_{i}\phi(\Omega^{\prime}_{j}).\]
Hence the lemma follows.
\end{proof}

Define $\Omega_{\varphi}=\varphi({\mathcal I}(\Omega))$.
Then, as in Proposition \ref{gau}, we obtain the isomorphism $c$ from $\Omega_{\phi}$ to $\Omega_{\varphi}$ in the category $F^{nil}_{gr}(A^{\ast}(M)\otimes \C, V_{\C}, W_{\ast})$
which is determined by $H(\Omega)$.
Now we have the isomorphism $b:\omega\to  \Omega_{\phi}$ in the category $F^{nil}_{gr}(A^{\ast}(M), V, W_{\ast})$ and the isomorphisms
$a:\omega \to \omega^{\prime}$,
$b^{\prime}:\omega^{\prime}\to \Omega^{\prime}_{\varphi}$
and $c:\Omega_{\phi} \to \Omega_{\varphi}$ in the category $F^{nil}_{gr}(A^{\ast}(M)\otimes \C, V_{\C}, W_{\ast})$.
Denote $c^{\prime}=c^{-1}b^{-1}ab^{\prime}$.
Then, by Lemma \ref{f000},  $(\omega,\omega^{\prime}, a)$ is isomorphic to $(\Omega_{\phi}, \Omega^{\prime}_{\varphi}, cc^{\prime}) $ in the category $VMHS^{unip}_{\R}(M,V,W_{\ast}, F^{\ast})$ via $(b,b^{\prime})$.

By Proposition \ref{FFN} and ${\mathcal N}^{0}=\C$, we have $c^{\prime}\in {\rm Id}_{\vert V_{\C}}+W_{-1}({\rm End}(V_{\C}))$ and
\[c^{\prime-1}{\mathcal I}(\Omega)c^{\prime}=\Omega^{\prime}.
\]
By Lemma \ref{f000}, ${\frak V}=(V,W_{\ast},c^{\prime}F^{\ast}, \Omega)$ is a mixed Hodge $\frak u$-representation.
We can easiy check that the $\R$-VMHS $({\bf E}_{\frak V}, {\bf W}_{\frak V \ast},{\bf F}^{\ast}_{\frak V})$ associated with ${\frak V}=(V,W_{\ast},c^{\prime}F^{\ast}, \Omega)$ as in Theorem (Prototype) corresponds to $(\Omega_{\phi}, \Omega^{\prime}_{\varphi}, cc^{\prime})\in {\rm Ob}(VMHS^{unip}_{\R}(M,V,W_{\ast}, F^{\ast}))$.
Thus we obtain Theorem \ref{UNIPMM}.

\end{document}